\documentclass[letterpaper,10pt,english]{article}
\usepackage{babel}
\usepackage[latin1]{inputenc}
\usepackage{amsmath,amsthm,amssymb,amsfonts,indentfirst}

\newtheorem{theo}{Theorem}[section]
\newtheorem{cor}[theo]{Corollary}
\newtheorem{lem}[theo]{Lemma}
\newtheorem{prop}[theo]{Proposition}

\theoremstyle{remark}

\newtheorem{dfn}[theo]{\bf Definition}
\newtheorem{obs}[theo]{\bf Remark}

\newcommand\R{\text{I\!R}}

\newcommand\e{\epsilon}
\newcommand\de{\delta}
\newcommand\be{\beta}

\newcommand{\fr}{\partial}
\newcommand{\equ}{\eqref}
\newcommand{\grad}{\nabla}
\newcommand{\ml}{\mathcal}

\newcommand{\DD}{\ml{D}}
\newcommand{\sm}{\setminus}
\newcommand{\la}{\lambda}
\newcommand{\st}{such that }
\newcommand{\dem}{\bf Proof:}

\newcommand\lap{\Delta}
\newcommand\lab{\Delta_g}

\newcommand\ti{\tilde}

\newcommand{\lf}{\left}
\newcommand{\rg}{\right}

\newcommand\ds{\displaystyle}

\newcommand{\bebs}{\begin{equation*}\begin{split}}
\newcommand{\ee}{\end{equation*}}
\newcommand{\esp}{\end{split}}

\DeclareMathAlphabet{\mathpzc}{OT1}{pcz}{m}{it}

\begin{document}
\title{Singular mean field equations on compact Riemann surfaces}
\author{Pierpaolo Esposito\thanks{Dipartimento di Matematica, Universit\`a degli
Studi ``Roma Tre", Largo S. Leonardo Murialdo 1, 00146 Roma,
Italy, e-mail: esposito@mat.uniroma3.it. Author supported by
Prin project ``Critical Point Theory and Perturbative Methods for Nonlinear Differential Equations" and FIRB-IDEAS project ``Analysis and Beyond". } \qquad and \qquad Pablo
Figueroa\thanks{Departamento de Matem\'atica, Pontificia
Universidad Cat\'olica de Chile, Avenida Vicu\~na Mackenna 4860,
Macul, Santiago, Chile. E-mail: pfigueros@mat.puc.cl. Author
supported by grants Becas de Pasant\'ias Doctorales en el
Extranjero BECAS-CHILE and Fondecyt Postdoctorado 3120039,
Chile.}}
\date{\today}
\maketitle

\begin{abstract}
\noindent For a general class of elliptic PDE's  in mean field form on compact Riemann surfaces with exponential nonlinearity, we address the question of the existence of solutions with concentrated nonlinear term, which, in view of the applications, are physically of definite interest. In the model, we also include the possible presence of singular sources in the form of Dirac masses, which makes the problem more difficult to attack.
\end{abstract}

\emph{Keywords}:
\\[0.1cm]

\emph{2000 AMS Subject Classification}:

\section{Introduction}
\noindent Let us consider the problem
\begin{equation}\label{mfeot}
-\lab u=\lambda\lf(\frac{k \,e^{u}}{\int_S k \,e^{u}dv_g} -
\frac{1}{|S|}\rg)
\end{equation}
on a compact, orientable Riemann surface $(S,g)$, where $\lambda>0$, $k$ is a
smooth function and $|S|$ is the area of $S$. Here,
$\Delta_g$ is the Laplace-Beltrami operator and $dv_g$ is the area
element in $(S,g)$.

\medskip \noindent Equation (\ref{mfeot}) and its variants arise in many different
contexts. The Nirenberg problem concerns the existence on
$\mathbb{S}^2$ of metrics --conformal to the standard round metric
$g_0$-- with Gaussian curvature $k$, and corresponds to equation
\eqref{mfeot} with $\lambda=8\pi$. Indeed, a solution $u$ of
\eqref{mfeot} on $(\mathbb{S}^2,g_0)$ with $\lambda=8\pi$ provides
a metric $\frac{4\pi e^u}{\int_{\mathbb{S}^2} k e^u dv_{g_0}}
g_0$, conformal to $g_0$, with Gaussian curvature $k$. For a
general compact Riemann surface, the prescribed Gaussian curvature
problem is referred to as the Kazdan-Warner problem. Since there
are plenty of results in literature, let us just quote the ones
due to Kazdan and Warner \cite{KW}, Chang and Yang \cite{ChY} and
Chang, Gursky, Yang \cite{ChGY}. For bounded domains of $\R^2$, a
variant of \eqref{mfeot} with Dirichlet boundary condition arises
in fluid mechanics as the equation for the stream function of a
turbulent Euler flow with vortices of same orientation. By a
statistical mechanics approach a rigorous derivation of it can be
given as the mean-field limit of the Onsager's vortex theory, as
shown by Kiessling \cite{ChK,K} and Caglioti, Lions, Marchioro,
Pulvirenti \cite{CLMP1, CLMP2}, and it is referred to as the
``mean field equation". In all these contexts, the function $k$ is
typically positive.

\medskip \noindent Notice that (\ref{mfeot}) contains also the singular mean-field
equation
\begin{equation}
-\Delta_g v=\la\lf({h e^{v}\over\int_{S}h e^{v} dv_g}-
\frac{1}{|S|}\rg)+ {4\pi N\over|S|} -4\pi\sum_{j=1}^l
n_j\de_{p_j}\qquad\text{in $S$}
\label{singularEq}\end{equation}
as a special case, where $h>0$, $p_j\in S$, $j=1,\dots,l$, are distinct points, $n_j>0$ and $N=\displaystyle \sum_{j=1}^l n_j$. Indeed, introducing
the Green function $G(x,p)$ with pole at $p \in S$ as the solution
of
\begin{equation} \label{green}
\left\{ \begin{array}{ll} -\Delta_g G(\cdot,p)= \delta_{p}-\frac{1}{|S|} &\text{in $S$}\\
\int_S G(x,p)dv_g=0, &
\end{array} \right.
\end{equation}
the function $u(x)=v(x)+4\pi \displaystyle
\sum_{j=1}^l n_j G(x,p_j)$ does solve (\ref{mfeot}) with $k(x)=h(x)  e^{-4 \pi \sum_{j=1}^l n_j G(x,p_j)}$.
Here, the function $k$ is no longer positive, but is still nonnegative with zero set $\{p_1,\dots,p_l\}.$ On a flat torus $T$, singular mean-field equations with integer multiplicities $\{n_1,\dots,n_l\} \subset \mathbb{N}$ arise in the study of the asymptotics for non-topological (stationary) condensates in the relativistic abelian Chern-Simons-Higgs model as the Chern-Simons parameter tends to zero, as shown by Nolasco and
Tarantello \cite{NT}. In the context of Euler flows, the presence of singular sources model the interaction of the fluid running on the given surface $S$ with sinks of given vorticities and opposite orientation w.r.t. all the vortices present in the fluid.

\medskip \noindent Observe that \eqref{mfeot} admits a variational structure,
in the sense that weak solutions for \eqref{mfeot} are critical
points of the following energy functional
\begin{equation}\label{energy}
J_\lambda(u)={1\over2}\int_S|\grad u|_g^2 dv_g -
\lambda\log\lf(\int_S k e^{u} dv_g\rg),\:\: u\in \bar H,
\end{equation}
where $\bar H=\{u \in H^1(S): \int_S u dv_g=0\}$. For  $\lambda
< 8\pi$, $J_{\lambda}$ is bounded from below and the infimum of
$J_{\lambda}$ is achieved by the well-known Moser-Trudinger
inequality.

\medskip \noindent Let us focus first on the regular case $k>0$. For $k=1$ Struwe and Tarantello \cite{ST} were able to obtain non-trivial solutions of
\eqref{mfeot} for $8\pi < \lambda < 4\pi^2$ on the square flat torus $T$. In the case of compact Riemann surfaces with genus $g\ge 1$ the existence of solutions for \eqref{mfeot} with $8\pi < \lambda < 16 \pi$ was shown by Ding, Jost, Li, Wang \cite{DJLW} still by a variational approach.  The case $S= \mathbb{S}^2$ of zero genus was considered by Lin \cite{Lin} who proved nonvanishing of the Leray-Schauder degree $d_\lambda$ associated to \eqref{mfeot}
for $8\pi <\lambda < 16 \pi$ (and $d_\lambda=0$ for $16\pi < \lambda <24 \pi$).\\
Since the solutions set of \eqref{mfeot} is bounded in $C^{2,\alpha}(S)$, $\alpha \in (0,1)$, as long as $\lambda$ is far from the critical parameter's range $8\pi \mathbb{N}$, the degree $d_\lambda$ is well-defined and constant for all $\lambda \in (8\pi m,8\pi (m+1))$, $m \in \mathbb{N}$. As observed by Y.Y. Li \cite{Li}, its value should just depend on $m$ and the topology of $S$. The program for computing $d_\lambda$, initiated in \cite{Li}, was completely settled by Chen and Lin \cite{CL} showing that
$$d_\lambda=\left(\begin{array}{cc} m-\chi(S) \\ m \end{array} \right),$$
where $\chi(S)=2-2g$ is the Euler characteristic of $S$ (see also the variational approach later developed by Malchiodi \cite{Mal}). For $S\not=  \mathbb{S}^2$, the degree $d_\lambda$ is always non-trivial yielding to a solution of \eqref{mfeot} for all $\lambda \notin 8\pi \mathbb{N}$. While, as already partially proved by Lin \cite{Lin}, for $\mathbb{S}^2$ there holds $d_\lambda=0$ for all $\lambda>16 \pi$ with $\lambda \notin 8\pi \mathbb{N}$, and no existence statements can be deduced. A complete
positive answer to the existence issue for \eqref{mfeot} has been provided by Djadli \cite{Dja} for all $\lambda \notin 8\pi \mathbb{N}$ by means of a variational approach of min-max type, inspired by the result of Djadli and Malchiodi \cite{DjMa} concerning the fourth-order Paneitz operator in conformal geometry. Multiplicity results have been provided by De Marchis \cite{DeM1,DeM2}.

\medskip \noindent Solutions of \eqref{mfeot} are no longer a pre-compact set when $\lambda \to 8\pi \mathbb{N}$: blow-up in $L^\infty-$norm along with the concentration of the measure $\lambda \frac{ke^u}{\int_S ke^u dv_g}$ as a sum of Dirac masses possibly arise for sequences of solutions as $\lambda \to 8\pi \mathbb{N}$. Since $d_\lambda$ can change just when $\lambda$ crosses the values $8\pi m$, $m \in \mathbb{N}$,  it is crucial to have a precise asymptotic knowledge of blow-up solutions $u_\lambda$ and uniquely characterize them as $\lambda \to 8 \pi m$. The most refined asymptotic analysis is given by Chen and Lin \cite{CL0}: in particular, as $\lambda \to 8 \pi m$ $u_\lambda$ has $m$ well-separated maximum points (up to a subsequence) which converge to a critical point in $S^m \setminus \Delta$ of
\begin{equation}\label{fim}
\varphi_m (\xi)=\frac{1}{4\pi}\sum_{j=1}^m \log k(\xi_j )
+\sum_{j=1}^m H(\xi_j,\xi_j) +\sum_{l\not= j} G(\xi_l,\xi_j),
\end{equation}
where $H(x,\xi)$ is the regular part of $G(x,\xi)$ and $\Delta=\{\xi \in S^m:\,\xi_i=\xi_j \hbox{ for }i\not=j\}$ is the diagonal set in $S^m$. Let us notice that a critical point $\xi$ of $\varphi_m$ does satisfy 
$$\nabla \Big[\log k(x)+8\pi H(x,\xi_i) +8\pi \sum_{j\not= i} G(x,\xi_j)\Big] \Big|_{x=\xi_i}=0$$
for all $i=1,\dots,m$. In \cite{CL} blow-up solutions are constructed and their contribution to the degree is computed, so to determine (by local uniqueness of blow-up solutions) the jump in the values of $d_\lambda$ across $\lambda=8\pi m$. Since the degree $d_\lambda$ does not depend on $k$, it is possible to choose a positive function $k$ so that all the c.p.'s of $\varphi_m$ are non-degenerate, and then in \cite{CL} the authors simply address the existence of blow-up sequences of solutions for \eqref{mfeot} which concentrate at non-degenerate c.p.'s of $\varphi_m$ as $\lambda \to 8\pi m$.

\medskip \noindent The aim of the present paper is twofold. On one hand, we are interested in the construction of blow-up solutions with a general potential $k$ for which the corresponding $\varphi_m$ can possibly have degenerate but ``stable" c.p.'s. On the other hand, we are interested to the singular mean-field equation or, equivalently, to (\ref{mfeot}) with a nonnegative potential $k$ which vanishes somewhere. 

\medskip \noindent Let us focus now on the singular case. The first asymptotic analysis has been carried out by Bartolucci and Tarantello \cite{BT}, with an application in the electroweak theory following \cite{DJLW}. The asymptotic analysis has been refined later in \cite{BCLT,CL1}, with the on-going project by Chen and Lin \cite{CL2} of computing the Leray-Schauder degree $d_\lambda$, $\lambda \notin \Lambda$, where
$$\Lambda=8\pi \mathbb{N}+\{8\pi \sum_{j \in J}(1+n_j):\: J \subset \{1,\dots,l\}\}$$
is the correponding critical set of parameters where compactness might fail, see \cite{BT}. For $n_j \geq 1$ the degree $d_\lambda$ has been computed by Chen, Lin and Wang \cite{CLW} for $\lambda \in (8\pi,16\pi)$, revealing the special role played by $\mathbb{S}^2$, the sphere being the only surface for which the degree can vanish (precisely, it vanishes only for $l=1$). The critical regime $\lambda=8\pi$ has been considered in \cite{DJLW1,NT0} for a general surface. However, as we will explain below, the problem on the torus with total multiplicity $N=2$ becomes more degenerate. In this case, existence/non-existence issues have been discussed in \cite{CLW} for a rectangular torus (along with the computation of $d_{8\pi}$) and in \cite{LiWa} for the general case, physically relevant issues in connection with non-topological $2-$condensates in the Chern-Simons-Higgs model \cite{NT}. Existence results have been recently obtained by means of a variational approach of min-max type, inspired by \cite{DJLW,DjMa}, confirming the special role of $\mathbb{S}^2$ (see also the discussion in \cite{BLT,T1}). For $\lambda \notin \Lambda$, the singular problem \eqref{singularEq} is solvable for $S \not=\mathbb{S}^2$ \cite{BDM} (see also \cite{BD} for an application in the electroweak theory). The case of the sphere has been first considered by Malchiodi and Ruiz \cite{MR} for $n_j \leq 1$ and $\lambda \in (8\pi,16\pi)\setminus \Lambda$: the crucial assumption to have existence for \eqref{singularEq} is that $\# \:J\not=1$, where $J=\{j=1,\dots,l: \lambda<8\pi(1+n_j)\}$. The result has been extended by Bartolucci and Malchiodi \cite{BM} to general $n_j$'s and $\lambda$ under the condition $l\geq 2$ and $\lambda<8\pi \min\{1+n_j:j=1,\dots,l \}$, corresponding to the situation $\# J=l$.

\medskip \noindent In some of the above-mentioned papers, the regular/singular mean field equation has been also considered on a bounded domain $\Omega \subset \mathbb{R}^2$ with homogeneous Dirichlet b.c. Since $\int_S ke^{u_\lambda}dv_g \to +\infty$ along any non-compact sequence of solutions $u_\lambda$ for \eqref{mfeot}, through the setting $\rho= \frac{\lambda}{\int_S k e^{u}dv_g}$ problem (\ref{mfeot}) is naturally related (but not equivalent) to
$-\lab u=\rho \left(k e^u-\frac{1}{|S|}\int_S ke^u dv_g \right)$ with $\rho \to 0^+$, which has been recently studied by the second author in \cite{F}. Blow-up solutions for the corresponding Dirichlet problem
$$\left\{ \begin{array}{ll}-\Delta u=\rho k e^{u}& \hbox{in }\Omega\\
u=0 &\hbox{on }\partial \Omega \end{array} \right. $$
on a bounded domain $\Omega \subset \mathbb{R}^2$ have been constructed at c.p.'s of $\varphi_m$ which are non-degenerate \cite{bp} or, more generally, ``stable" \cite{DeKM,EGP}. A ``stable" critical value for $\varphi_m$ has been constructed by del Pino, Kowalczyk and Musso \cite{DeKM} for the regular problem on a non-simply connected domain and for the singular problem with $l=1$. The latter case has been extended to the flat torus \cite{F}, and a similar result is still in order for multiple singular sources as shown by D'Aprile \cite{Dap}, under suitable relations between $m$ and the $n_j$'s.

\medskip \noindent Setting
\begin{equation}\label{ro}
\rho_j(x)=k(x)e^{ 8\pi H(x,\xi_j)+ 8\pi \sum_{l\ne j}
G(x,\xi_l)},
\end{equation}
for $\xi \in S^m \setminus \Delta$ we introduce the notation
\begin{equation}\label{v}
A(\xi)=4\pi\sum_{j=1}^m\left[\Delta_g
\rho_j(\xi_j)-2K(\xi_j)\rho_j(\xi_j)\right],
\end{equation}
where $K$ is the Gaussian curvature of $(S,g)$. Letting $\tilde S=\{k>0\}$, our first main result is
\begin{theo}\label{main1}
Let $\mathcal{D} \subset \subset \tilde S^m \setminus \Delta$ be a stable critical set of $\varphi_m$. Assume that $A(\xi)>0$ ($<0$ resp.) for all $\xi \in \mathcal{D}$. Then, for all $\lambda$ in a small right (left resp.) neighborhood of $8 \pi m$ there is a solution $u_\lambda$ of \eqref{mfeot} so that (along sub-sequences)
\begin{equation}\label{cism}
\frac{\lambda k\,e^{u_\lambda}}{\int_S
k\,e^{u_\lambda}dv_g}\rightharpoonup 8\pi
\sum_{j=1}^{m}\delta_{q_j}
\end{equation}
as $\lambda \to 8\pi m$ in the sense of measures in $S$, for some $q=(q_1,\dots,q_m) \in \mathcal{D}$.
\end{theo}
\noindent Along with (\ref{cism}) notice that there always hold that $u_\lambda-\log \int_S ke^{u_\lambda}dv_g \to -\infty$ in $C_{\hbox{loc}}(S\setminus \{q_1,\dots,q_m\})$ and
$$\sup_{U_j} \left(u_\lambda-\log \int_S ke^{u_\lambda}dv_g \right)\to +\infty$$
as $\lambda \to 8\pi m$, for any neighborhood $U_j$ of $q_j$ in $S$, $j=1,\dots,m$. The notion of stability we are using here is the one introduced in \cite{Li0}:
\begin{dfn}
\label{stable} A critical set $\mathcal{D} \subset \subset \tilde S^m \setminus \Delta$ of $\varphi_m$ is stable if for any closed neighborhood $U$ of $\mathcal{D}$ in  $\tilde S^m \setminus \Delta$ there exists $\delta>0$ such that, if
$\|G-\varphi_m\|_{C^1(U)}\leq \delta$, then $G$ has at least one critical point in $U$. In particular, the minimal/maximal set of $\varphi_m$ is stable (if $\varphi_m$ is not constant) as well as any isolated c.p. of $\varphi_m$ with non-trivial local degree.
\end{dfn}

\medskip \noindent  Since $A(\xi)$ can be re-written as
\begin{eqnarray} \label{ppp}
A(\xi)&=& 4\pi \sum_{j=1}^m \rho_j(\xi_j) [\Delta_g \log \rho_j(\xi_j)+|\nabla \log \rho_j(\xi_j)|^2_g -2K(\xi_j)]\nonumber \\
&=& 4\pi \sum_{j=1}^m \rho_j(\xi_j) [\Delta_g \log k(\xi_j)+\frac{8\pi m}{|S|}+|\nabla \log \rho_j(\xi_j)|^2_g -2K(\xi_j)],
\end{eqnarray}
for a c.p. of $\varphi_m$ we have that
$$A(\xi)= 4\pi \sum_{j=1}^m \rho_j(\xi_j) [\Delta_g \log k(\xi_j)+\frac{8\pi m}{|S|} -2K(\xi_j)]$$
in view of $\nabla \rho_j(\xi_j)=0$ for all $j=1,\dots,m$. Since for $k>0$ the function $\varphi_m$ always attains its minimum value in $S^m \setminus \Delta$ and the minimal set is clearly stable, as a first by-product we have (see also \cite{CL}):
\begin{cor} \label{appl1}
Assume $k>0$. Let $m \in \mathbb{N}$ be so that either $1\leq m< \inf_S \frac{|S|}{8\pi}[2K-\Delta_g \log k]$ or $m> \sup_S \frac{|S|}{8\pi}[2K-\Delta_g \log k]$. Then there exist solutions $u_\lambda$ of (\ref{mfeot}) which concentrate at $m$ points $q_1,\dots, q_m$ in the sense \eqref{cism} as $\lambda \to 8\pi m$, where $q=(q_1,\dots,q_m)$ is a minimum point of $\varphi_m$ in $S^m \setminus \Delta$.
\end{cor}
\noindent When the surface $(S,g)$ has constant Gaussian curvature, by the Gauss-Bonnet formula we have that $K=\frac{2 \pi \chi(S)}{|S|}$. For $k=1$, Corollary \ref{appl1} then provides the existence of blow-up solutions $u_\lambda$ concentrating at $m$ points as $\lambda \to 8\pi m$ for all $m\geq 2$, where $\lambda$ belongs to a small right neighborhood of $8\pi m$. The case $m=1$ is problematic since $\varphi_1$ is a constant function.

\medskip \noindent Concerning the singular problem \eqref{singularEq}, in general the function $\varphi_m$ has neither maximum nor minimum points, and it is then natural to search for saddle critical points. The min-max scheme introduced in \cite{Dap} works in the Euclidean context as well as in the case of a surface \cite{DapE}. In particular, on $\mathbb{S}^2$ the function $\varphi_m$ has a ``stable" critical value of min-max type as soon as $l\geq 2$ and
\begin{equation} \label{conditionsphere}
8\pi m \notin 8\pi \mathbb{N}+8\pi(1+n_j) \:\:\forall\: j=1,\dots,l,\quad \#J\geq 2,
\end{equation}
where $J=\{j=1,\dots,l: 8\pi m<8\pi(1+n_j)\}$. In the construction, each singular source $p_i$ has to be coupled with some $p_j \not= p_i$ in order to deform $\mathbb{S}^2\setminus\{p_i,p_j\}$ onto a circle running around $p_i$, and the condition $l\geq 2$ is crucial. Notice that the min-max scheme provides a critical point $q$ of $\varphi_m$ so that $\{q \}$ is a stable critical set according to Definition \ref{stable}. Morover, since $\#J\geq 2$ yields to $2m<2+N$, for $k(x)=e^{-4 \pi \sum_{j=1}^l n_j G(x,p_j)}$ we have that
$$A(q)=\frac{16 \pi^2 }{|\mathbb{S}^2|} \sum_{j=1}^m \rho_j(q_j) [-N+2m -2]<0$$
in view of $K=\frac{4\pi}{|\mathbb{S}^2|}$. As a second by-product of Theorem \ref{main1}, we have:
\begin{cor} \label{appl2}
Let $h=1$ and $l\geq 2$. Assume that $S$ is topologically a sphere and that $m$ satisfies \eqref{conditionsphere}. Then, for all $\lambda$ in a small left neighborhood of $8 \pi m$ there is a solution $u_\lambda$ of \eqref{singularEq} which concentrates at $m$ points $q_1,\dots, q_m$ in the sense \eqref{cism} as $\lambda \to 8\pi m$.
\end{cor}
\noindent Theorem \ref{main1} and Corollary \ref{appl1} are the perturbative counter-parts of global existence results already available in literature, obtained via degree theory or a variational approach.  However, the behavior of such solutions as $\lambda \to 8\pi m$ is not known whereas the ones we construct exhibit blow-up phenomena, a property that has a definite interest in its own. More important, Corollary \ref{appl2} gives completely new results for the case of $\mathbb{S}^2$, by showing that in a perturbative regime the condition $\#J\geq 2$ in \cite{MR} is sufficient for the existence in the general case, beyond the results in \cite{BM}. Moreover, in \cite{DapE} the cases $\#J=0,1$ are also treated.

\medskip \noindent There are cases for which $A(\xi)$ can vanish. By invariance under rotations, it is easily seen that on $\mathbb{S}^2$ the function $H(\xi,\xi)$ is constant, and then the c.p.'s of $\varphi_1$ and $k$ do coincide. Since in particular $\nabla H(x,\xi)\Big|_{x=\xi}=0$, by (\ref{ppp}) for $S=\mathbb{S}^2$ and $m=1$ the coefficient $A(\xi)$ writes as
\begin{equation} \label{casopartA}
A(\xi)=4\pi k(\xi)e^{8\pi H(\xi,\xi)}[\Delta_g \log k(\xi)+|\nabla \log k(\xi)|^2_g]=4\pi e^{8\pi H(\xi,\xi)} \Delta_g  k(\xi),
\end{equation}
and might vanish at some c.p. of $k$. Another typical example is the singular mean-field equation \eqref{singularEq} on the flat torus $T$ with $h=1$ and even total multiplicity $N$: since $k=e^{-4\pi \sum_{j=1}^l n_j G(x,p_j)}$ and $K\equiv 0$, by (\ref{ppp}) the coefficient
$A(\xi)$ writes for $m=\frac{N}{2}$ as
$$A(\xi)=4\pi \sum_{j=1}^m \rho_j(\xi_j) |\nabla \log \rho_j(\xi_j)|^2_g=(4\pi)^3 \sum_{j=1}^m \rho_j(\xi_j) |\nabla_{\xi_j} \varphi_m(\xi)|^2_g \geq 0,$$
and vanishes exactly at the c.p.'s of $\varphi_m$. In all these
situations, a more refined analysis is necessary.

\medskip \noindent Introduce the following quantity
\begin{eqnarray} \label{B}
B(\xi)&=& -2\pi \sum_{j=1}^m [\Delta_g  \rho_j(\xi_j) -2 K(\xi_j) \rho_j(\xi_j)] \log \rho_j(\xi_j) -\frac{A(\xi)}{2}\\
&&+\lim_{r \to 0}\left[ 8 \int_{S \setminus \cup_{j=1}^m B_r(\xi_j)}   ke^{8\pi \sum_{j=1}^m G(x,\xi_j)} dv_g-\frac{8\pi}{r^2} \sum_{j=1}^m \rho_j(\xi_j)-A(\xi) \log \frac{1}{r} \right],
\nonumber \end{eqnarray}
where $B_r(\xi)$ denotes the pre-image of $B_r(0)$ through the isothermal coordinate system at $\xi$. The quantity $B(\xi)$ has been first used and derived by Chang, Chen and Lin \cite{ChChL} in the study of  the mean field equation on bounded domains (see also \cite{CLW,LiYa} for the case of the torus). We have the following general result, of which Theorem \ref{main1} is just a special case:
\begin{theo} \label{main2}
Let $\mathcal{D} \subset \subset \tilde S^m \setminus \Delta$ be a stable critical set of $\varphi_m$. Assume that
\begin{equation} \label{cond}
\hbox{either }A(\xi) >0 \:(<0 \hbox{ resp.)} \qquad \hbox{or} \qquad A(\xi)=0, \:B(\xi)>0 \: (<0 \hbox{ resp.)}
\end{equation}
do hold in a closed neighborhood $U$ of $\mathcal{D}$ in $\tilde S^m \setminus \Delta$. Then, for all $\lambda$ in a small right (left resp.) neighborhood of $8 \pi m$ there is a solution $u_\lambda$ of \eqref{mfeot} which concentrates (along sub-sequences) at $m$ points $q_1,\dots, q_m$ in the sense (\ref{cism}) as $\lambda \to 8\pi m$, for some $q \in \mathcal{D}$. 
\end{theo}
\noindent To deal a with stable critical set $\DD$ in the sense above, we need to require condition (\ref{cond}) on a neighborhood of $\DD$. In case we strengthen the stability assumption, we can relax the assumption (\ref{cond}) to hold just on $\DD$. As an instructive example, in Remark \ref{minmax}-(i) we present the case of a non-degenerate local minimum/maximum point.

\medskip \noindent We can now discuss the two previous examples for which
the coefficient $A(\xi)$ vanishes. For $S=\mathbb{S}^2$ and $m=1$,
there holds $\varphi_1=\frac{1}{4\pi} \log k+\hbox{const.}$ since $H(\xi,\xi)=\hbox{const.}$. In view of \eqref{casopartA}, assume that $\Delta_g k \geq
0$ in a small neighborhood $U$ of the minimal set
$\mathcal{D}=\{\xi \in S: \varphi_1(\xi)=\min_S \varphi_1\}$ so to have $A(\xi)\geq 0$ in $U$. We  just need to show that
$B(\xi)>0$ in $U$ so to use Theorem \ref{main2} with $\mathcal{D}$, which is clearly a stable critical set of
$\varphi_1$ as soon as $k$ is not a constant function. Up to take $U$ smaller, it is
clearly enough to show that $B(\xi)>0$ for all $\xi \in
\mathcal{D}$ with $A(\xi)=0$. Up to a rotation, we can assume that $\xi$ is the
south pole $P$ of $\mathbb{S}^2$. The stereographic projection
$\pi: (x,y,z) \to (\frac{2x}{1-z},\frac{2y}{1-z})$ through the
north pole is an isometry between $(\mathbb{S}^2 \setminus
\{\hbox{north pole}\},g_0)$ and
$(\mathbb{R}^2,\frac{16}{(4+u^2+v^2)^2} \delta_{\hbox{eucl}})$.
Since it is easily seen that
$\hbox{dist}\,(\pi^{-1}(u,v),P)=|(u,v)|$ and
$$G(\pi^{-1}(u,v),P)=-\frac{1}{2\pi} \log |(u,v)|+\frac{1}{4\pi} \log (4+u^2+v^2)+c_0,$$ in the coordinate system $\pi$ in terms of $\tilde k(u,v)=k(\pi^{-1}(u,v))$ we can write that
$$B(P)= 128 e^{c_0}  \lim_{r \to 0}  \int_{\mathbb{R}^2 \setminus B_r(0)} \frac{\tilde k(u,v)-\tilde k(0,0)}{(u^2+v^2)^2}dudv>0$$
in view of $k\geq k(P)$, $k\not= k(P)$. Similarly, we can treat  the case in which $\Delta_g k \leq 0$ does hold in a small neighborhood $U$ of the maximal set $\mathcal{D}$.\\
In the case of the flat torus $T$ with $N$ even,
$m=\frac{N}{2}$ and $k=e^{u_0}$, $u_0=-4\pi \sum_{j=1}^l n_j
G(x,p_j)$, at a c.p. $\xi$ of $\varphi_m$ the coefficient
$$B(\xi)= \lim_{r \to 0}\left[ 8 \int_{T \setminus \cup_{j=1}^m B_r(\xi_j)}   e^{u_0+8\pi \sum_{j=1}^m G(x,\xi_j)} dx-\frac{8\pi}{r^2} \sum_{j=1}^m \rho_j(\xi_j)  \right]$$
can be re-written in the following way:
\begin{itemize}
\item if $N=2$, $m=1$
$$B(\xi)= 8 e^{u_0(\xi)+8\pi H(\xi,\xi)} \left[ \int_T   \frac{e^{u_0(x)-u_0(\xi)+8\pi H(x,\xi)-8\pi H(\xi,\xi) }-1}{|x-\xi|^4} dx - \int_{\mathbb{R}^2\setminus  T} \frac{dx}{|x-\xi|^4}\right],$$
where the integral on $T$ is conditionally convergent in view of $\nabla(u_0(x)+8\pi H(x,\xi))\Big|_{x=\xi}=0$ and $\Delta (u_0(x)+8\pi H(x,\xi))\equiv 0$;
\item if $N\geq 4$ even, $m=\frac{N}{2}$
$$  B(\xi)=8  \sum_{j=1}^m \left[\int_{T_j } \frac{\rho_j(x)-\rho_j(\xi_j)}{|x-\xi_j|^4}dx-\rho_j(\xi_j) \int_{\mathbb{R}^2 \setminus T_j}\frac{dx}{|x-\xi_j|^4}\right],$$
where $T$ has been splitted into disjoint sets $T_1,\dots,T_m$ so that $B_r(\xi_j) \subset T_j$ for $r$ small and all $j$.
\end{itemize}
When $T$ is a rectangle, $l=1$ and $n_1=2$, the constant $B(\xi)$ has been used by Chen, Lin and Wang \cite{CLW} in the computation of the degree $d_{8\pi}$. The function $\varphi_1=\frac{u_0}{4\pi}+\hbox{const.}$ has exactly three non-degenerate critical points $\xi_1,\, \xi_2$ (saddle points) and $\xi_3$ (maximum point) with $B(\xi_1),\,B(\xi_2)>0$ and $B(\xi_3)<0$. By Theorem \ref{main2} and Remark \ref{minmax}-(i) we deduce the existence of
\begin{itemize}
\item two distinct families of solutions, for $\lambda$ in a small right neighborhood of $8\pi$, concentrating at $\xi_1$ and $\xi_2$ as $\lambda \to 8\pi $;
\item one family of solutions, for $\lambda$ in a small left neighborhood of $8\pi$, concentrating at $\xi_3$ as $\lambda \to 8\pi $.
\end{itemize}
Moreover, $B(\xi)$ has been recently used in the construction of non-topological condensates for the relativistic abelian Chern-Simons-Higgs model as the Chern-Simons parameter tends to zero, see \cite{LiYa}. Unfortunately, when $N\geq 4$ there are no examples where the sign of $B(\xi)$ can be determined.

\medskip \noindent To explain more clearly such a connection, recall that in the relativistic abelian Chern-Simon-Higgs model the $N$ vortex-condensates are gauge-periodic stationary matter configurations with finite-energy that, in the self-dual regime, express in terms of solutions for
\begin{equation} \label{CSoriginal}  -\Delta w=\frac{1}{\epsilon^2}e^w(1-e^w)-4\pi\sum_{j=1}^l
n_j \delta_{p_j}
\end{equation}
in a flat torus $T$. We refer to \cite{D} for a complete account on the model and to \cite{Tbook} for the analytical results concerning it. The quantity $2\epsilon >0$ is the Chern-Simons
parameter, $p_j\in S$, $j=1,\dots,l$, are distinct points and $n_j\in \mathbb{N}$. Physically, $\epsilon$ is
very small and two classes of solutions are relevant: either $e^w
\to 1$ as $\epsilon \to 0^+$ (``topological'' type) or $e^w \to 0$
as $\epsilon \to 0^+$ (``non-topological'' type). Topological solutions were first found by Caffarelli and Yang \cite{CY}. However, non-topological
condensates represent the main feature of the Chern-Simons-Higgs
model which were absent in the classical (Maxwell-Higgs) vortex
theory, whose existence was established by Tarantello \cite{T}. Through the change $w \to w-u_0$, $u_0=-4\pi \sum_{j=1}^l n_j G(x,p_j)$, the self-dual equation (\ref{CSoriginal}) reads
equivalently as
\begin{equation} \label{CS}  -\Delta w=\frac{1}{\epsilon^2}ke^w(1-ke^w)-\frac{4\pi N}{|T|}
\end{equation}
with $k=e^{u_0}$. Setting $c=\frac{1}{|T|}
\int_T w dx$ and $u=w-c \in \bar H$, an integration of
(\ref{CS}) provides a relation between $c$ and $u$ (see \cite{T}):
$$  e^c \int_T ke^u dx-e^{2c} \int_T k^2 e^{2u}dx=4\pi N \epsilon^2.$$
Hence, necessarily
$$u\in\ml{A}_\e=\bigg\{u\in \bar H \;\bigg|\; \bigg(\int_T k e^u\bigg)^2-16\pi N \epsilon^2 \int_T k^2
e^{2u}\ge 0\bigg\}$$ and then $c=c_\pm(u)$ with
$$e^{c_\pm(u)}=\frac{8\pi N \epsilon^2}{\int_T k e^u \mp \sqrt{(\int_T k e^u)^2-16\pi N \epsilon^2 \int_T k^2
e^{2u}}}.$$ For solutions of ``non-topological'' type it is
natural to choose $c_-(u)$, and then equation (\ref{CS}) reads in
terms of $u \in \ml{A}_\e$ as
\begin{eqnarray} \label{CSMF}  -\Delta u&=&4\pi
N\left(\frac{k e^u}{\int_T k e^u}-\frac{1}{|T|}\right) \nonumber \\
&&+ \frac{64 \pi^2N^2 \epsilon^2 \int_T k^2 e^{2u}}{\Big(\int_T k
e^u+\sqrt{(\int_T k e^u)^2-16\pi N\epsilon^2\int_T k^2
e^{2u}}\Big)^2}\left(\frac{ke^u}{\int_T k e^u}-\frac{k^2
e^{2u}}{\int_T k^2 e^{2u}}\right). \end{eqnarray} When $N$ is even
and $m=\frac{N}{2}$, equation (\ref{CSMF}) is a perturbation of
(\ref{mfeot})$_{\lambda=8 \pi m}$ as $\epsilon \to 0^+$. The
parallel becomes clear if we re-consider (\ref{mfeot}) itself as a
perturbation of (\ref{mfeot})$_{\lambda=8\pi m}$ as $\lambda \to
8\pi m$. As far as (\ref{mfeot}) is concerned,  the sign of the
perturbation can be chosen since it depends on $\lambda-8\pi m$.
For  (\ref{CSMF}) the sign of the perturbation is given and is like the case $\lambda<8\pi m$ in which we need to require (\ref{cond}) with the negative sign $<0$. Even if we always have the wrong sign $A(\xi)\geq 0$, the coefficient $A(\xi)$ behaves like $|\nabla \varphi_m(\xi)|_g^2:=\displaystyle \sum_{j=1}^m |\nabla_{\xi_j} \varphi_m(\xi)|_g^2$ and, near a critical set $\DD$ of $\varphi_m$, is very small. The condition $B(\xi)<0$ on $\DD$ will then be enough, as stated in the following:
\begin{theo} \label{main3}
Assume $N$ even. Let $\mathcal{D} \subset \subset (T \setminus \{p_1,\dots,p_l\})^m \setminus \Delta$ be a stable critical set of
$$\varphi_m (\xi)=\frac{1}{4\pi}\sum_{j=1}^m u_0(\xi_j )+\sum_{l\not= j} G(\xi_l,\xi_j).$$
Assume that $B(\xi)<0$ does hold in $\mathcal{D}$. Then, for all $\epsilon$ small there is a solution $w_\epsilon$ of \eqref{CSoriginal} which concentrates at $m$ points $q_1,\dots, q_m$, with $q=(q_1,\dots,q_m) \in \mathcal{D}$, as $\epsilon \to 0$ in the sense of measures:
$$\frac{1}{\epsilon^2}e^{w_\epsilon}(1-e^{w_\epsilon}) \rightharpoonup 8\pi
\sum_{j=1}^{m}\delta_{q_j}.$$
Correspondingly, there exist non-topological $N$ vortex-condensates of gauge potential $A_\epsilon$ and Higgs field $\phi_\epsilon$ for which the magnetic field $(F_{12})_\epsilon$ is very concentrated at the $m$ points $q_1,\dots,q_m$ (external to the so-called vortex-set $\{p_1,\dots, p_l \}$) as $\epsilon \to 0$.
\end{theo}
\noindent Theorem \ref{main3} slightly improves the result in \cite{LiYa} (see \cite{DEFM,DEM,LiYa0} for concentration at the vortices) where they just deal with isolated c.p.'s of $\varphi_m$ with non-trivial local degree. Our ``stability" assumption is more general as already explained in Definition \ref{stable}. Even if $\varphi_1$ has always the maximal set as a ``stable" critical set, a general existence result for $1-$point concentration does not follow since we don't know whether the coefficient $B(\xi)<0$ or not (apart from the case $l=1$, $n_1=2$, $T$ a rectangle).

\bigskip
\section{Approximation of the solution}
\noindent To construct approximating solutions of \eqref{mfeot},
the main idea is to use as ``basic cells'' the functions
\begin{equation*}
u_{\delta,\xi}(x)=u_0 \Big(\frac{|x-\xi|}{\delta}\Big)-2\log
\delta \qquad \de>0,\: \xi\in\R^2,\end{equation*} where
$$ u_0(r)=\log\frac{8}{(1+r^2)^2}.$$
They are all the solutions of
\begin{equation*}
\left\{ \begin{array}{ll}\Delta u+e^{u}=0 &\text{in $\R^2$}\\
\int_{\R^2} e^u <\infty, & \end{array} \right.
\end{equation*}
and do satisfy the following concentration property:
$$e^{u_{\delta,\xi}}\rightharpoonup 8\pi\delta_\xi
\quad\text{in measure sense}$$ as $\delta \to 0$. We will use now
isothermal coordinates to pull-back $u_{\delta,\xi}$ in $S$.

\medskip \noindent Let us recall that every Riemann surface $(S,g)$ is locally
conformally flat, and the local coordinates in which $g$ is
conformal to the Euclidean metric are referred to as isothermal
coordinates (see for example the simple existence proof provided
by Chern \cite{Chern}). For every $\xi \in S$ it amounts to find a
local chart $y_\xi$, with $y_\xi(\xi)=0$, from a neighborhood of
$\xi$ onto $B_{2r_0}(0)$ (the choice of $r_0$ is independent of
$\xi$) in which $g=e^{\hat \varphi_\xi(y_\xi(x))}dx$, where $\hat
\varphi_\xi \in C^\infty(B_{2r_0}(0),\mathbb{R})$. In particular,
$\hat \varphi_\xi$ relates with the Gaussian curvature $K$ of
$(S,g)$ through the relation:
\begin{equation} \label{equationvarphi}
\Delta \hat \varphi_\xi(y) =-2K(y_\xi^{-1}(y)) e^{\hat
\varphi_\xi(y)} \qquad \hbox{ for }y \in B_{2r_0}(0).
\end{equation}
We can also assume that $y_\xi$, $\hat \varphi_\xi$ depends
smoothly in $\xi$ and that $\hat \varphi_\xi(0)=0$, $\nabla \hat
\varphi_\xi(0)=0$.

\medskip \noindent  We now pull-back $u_{\delta,0}$ in $\xi \in S$, for $\delta>0$, by simply
setting
$$U_{\delta,\xi}(x)=u_{\delta,0}(y_\xi(x))=\log \frac{8\delta^2}{(\delta^2+|y_\xi(x)|^2)^2}$$ for $x
\in y_\xi^{-1}(B_{2r_0}(0))$. Letting $\chi\in
C_0^\infty(B_{2r_0}(0))$ be a radial cut-off function so that
$0\le\chi\le 1$, $\chi\equiv 1$ in $B_{r_0}(0)$, we introduce the
function $PU_{\de,\xi}$ as the unique solution of
\begin{equation}\label{ePu}
\left\{ \begin{array}{ll} -\Delta_g PU_{\de,\xi} (x)=\chi_\xi(x)
e^{-\varphi_\xi(x)} e^{U_{\de,\xi}(x)}-\frac{1}{|S|}\int_S
\chi_\xi e^{-\varphi_\xi} e^{U_{\de,\xi}} dv_g &\text{in }S\\
\int_S PU_{\de,\xi} dv_g=0,
\end{array}\right.
\end{equation}
where $\chi_\xi(x)=\chi(|y_\xi(x)|)$ and $\varphi_\xi(x)=\hat
\varphi_\xi(y_\xi(x))$. Notice that the R.H.S. in (\ref{ePu}) has
zero average and smoothly depends in $x$, and then (\ref{ePu}) is
uniquely solvable by a smooth solution $PU_{\de,\xi}$.

\medskip \noindent Let us recall the transformation law for $\Delta_g$ under
conformal changes: if $\tilde g=e^{\varphi} g$, then
\begin{equation} \label{laplacian} \Delta_{\tilde g}=e^{-\varphi} \Delta_g.\end{equation}
Decompose now the Green function $G(x,\xi)$, $\xi \in S$, as
$$G(x,\xi)=-\frac{1}{2\pi} \chi_\xi(x) \log |y_\xi(x)|+H(x,\xi),$$
and by (\ref{green}) then deduce that
\begin{equation*}
\left\{ \begin{array}{ll} -\Delta_g H= - \frac{1}{2\pi} \lab
\chi_\xi  \,\log |y_\xi(x)| -\frac{1}{\pi}\langle \grad
\chi_\xi,\grad \log
|y_\xi(x)| \rangle_g-\frac{1}{|S|} &\text{in $S$}\\
\int_S H(\cdot,\xi)\, dv_g=\frac{1}{2\pi} \int_S \chi_\xi \log
|y_\xi(\cdot)| dv_g.&
\end{array} \right.
\end{equation*}
We have used that
$$\Delta_g \log |y_\xi(x)|= e^{-\hat \varphi_\xi(y)}
\Delta \log|y| \Big|_{y=y_\xi(x)}=2\pi \delta_\xi$$ in view of
(\ref{laplacian}).

\medskip \noindent For $r\leq 2r_0$ define
$B_r(\xi)=y_\xi^{-1}(B_r(0))$, $A_{r}(\xi)=B_{r}(\xi) \sm
B_{r/2}(\xi)$, and set
$$f_\xi= {\lab\chi_\xi \over |y_\xi(x)|^2} +2\Big\langle \grad\chi_\xi,\grad |y_\xi(x)|^{-2} \Big\rangle_g+{2\over |S|}
\int_{\mathbb{R}^2} {\chi'(|y|)\over |y|^3}\, dy.$$
By (\ref{Psi}) it follows that
$$\int_S f_\xi dv_g=\frac{1}{2\delta^2} \int_S \Delta_g \Psi_{\delta,\xi} dv_g+O(\delta^2)=O(\delta^2)$$
as $\delta \to 0$, where $\Psi_{\delta,\xi} \in H^1(S)$ is defined in \eqref{psidx}. Thus, $\int_S f_\xi dv_g=0$, and then
$F_\xi$ is well defined as the unique solution of
\begin{equation}\label{d2t} \left\{ \begin{array}{ll}-\lab
F_\xi=f_\xi &\text{in }S\\
\int_S F_\xi dv_g=0.&
\end{array}\right. \end{equation}
We have the following asymptotic expansion of $PU_{\de,\xi}$ as
$\delta \to 0$:
\begin{lem}\label{ewfxi}
The function $PU_{\delta,\xi}$ satisfies
$$PU_{\delta,\xi}=\chi_\xi \lf[U_{\delta,\xi}-\log(8\delta^2)\rg]+
8\pi H(x,\xi)+\alpha_{\delta,\xi}-2\delta^2 F_\xi+O(\delta^4|\log
\delta|)$$
uniformly in $S$, where $F_\xi$ is given in
\eqref{d2t} and
$$\alpha_{\delta,\xi}=-{4\pi\over|S|} \delta^2 \log \delta +2{\delta^2\over|S|}\lf(\int_{\mathbb{R}^2}
\chi(|y|) \frac{e^{\hat \varphi_\xi(y)}-1}{|y|^2}dy+ \pi- \int_{\mathbb{R}^2} {\chi'(|y|) \log |y|\over
|y| } dy \rg).$$
In particular, there holds
$$PU_{\delta,\xi}=8\pi G(x,\xi)-2{\de^2 \chi_\xi \over
|y_\xi(x)|^2}+\alpha_{\delta,\xi}-2 \delta^2 F_\xi+O(\delta^4|\log \delta|)$$
locally uniformly in $S \sm\{\xi\}$.
\end{lem}

\begin{proof}[\dem] Let us define
\begin{equation}\label{psidx}
\Psi_{\de,\xi}(x)= PU_{\de,\xi}(x)-\chi_\xi \hat U_{\delta,\xi}-8\pi H(x,\xi),
\end{equation}
where $\hat
U_{\de,\xi}=U_{\de,\xi}-\log(8\delta^2)$, for which there holds
\begin{equation*}
\int_S \Psi_{\de,\xi} dv_g= -\int_{S} \lf[\chi_\xi\hat
U_{\de,\xi}+ 8\pi H(x,\xi)\rg]\, dv_g=- \int_{S}
\chi_\xi\log{|y_\xi(x)|^4 \over(\delta^2 + |y_\xi(x)|^2)^2}
\,dv_g.
\end{equation*}
Since $\hat U_{\de,\xi}$ satisfies in $B_{2r_0}(\xi)$
$$-\Delta_g \hat U_{\de,\xi}=-e^{-\hat \varphi_\xi(y)} \Delta
u_{\delta,0} \Big|_{y=y_\xi(x)}= e^{-\varphi_\xi
}e^{U_{\de,\xi}}$$ in view of (\ref{laplacian}), by the equation
of $H(x,\xi)$ now we have that
\begin{eqnarray*}
-\Delta_g \Psi_{\de,\xi} &=& 2\lf\langle\grad\chi_\xi,\grad\hat
U_{\de,\xi}+4 \grad \log |y_\xi(x)|
\rg\rangle_g+\lap_g\chi_\xi\lf(\hat U_{\de,\xi}+4\log
|y_\xi(x)|\rg)\\
&&+{1\over|S|}\lf(8\pi- \int_S \chi_\xi e^{-\varphi_\xi}
e^{U_{\de,\xi}}\,dv_g \rg).
\end{eqnarray*}
Also, we have that in $A_{2r_0}(\xi)$
$$\hat U_{\de,\xi}+4\log |y_\xi(x)|=\log\frac{|y_\xi(x)|^4}{(\de^2+|y_\xi(x)|^2)^2}=-2{\de^2\over |y_\xi(x)|^2}+O(\de^4)$$
and
$$\nabla(\hat U_{\de,\xi}+4\log |y_\xi(x)|)=-2 \de^2 \nabla |y_\xi(x)|^{-2}+O(\de^4),$$
and there holds
\bebs \int_S \chi_\xi e^{-\varphi_\xi}
e^{U_{\de,\xi}}dv_g&=\int_{B_{2r_0}(0)\setminus
B_{r_0}(0)}\chi(|y|) {8\de^2\over
|y|^4}\,dy+O(\de^4)+\int_{B_{r_0}(0)}{8\de^2\over(\de^2+|y|^2)^2}\,dy\\
&= 8\pi-8\de^2\lf({\pi\over r_0^2}-\int_{B_{2r_0}(0)\setminus
B_{r_0}(0)}\frac{\chi(|y|)} {|y|^4} dy\rg)+O(\de^4)\\
&=8\pi+4 \de^2 \int_{\mathbb{R}^2} \frac{\chi'(|y|)}{|y|^3}
dy+O(\de^4)
\end{split}\ee in
view of $dv_g=e^{\hat \varphi_\xi} dy$ in the coordinate system
$y_\xi$ and
\begin{eqnarray*}
\int_{B_{2r_0}(0)\setminus B_{r_0}(0)}\frac{\chi(|y|)}{|y|^4}
dy=2\pi\int_{r_0}^{2r_0} \frac{\chi(r)}{r^3}dr=
\frac{\pi}{r_0^2}+\frac{1}{2}\int_{B_{2r_0}(0)\setminus
B_{r_0}(0)} \frac{\chi'(|y|)}{|y|^3}dy.
\end{eqnarray*}
By the definition of $f_\xi$ we then have that
\begin{equation} \label{Psi}
-\Delta_g \Psi_{\de,\xi}=-2\de^2 f_\xi+O(\de^4)\qquad \hbox{ in }
S. \end{equation}
Since $\int_S F_\xi dv_g=0$, by elliptic
regularity theory we get that
$$\Psi_{\de,\xi}=-2\delta^2 F_\xi+\frac{1}{|S|}\int_S
\Psi_{\delta,\xi}dv_g+O(\delta^4),$$ in view of \eqref{d2t}. On
the other hand, we have that
\begin{eqnarray*}
\int_S \Psi_{\de,\xi} dv_g&=& -\int_{S} \chi_\xi
\log\frac{|y_\xi(x)|^4}{(\delta^2 + |y_\xi(x)|^2)^2}\,dv_g=2
\int_{B_{r_0}(0)}
\log\frac{\delta^2 + |y|^2}{|y|^2}e^{\hat \varphi_\xi(y)}dy\\
&&+\int_{B_{2r_0}(0)\setminus B_{r_0}(0)}
\chi(|y|)\lf(2{\de^2\over |y|^2}+O(\de^4)\rg)e^{\hat
\varphi_\xi(y)}dy\\
&&=2\int_{B_{r_0}(0)} \log\frac{\delta^2 + |y|^2}{|y|^2}e^{\hat
\varphi_\xi(y)}dy+2\de^2\int_{\mathbb{R}^2\setminus B_{r_0}(0)} \chi(|y|) {e^{\hat \varphi_\xi(y)}-1 \over
|y|^2 }  dy\\
&&-4\pi \delta^2 \log r_0- 2\delta^2 \int_{\mathbb{R}^2} {\chi'(|y|) \log |y|\over
|y| } dy+O(\delta^4)
\end{eqnarray*}
in view of
\begin{equation} \label{chiovery2}
\int_{B_{2r}(0)\setminus B_{r}(0)}  {\chi(|y|) \over
|y|^2 }  dy=2\pi \int_{r}^{2r}  {\chi(t) \over
t }  dt= -2\pi \log r- \int_{B_{2r}(0)\setminus B_{r}(0)} {\chi'(|y|) \log |y|\over
|y| } dy
\end{equation}
for $r \leq r_0$. Since
\begin{eqnarray*}
\int_{B_{r_0}(0)} \log\frac{\delta^2 + |y|^2}{|y|^2}dy&=&\delta^2
\int_{B_{r_0/\delta}(0)}\log\frac{1 + |z|^2}{|z|^2}dz=2\pi
\delta^2 \int_0^\infty \Big[\log\frac{1 +
r^2}{r^2}-\frac{1}{r^2+1}\Big] r
dr\\
&&+\pi \delta^2 \log\Big(\frac{r_0^2}{\delta^2}+1\Big)+O(\delta^4)
\end{eqnarray*}
where $y=\delta z$, by $e^{\hat \varphi_\xi(y)}=1+O(|y|^2)$ we can
write that
\begin{eqnarray*}
&&2\int_{B_{r_0}(0)} \log\frac{\delta^2 + |y|^2}{|y|^2}e^{\hat
\varphi_\xi(y)}dy=2\int_{B_{r_0}(0)} \log\Big(
\frac{\delta^2}{|y|^2}+1\Big)(e^{\hat \varphi_\xi(y)}-1)dy-4 \pi
\delta^2 \log
\delta\\
&&+ 4\pi \delta^2\bigg[\log r_0+ \int_0^\infty \Big(\log\frac{1 +
r^2}{r^2}-\frac{1}{r^2+1}\Big) r dr\bigg]+O(\delta^4)\\
&&=2\delta^2 \int_{B_{r_0}(0)} \frac{e^{\hat
\varphi_\xi(y)}-1}{|y|^2}dy-4 \pi \delta^2 \log \delta + 4\pi
\delta^2\bigg[\log r_0+ \int_0^\infty \Big(\log\frac{1 +
r^2}{r^2}-\frac{1}{r^2+1}\Big) r dr\bigg]\\
&&+O(\delta^4|\log \delta|)
\end{eqnarray*}
by using that
\begin{eqnarray*}
&&\int_{B_{r_0}(0)} \bigg[\log\Big(
\frac{\delta^2}{|y|^2}+1\Big)-{\de^2\over |y|^2}\bigg](e^{\hat \varphi_\xi(y)}-1)dy=O\bigg(\delta^4 \int_{B_{\frac{r_0}{\delta}}(0)} \bigg| \log\Big(
\frac{1}{|y|^2}+1\Big)-{1\over |y|^2}\bigg|\,|y|^2\,dy\bigg)\\
&&=O\bigg(\delta^4 \int_{B_{\frac{r_0}{\delta}}(0)\setminus B_1(0)} \bigg| \log\Big(
\frac{1}{|y|^2}+1\Big)-{1\over |y|^2}\bigg|\,|y|^2\,dy\bigg)+O(\delta^4)=O(\delta^4|\log \delta|)
\end{eqnarray*}
in view of $\log(\frac{1}{|y|^2}+1)-{1\over |y|^2}=O({1\over |y|^4})$ as $|y|\to +\infty$. In conclusion, we get that
\begin{eqnarray*}
\int_S \Psi_{\de,\xi} dv_g= -4 \pi \delta^2 \log
\delta +2\delta^2 \left[\int_{\mathbb{R}^2}
\chi(|y|) \frac{e^{\hat \varphi_\xi(y)}-1}{|y|^2}dy+ \pi- \int_{\mathbb{R}^2} {\chi'(|y|) \log |y|\over
|y| } dy\right]+O(\delta^4 |\log \delta|)
\end{eqnarray*}
in view of $\int_0^\infty (\log\frac{1 +r^2}{r^2}-\frac{1}{r^2+1}) r dr=\frac{1}{2}.$ This completes the
proof.
\end{proof}

\noindent The ansatz will be constructed as follows. Given $m \in
\mathbb{N}$, let us consider distinct points $\xi_j \in \tilde S$
(i.e. $\xi_j \in S$ with $k(\xi_j)>0$) and $\delta_j>0$,
$j=1,\dots,m$. In order to have a good approximation, we will
assume that
\begin{equation}\label{repla0}
\delta_j^2=\delta^2 \rho_j(\xi_j)\quad \forall \,j=1,\dots,m,
\end{equation}
and
\begin{equation}\label{repla1}
\exists\, C>1\,:\,|\lambda-8\pi m |\le C \de^2|\log \de|,
\end{equation}
where $\de>0$ and $\rho_j$ is as in (\ref{ro}). Up to take $r_0$
smaller, we assume that the points $\xi_j$'s are well separated
and $k(\xi_j)$ is uniformly far from zero, namely, we choose
$\xi=(\xi_1,\dots,\xi_m)\in\Xi$, where
\begin{equation*}
\Xi=\{(\xi_1,\dots,\xi_m) \in S^m \mid d_g(\xi_i,\xi_j)\geq
4r_0\hbox{ and }k(\xi_j)\ge r_0
\:\:\forall\:i,j=1,\dots,m,\:i\not=j\}.
\end{equation*}
Denote $U_j:= U_{\delta_j,\xi_j}$ and $W_j=PU_j$, $j=1,\dots,m$,
where $P$ is the projection operator defined by \eqref{ePu}. Thus,
our approximating solution is $W(x)= \displaystyle \sum_{j=1}^m
W_j(x)$, parametrized by $(\de,\xi) \in (0,\infty) \times \Xi$.
Notice that for $r_0$ small enough we have that
$\ml{D}\subset\Xi\subset \ti S^m\sm\Delta$. We will look for a
solution $u$ of \eqref{mfeot} in the form $u=W+\phi$, for some
small remainder term $\phi$. In terms of $\phi$, the problem
\eqref{mfeot} is equivalent to find $\phi\in \bar H$ so that
\begin{equation}\label{ephi}
L(\phi)=-[R+N(\phi)] \qquad\text{ in $S$},
\end{equation}
where the linear operator $L$ is defined as
\begin{equation}\label{ol}
L(\phi) = \Delta_g \phi + \lambda {ke^{W}\over\int_S
ke^{W}dv_g}\lf(\phi - {\int_{S} ke^{W}\phi dv_g \over\int_S
ke^{W}dv_g} \rg),
\end{equation}
the nonlinear part $N$ is given by
\begin{equation}\label{nlt}
N(\phi)=\lambda \lf({ke^{W+\phi}\over\int_S
ke^{W+\phi}dv_g}-{ke^{W}\phi\over\int_S
ke^{W}dv_g}+\frac{ke^{W}\int_{S} ke^{W}\phi dv_g}{\lf(\int_S ke^W
dv_g\rg)^2}-{ke^{W}\over\int_S ke^{W}dv_g }\rg)
\end{equation}
and the approximation rate of $W$ is encoded in
\begin{equation}\label{R} R=\Delta_g W+\lambda
\lf({ke^{W}\over\int_S ke^{W}dv_g} - {1\over |S|}\rg).
\end{equation}
Notice that for all $\phi \in \bar H$
$$\int_S L(\phi) dv_g=\int_S N(\phi)dv_g=\int_S R dv_g=0.$$

\noindent In order to get the invertibility of $L$, let us
introduce the weighted norm
$$\| h \|_*=\sup_{x\in S} \lf[\sum_{j=1}^m \frac{\de_j^\sigma}{(\de_j^2 + \chi_{B_{r_0}(\xi_j)}(x) |y_{\xi_j}(x)|^2+r_0^2 \chi_{S\setminus B_{r_0}(\xi_j)}(x))^{1+\sigma/2}}\rg]^{-1} |h(x)|$$
for any $h\in L^\infty(S)$, where $0<\sigma<1$ is a small fixed
constant and $\chi_A$ denotes the characteristic function of the set $A$. Let us
evaluate the approximation rate of $W$ in $\|\cdot\|_*$:

\begin{lem}\label{estrr0}
Assume \eqref{repla0}-\equ{repla1}. There exists a constant
$C>0$, independent of $\de>0$ small, \st
\begin{equation}\label{re}
\|R\|_*\le  C\left(\delta |\nabla \varphi_m(\xi)|_g+\de^{2-\sigma}|\log \de| \right)
\end{equation}
for all $\xi \in \Xi$, where $|\nabla \varphi_m(\xi)|_g^2$ stands for $\displaystyle \sum_{j=1}^m |\nabla_{\xi_j} \varphi_m(\xi)|_g^2$.
\end{lem}
\begin{proof}[\dem]
First, from  Lemma \ref{ewfxi} we note that for any
$j\in\{1,\dots,m\}$
$$W_j(x)=U_j(x) - \log (8\delta_j^2) + 8\pi H(x,\xi_j)+O(\de^2 |\log \de|)$$
uniformly for $x\in B_{r_0}(\xi_j)$ and
$$W_j(x)=8\pi G(x,\xi_j) +O(\de^2 |\log \de|)$$
uniformly for $x$ on compact subsets of $S \sm\{\xi_j\}$. Since by symmetry and $\hat \varphi_{\xi_j}(0)=0$ we have
\begin{eqnarray*}
\int_{B_{r_0}(\xi_j)} \rho_j(x) e^{U_j} dv_g&=& \int_{B_{r_0}(0)}  \rho_j(y_{\xi_j}^{-1}(y)) {8 \delta_j^2\over (\delta_j^2 +
|y|^2 )^2}  e^{\hat \varphi_{\xi_j}(y)}dy\\
&=&\int_{B_{{r_0\over\delta_j}}(0)} \rho_j(\xi_j) {8 \over ( 1 + |y|^2)^2}\,(1+O(\delta_j^2 |y|^2))dy=8 \pi \rho_j(\xi_j) + O(\de^2|\log\de|),
\end{eqnarray*}
we then get that
\begin{eqnarray}\label{ikeW}
\int_S ke^W dv_g &=&\sum_{j=1}^m\frac{1}{8\delta_j^2}\int_{B_{r_0}(\xi_j)} k e^{U_j + 8\pi H(x,\xi_j)+8\pi \sum_{l\ne j}G(x,\xi_l)+O(\de^2|\log \de |)}dv_g + O(1) \nonumber\\
&=&\sum_{j=1}^m\frac{1}{8\delta_j^2}\int_{B_{r_0}(\xi_j)} \rho_j(x) e^{U_j} (1 + O(\de^2|\log \de |))dv_g + O(1) \nonumber \\
&=& \sum_{j=1}^m {1\over \delta_j^2}[\pi \rho_j(\xi_j) + O(\de^2|\log\de|)] + O(1)= {\pi m\over \de^2} + O(|\log\de|).
\end{eqnarray}
By Lemma \ref{ewfxi} and \equ{repla0}, (\ref{ikeW}) we have that
\begin{itemize}
\item in $S \setminus \cup_{j=1}^m B_{r_0}(\xi_j)$ there holds $8\pi m \frac{ k e^W}{\int_S ke^W dv_g}=O(\de^2)$  in view of $W(x)=O(1)$;
\item in $B_{r_0}(\xi_j)$, $j\in\{1,\dots,m\}$, there holds
\begin{eqnarray*}
8\pi m \frac{ k e^W}{\int_S ke^W dv_g}&=& 8\pi m \frac{k
e^{-\log(8\delta_j^2)+8\pi H(x,\xi_j) + 8\pi\sum_{l\ne
j}G(x,\xi_l)+O(\de^2|\log \de|)}}
{\pi m \de^{-2} + O(|\log\de|)} e^{U_j}\\
&=&\frac{8\pi m
\rho_j(x)+O(\de^2|\log \de|)}
{8\pi m\rho_j(\xi_j) + O(\de^2|\log\de|)}e^{U_j}\\
&=& \bigg[1+\Big\langle\frac{\nabla (\rho_j \circ
y_{\xi_j}^{-1})(0)}{\rho_j(\xi_j)},y_{\xi_j}(x)\Big\rangle+O(|y_{\xi_j}(x)|^2+\de^2
|\log \de|)\bigg]  e^{U_j},
\end{eqnarray*}
\end{itemize}
which can be summarized as follows:
\begin{eqnarray} \label{important}
\frac{8\pi m  k e^W}{\int_S ke^W dv_g}&=&\sum_{j=1}^m \chi_j
\bigg[1+\Big\langle\frac{\nabla (\rho_j \circ
y_{\xi_j}^{-1})(0)}{\rho_j(\xi_j)},y_{\xi_j}(x)\Big\rangle
+O(|y_{\xi_j}(x)|^2+\de^2 |\log \de|)\bigg] e^{U_j}\\
&&+O(\de^2) \chi_{S \setminus \cup_{j=1}^m B_{r_0}(\xi_j)},\nonumber
\end{eqnarray}
where $\chi_j=\chi_{\xi_j}$. Since as before
$$\int_S \chi_j e^{-\varphi_j} e^{U_j} dv_g=\int_{B_{r_0}(0)} {8 \delta_j^2\over (\delta_j^2 +
|y|^2 )^2}  dy+O(\delta^2)=8\pi + O(\de^2)$$
with $\varphi_j=\varphi_{\xi_j}$, for
$$R_{8\pi m}=\Delta_g W+8\pi m \lf({ke^{W}\over\int_S ke^{W}dv_g} - {1\over
|S|}\rg)$$
we then have that
\begin{eqnarray*}
R_{8\pi m}& &=-\sum_{j=1}^m \chi_j e^{-\varphi_j}e^{U_j} +8\pi m {ke^{W}\over\int_S ke^{W}dv_g} + {1\over
|S|}\sum_{j=1}^m \int_S \chi_j e^{-\varphi_j } e^{U_j}dv_g - {8\pi
m \over|S|}\\
&=&-\sum_{j=1}^m \chi_j e^{-\varphi_j}e^{U_j} +8\pi m {ke^{W}\over\int_S ke^{W}dv_g}+O(\de^2).
\end{eqnarray*}
By (\ref{important}) we now deduce that $R_{8\pi
m}(x)=O(\de^2)$ in $S \setminus \cup_{j=1}^m B_{r_0}  (\xi_j)$ and
\begin{eqnarray*}
R_{8\pi m}&=&  \lf[-e^{-\varphi_j}+1+O( |\nabla\log (\rho_j \circ y_{\xi_j}^{-1})(0)||y_{\xi_j}(x)|+\de^2 |\log \de|)\rg] e^{U_j}+O(\de^2)\\
&=& e^{U_j}O\lf(|\nabla \log(\rho_j \circ y_{\xi_j}^{-1})(0)|
|y_{\xi_j}(x)|+|y_{\xi_j}(x)|^2+\de^2|\log \de|\rg) +O(\de^2)
\end{eqnarray*}
in $B_{r_0}  (\xi_j)$, $j\in\{1,\dots,m\}$, in view of $\varphi_j(\xi_j)=0$ and $\nabla \varphi_j(\xi_j)=0$. From the definition of $\|\cdot \|_*$ we deduce the validity of
\begin{equation}\label{R8pim}
\|R_{8\pi m}\|_*\le  C\left(\delta |\nabla
\varphi_m(\xi)|_g+\de^{2-\sigma}\right)
\end{equation}
in view of $|\nabla \log(\rho_j \circ y_{\xi_j}^{-1})(0)|\leq |\nabla_{\xi_j} \varphi_m(\xi)|_g$. Since by (\ref{important})
$$R-R_{8\pi m}= (\lambda-8\pi m) \lf({ke^{W}\over\int_S ke^{W}dv_g} - {1\over |S|}\rg)=O\bigg(|\lambda-8\pi m| \sum_{j=1}^m
\chi_j e^{U_j}+|\lambda-8\pi m|\bigg),$$ we get that
$\|R-R_{8\pi m}\|_*= O(\delta^{-\sigma} |\lambda-8\pi
m|)$. In conclusion, by \equ{repla1} and \equ{R8pim} we deduce the
validity of \eqref{re}. \end{proof}


\section{The reduced energy}
\noindent The purpose of this section is to give an asymptotic
expansion of the ``reduced energy" $J_\la(W)$, where $J_\la$ is the energy functional given by
\eqref{energy}. For technical reasons, we will be concerned with establishing it in a $C^2$-sense in $\de$ and
just in a $C^1$-sense in $\xi$. To this aim, the following result will be very useful:

\begin{lem}\label{ieuf}
Letting $f\in C^{2,\gamma}(S)$ (possibly depending in $\xi$), $0<\gamma<1$, denote as $P_2(f)$ the second-order Taylor expansion of $f(x)$ at $\xi$:
$$P_2 f(x)=f(\xi)+\langle\grad (f \circ y_\xi^{-1}) (0), y_\xi(x)\rangle+{1\over2}\langle
D^2 (f\circ y_\xi^{-1})(0)y_\xi(x), y_\xi(x) \rangle.$$
The following expansions do hold as $\delta \to 0$:
\begin{eqnarray*}
\int_S \chi_\xi e^{-\varphi_\xi} f(x) e^{U_{\de,\xi}} dv_g&=& 8\pi f(\xi)-2 \de^2 \Delta_g f (\xi) \left[ 2\pi \log \delta+ \int_{\mathbb{R}^2} {\chi'(|y|) \log |y|\over |y| } dy +\pi \right]\\
&&+8\de^2\int_S \chi_\xi e^{-\varphi_\xi} {f(x)-P_2(f)(x)\over |y_{\xi}(x)|^4}\,dv_g+
4\de^2 f(\xi) \int_{\mathbb{R}^2} {\chi'(|y|) \over
|y|^3 } dy+o(\delta^2),
\end{eqnarray*}
\begin{eqnarray*}
\int_S \chi_\xi e^{-\varphi_\xi} f(x) e^{U_{\de,\xi}}\frac{dv_g}{\delta^2+|y_\xi(x)|^2} =
\frac{4\pi}{\delta^2}f(\xi)+\pi \Delta_g f(\xi)+O(\delta^{\gamma})
\end{eqnarray*}
and
\begin{eqnarray*}
\int_S \chi_\xi e^{-\varphi_\xi} f(x) e^{U_{\de,\xi}}\frac{a \delta^2-|y_\xi(x)|^2}{(\delta^2+|y_\xi(x)|^2)^2} dv_g
=\frac{4 \pi}{3 \de^2}(2a-1) f(\xi)+(a-2)\frac{\pi}{3} \Delta_g f(\xi) +O(\delta^\gamma)
\end{eqnarray*}
for $a \in \mathbb{R}$.
\end{lem}

\begin{proof}[\dem]
Since $dv_g=e^{\hat \varphi_\xi(y)}dy$, by symmetry observe that
\begin{eqnarray*}
&& \int_{S \sm B_{r_0}(\xi)} \chi_\xi e^{-\varphi_\xi}  f(x) e^{U_{\de,\xi}} dv_g= 8\de^2\int_{S \sm
B_{r_0}(\xi)} {\chi_\xi e^{-\varphi_\xi}  f(x)\over |y_{\xi}(x)|^4}\,dv_g+O(\de^4)\\
&&=
8\de^2\int_{S \sm
B_{r_0}(\xi)} \chi_\xi e^{-\varphi_\xi} {f(x)-P_2(f)(x)\over |y_{\xi}(x)|^4}\,dv_g+
8\de^2 f(\xi) \int_{B_{2r_0}(0) \sm
B_{r_0}(0)}  {\chi(|y|) \over |y|^4}\,dy\\
&&+2 \de^2 \Delta(f\circ y_\xi^{-1})(0) \int_{B_{2r_0}(0) \sm
B_{r_0}(0)} {\chi(|y|) \over |y|^2}dy+O(\de^4)
\end{eqnarray*}
as $\delta \to 0$. On $B_{r_0}(\xi)$ we get that
\begin{eqnarray*}
&&\int_{B_{r_0}(\xi)}\chi_\xi e^{-\varphi_\xi}  f(x) e^{U_{\de,\xi}} dv_g =\int_{B_{r_0}(0)}f(y_\xi^{-1}(y)) \frac{8\delta^2}{(\delta^2+|y|^2)^2}dy\\
&&=\int_{B_{r_0}(0)} P_2(f)(y_\xi^{-1}(y))
\frac{8\delta^2}{(\delta^2+|y|^2)^2} dy+\int_{B_{r_0}(0)} \lf((f \circ y_\xi^{-1})(y) -P_2(f)(y_\xi^{-1}(y))\right)
\frac{8\delta^2}{(\delta^2+|y|^2)^2} dy.\end{eqnarray*}
Since $f (x)-P_2(f)(x)=O(|y_\xi(x)|^{2+\gamma})$, by symmetry and the Lebesgue Theorem we get that
\begin{eqnarray*}
&&\int_{B_{r_0}(\xi)}\chi_\xi e^{-\varphi_\xi}f(x) e^{U_{\de,\xi}} dv_g =f(\xi) \int_{B_{r_0/\delta}(0)} \frac{8}{(1+|y|^2)^2} dy\\
&&+\delta^2 \Delta (f \circ y_\xi^{-1} )(0)
\int_{B_{r_0/\delta}(0)} \frac{2 |y|^2}{(1+|y|^2)^2} dy+8 \delta^2 \int_{B_{r_0}(\xi)} e^{-\varphi_\xi} \frac{f(x)  - P_2(f)(x)}{|y_\xi(x)|^4} dv_g+o(\delta^2)\\
&&=8 \pi f(\xi) \Big(1- \frac{\delta^2}{\delta^2+r_0^2}\Big)+2\pi
\delta^2 \Delta (f \circ y_\xi^{-1} )(0)
\left(\log{\delta^2+r_0^2  \over \delta^2}+{\delta^2 \over \delta^2+r_0^2}-1\right)\\
&&+8 \delta^2 \int_{B_{r_0}(\xi)} e^{-\varphi_\xi} \frac{f(x)  - P_2(f)(x)}{|y_\xi(x)|^4} dv_g+o(\delta^2)
=8 \pi f(\xi) \Big(1- \frac{\delta^2}{r_0^2}\Big)\\
&&+2\pi \delta^2 (-2\log \delta+2 \log r_0-1) \Delta (f \circ y_\xi^{-1} )(0)
+8 \delta^2 \int_{B_{r_0}(\xi)} e^{-\varphi_\xi} \frac{f(x)  - P_2(f)(x)}{|y_\xi(x)|^4} dv_g+o(\delta^2).\end{eqnarray*}
In view of (\ref{chiovery2}) and
\begin{equation} \label{chiovery2bis}
\int_{B_{2r}(0)\setminus B_{r}(0)}  {\chi(|y|) \over
|y|^4 }  dy=2\pi \int_{r}^{2r}  {\chi(t) \over
t^3 }  dt= \frac{\pi}{r^2}+ \frac{1}{2} \int_{B_{2r}(0)\setminus B_r(0)} {\chi'(|y|) \over
|y|^3 } dy\end{equation}
for $r \leq r_0$, summing up the two previous expansions we get that
\begin{eqnarray*}
\int_S \chi_\xi e^{-\varphi_\xi}  f(x) e^{U_{\de,\xi}} dv_g&=&8 \pi f(\xi) -2 \de^2 \Delta(f\circ y_\xi^{-1})(0)\left[ 2\pi \log \delta+ \int_{\mathbb{R}^2} {\chi'(|y|) \log |y|\over
|y| } dy +\pi \right]\\
&&+8\de^2\int_S \chi_\xi e^{-\varphi_\xi} {f(x)-P_2(f)(x)\over |y_{\xi}(x)|^4}\,dv_g+
4\de^2 f(\xi) \int_{\mathbb{R}^2} {\chi'(|y|) \over
|y|^3 } dy+o(\delta^2).
\end{eqnarray*}
Since by (\ref{laplacian}) $\Delta_g f(x)=e^{-\varphi_\xi(x)} \Delta (f \circ y_\xi^{-1}) (y_\xi(x))$, we get that
$\Delta (f \circ y_\xi^{-1}) (0)=\Delta_g f(\xi)$, and the validity of the first expansion then follows.
The other two expansions are simpler because of the stronger decay. Indeed, by the Taylor expansion of $f$ at $\xi$ and the symmetries we get that
\begin{eqnarray*}
&&\int_S \chi_\xi e^{-\varphi_\xi} f(x) e^{U_{\de,\xi}}\frac{dv_g}{\delta^2+|y_\xi(x)|^2} =
\frac{8}{\de^2} \int_{B_{r_0/\de}(0)} (f \circ y_\xi^{-1})(\de y) \frac{dy}{(1+|y|^2)^3}
+O(\delta^2)\\
&&=\frac{8}{\de^2} \left[f(\xi) \int_{\mathbb{R}^2} \frac{dy}{(1+|y|^2)^3}+\frac{\delta^2}{4} \Delta_g f(\xi) \int_{\mathbb{R}^2} \frac{|y|^2}{(1+|y|^2)^3} dy+O\Big(\de^{2+\gamma} \int_{\mathbb{R}^2} \frac{|y|^{2+\gamma}}{(1+|y|^2)^3} dy\Big)\right]+O(\delta^2)\\
&&=\frac{4\pi}{\delta^2}f(\xi)+\pi \Delta_g f(\xi)+O(\delta^{\gamma})
\end{eqnarray*}
in view of $\int_{\mathbb{R}^2} \frac{dy}{(1+|y|^2)^3}=\frac{\pi}{2}$ and
$$\int_{\mathbb{R}^2} \frac{|y|^2}{(1+|y|^2)^3} dy=
\int_{\mathbb{R}^2} \frac{dy}{(1+|y|^2)^2}-\int_{\mathbb{R}^2} \frac{dy}{(1+|y|^2)^3}=\frac{\pi}{2}.$$
Similarly, we have that
\begin{eqnarray*}
&&\int_S \chi_\xi e^{-\varphi_\xi} f(x) e^{U_{\de,\xi}}\frac{a \delta^2-|y_\xi(x)|^2}{(\delta^2+|y_\xi(x)|^2)^2} dv_g=
\frac{8}{\de^2} \int_{B_{r_0/\de}(0)} (f \circ y_\xi^{-1})(\de y) \frac{a-|y|^2}{(1+|y|^2)^4} dy+O(\delta^2)\\
&&=\frac{8}{\de^2} \left[f(\xi) \int_{\mathbb{R}^2} \frac{a-|y|^2}{(1+|y|^2)^4}dy+\frac{\delta^2}{4} \Delta_g f(\xi) \int_{\mathbb{R}^2} \frac{|y|^2(a-|y|^2)}{(1+|y|^2)^4} dy \right]+O(\delta^\gamma)\\
&&=\frac{4 \pi}{3 \de^2}(2a-1) f(\xi)+(a-2)\frac{\pi}{3} \Delta_g f(\xi) +O(\delta^\gamma)
\end{eqnarray*}
in view of
$$\int_{\mathbb{R}^2} \frac{a-|y|^2}{(1+|y|^2)^4}dy=(a+1) \int_{\mathbb{R}^2} \frac{dy}{(1+|y|^2)^4}-\int_{\mathbb{R}^2} \frac{dy}{(1+|y|^2)^3}=(2a-1)\frac{\pi}{6}$$
and
$$\int_{\mathbb{R}^2} \frac{|y|^2(a-|y|^2)}{(1+|y|^2)^4} dy=
- \int_{\mathbb{R}^2} \frac{dy}{(1+|y|^2)^2}
+(2+a) \int_{\mathbb{R}^2} \frac{dy}{(1+|y|^2)^3} -(1+a)\int_{\mathbb{R}^2} \frac{dy}{(1+|y|^2)^4}=(a-2)\frac{\pi}{6}.$$
The Lemma is completely established. \end{proof}

\medskip \noindent We are now ready to establish the expansion of $J_\la(W)$:
\begin{theo} \label{expansionenergy}
Assume \eqref{repla0}-\eqref{repla1}. The following expansion does
hold
\begin{equation} \label{JUt}
J_\lambda (W) =-8\pi m -\lambda \log (\pi m) -32\pi^2
\varphi_m(\xi)+ 2(\lambda -8\pi m)\log\delta
 +A(\xi) \delta^2 \log \delta-B(\xi)\delta^2+o(\de^2)
\end{equation}
in $C^2(\mathbb{R})$ and $C^1(\Xi)$ as $\de\to 0^+$, where
$\varphi_m(\xi)$, $A(\xi)$ and $B(\xi)$ are given by \eqref{fim},
\eqref{v} and \eqref{B}, respectively.
\end{theo}
\noindent The proof will be divided into several steps.
\begin{proof}[{\bf Proof (of (\ref{JUt}) in $C(\mathbb{R} \times \Xi)$):}]
First, let us consider the term
$$\int_S |\grad W|_g^2 dv_g= \int_S W (-\Delta_g W)dv_g= \sum_{j,l=1}^m \int_S \chi_j e^{-\varphi_j} e^{U_j}W_l dv_g$$
in view of $\int_S W dv_g=0$. Since by (\ref{green}) and (\ref{ePu})
\begin{eqnarray} \label{tricky}
\int_S \chi_j e^{-\varphi_j} e^{U_j} G(x,\xi_l) dv_g= \int_S (-\Delta_g PU_j) G(x,\xi_l) dv_g= PU_j(\xi_l)
\end{eqnarray}
for all $j,l=1,\dots,m$, by Lemmata \ref{ewfxi}, \ref{ieuf} and (\ref{tricky}) we have that for $l=j$
\begin{eqnarray*}
&& \int_S \chi_j e^{-\varphi_j}e^{U_j}W_j dv_g\\
&&= \int_S \chi_j e^{-\varphi_j} e^{U_j}\lf[\chi_j(U_j-\log(8\de_j^2))+8\pi H(x,\xi_j)+\alpha_{\de_j,\xi_j}-2\delta_j^2 F_{\xi_j}\rg]dv_g+O(\de^4 |\log \delta|)\\
&&= \int_S \chi_j e^{-\varphi_j} e^{U_j}\lf[\chi_j \log \frac{|y_{\xi_j}(x)|^4}{(\de_j^2+|y_{\xi_j}(x)|^2)^2}+8\pi G(x,\xi_j)+\alpha_{\de_j,\xi_j}-2\delta_j^2 F_{\xi_j}\rg]dv_g +O(\de^4 |\log \delta|)\\
&&=8 \int_{B_{2r_0/\de_j}(0)} \frac{\chi^2(\de_j |y|) }{(1+|y|^2)^2} \log \frac{|y|^4}{(1+|y|^2)^2} dy
+8\pi PU_j(\xi_j) +8\pi \alpha_{\de_j,\xi_j}-16\pi \delta_j^2 F_{\xi_j}(\xi_j)\\
&&+O(\de^4 |\log \delta|^2)=-16\pi -32 \pi \log \de_j +64\pi^2 H(\xi_j,\xi_j)+16\pi \alpha_{\de_j,\xi_j}-32\pi \delta_j^2 F_{\xi_j}(\xi_j)+O(\de^4 |\log \delta|^2)
\end{eqnarray*}
in view of
\begin{eqnarray*}
\int_{\mathbb{R}^2}{dy \over(1+|y|^2)^2}\log{|y|^4\over
(1+|y|^2)^2} =  2 \pi\int_0^\infty{ds \over(1+s)^2}\log {s \over 1+s} =- 2 \pi  \int_0^\infty {ds \over (1+s)^2}=- 2\pi
\end{eqnarray*}
by means of an integration by parts. Similarly, by Lemmata \ref{ewfxi}, \ref{ieuf}  and (\ref{tricky}) we have that for $l \not= j$
\begin{eqnarray*}
&&\int_S \chi_j e^{-\varphi_j}e^{U_j}W_l dv_g= \int_S \chi_j e^{-\varphi_j} e^{U_j}\lf[8\pi G(x,\xi_l)+\alpha_{\de_l,\xi_l}-2\delta_l^2 F_{\xi_l} \rg]dv_g+O(\de^4 |\log \delta|)\\
&&= 64\pi^2 G(\xi_l,\xi_j) +8\pi (\alpha_{\de_j,\xi_j}+\alpha_{\de_l,\xi_l})-16\pi (\delta_j^2 F_{\xi_j}(\xi_l)+ \delta_l^2 F_{\xi_l}(\xi_j))+O(\de^4 |\log \delta|^2).
\end{eqnarray*}
Setting
$$\alpha_{\de,\xi}=\sum_{j=1}^m \alpha_{\de_j,\xi_j}\,\qquad F_{\de,\xi}(x)=\sum_{j=1}^m \delta_j^2
F_{\xi_j}(x),$$ summing up the two previous expansions, for the gradient term we get that
\begin{eqnarray*}
{1\over 2}\int_S  |\grad W|_g^2 dv_g= -8 \pi m -16 \pi m \log \delta- 32 \pi^2 \varphi_m(\xi)+8 \pi m \alpha_{\de,\xi}-16 \pi \sum_{j=1}^m F_{\delta,\xi}(\xi_j) +o(\de^2)
\end{eqnarray*}
in view of (\ref{repla0}) and
$$8 \pi \sum_{j=1}^m \log \rho_j(\xi_j)-32 \pi^2\bigg[ \sum_{j=1}^m H(\xi_j,\xi_j)+
\sum_{l \not=j} G(\xi_l,\xi_j)\bigg]=32\pi^2 \varphi_m(\xi).$$

\medskip \noindent Let us now expand the potential term in $J_\lambda(W)$. By Lemma \ref{ewfxi} for any $j=1,\dots,m$ we find that
\begin{eqnarray*}
\int_{B_{r_0}(\xi_j)}k e^W dv_g &=&\int_{B_{r_0} (\xi_j)} \rho_j e^{U_j-\log(8\de_j^2)+\alpha_{\de,\xi}-2F_{\delta,\xi}
+O(\de^4|\log \delta|)} dv_g\\
&=&{1\over 8\de_j^2}\left[\int_S \chi_j e^{U_j} \rho_j  e^{\alpha_{\de,\xi}-2 F_{\delta,\xi}} dv_g
-8\delta_j^2 \int_{A_{2r_0}(\xi_j)} \frac{\chi_j \rho_j}{|y_{\xi_j}(x)|^4} dv_g+O(\delta^4 |\log \delta|)\right].
\end{eqnarray*}
By Lemma \ref{ieuf} (with $f(x)=e^{\varphi_j}\rho_j e^{\alpha_{\delta,\xi}-2F_{\delta,\xi}}$) we can now deduce that
\begin{eqnarray*}
&& 8 \delta_j^2 \int_{B_{r_0}(\xi_j)}k e^W dv_g = 8\pi \rho_j(\xi_j)  e^{\alpha_{\de,\xi}-2F_{\delta,\xi}(\xi_j)}
-4\pi   \left(\Delta_g  \rho_j (\xi_j)-2 K(\xi_j)\rho_j(\xi_j)\right) \delta_j^2 \log \delta_j \\
&&-2 \left(\Delta_g  \rho_j (\xi_j)-2 K(\xi_j)\rho_j(\xi_j)\right) \left( \int_{\mathbb{R}^2} {\chi'(|y|) \log |y|\over |y| } dy +\pi \right) \de_j^2
+4 \de_j^2 \rho_j(\xi_j) \int_{\mathbb{R}^2} {\chi'(|y|) \over
|y|^3 } dy\\
&&+8\de_j^2\int_{B_{r_0}(\xi_j)} \left[ ke^{8\pi \sum_{j=1}^m
G(x,\xi_j)} - e^{-\varphi_j}
\frac{ P_2(e^{\varphi_j}\rho_j)}{|y_{\xi_j}(x)|^4}\right]dv_g
-8\de_j^2\int_{A_{2r_0}(\xi_j)}  \chi_j e^{-\varphi_j}
\frac{ P_2(e^{\varphi_j}\rho_j)}{|y_{\xi_j}(x)|^4}\,dv_g
+o(\delta^2)
\end{eqnarray*}
in view of $\frac{\rho_j}{|y_{\xi_j}(x)|^4}=ke^{8\pi \sum_{j=1}^m
G(x,\xi_j)}$ in $B_{r_0}(\xi_j)$ and by (\ref{equationvarphi})
\begin{eqnarray} \label{gaussian}
\Delta_g  \left[e^{\varphi_j} \rho_j \right](\xi_j)=
\Delta_g  \rho_j (\xi_j)-2 K(\xi_j)\rho_j(\xi_j).
\end{eqnarray}
On the other hand, we have that
\begin{equation*}
\int_{S \sm \cup_{j=1}^m B_{r_0}(\xi_j)}
ke^W dv_g=\int_{S \sm \cup_{j=1}^m B_{r_0}(\xi_j)}
ke^{8\pi \sum_{j=1}^m G(x,\xi_j)}dv_g+O(\de^2|\log \delta|).
\end{equation*}
Since
$$\sum_{j=1}^m e^{-2 F_{\delta,\xi}(\xi_j)}=m-2 \sum_{j=1}^m
F_{\delta,\xi}(\xi_j)+O(\delta^4)$$
and by (\ref{repla0}) there holds
$$\delta_j^2 \log \delta_j=\rho_j(\xi_j) \delta^2 \log \delta+\frac{1}{2}\rho_j(\xi_j) \log \rho_j(\xi_j) \delta^2,$$
we then obtain that
\begin{eqnarray}
\frac{1}{\pi}e^{-\alpha_{\delta,\xi}}\delta^2 \int_{S}ke^W
dv_g=m
-   \frac{A(\xi)}{8\pi} \delta^2 \log \delta  +\frac{B_\chi(\xi)}{8\pi} \delta^2-2 \sum_{j=1}^m
F_{\delta,\xi}(\xi_j)+o(\delta^2),\label{intkeW}
\end{eqnarray}
where
\begin{eqnarray*}
B_\chi(\xi)&=&-2\pi \sum_{j=1}^m [\Delta_g  \rho_j(\xi_j) -2 K(\xi_j) \rho_j(\xi_j)] \log \rho_j(\xi_j)-\frac{A(\xi) }{2 \pi} \bigg( \int_{\mathbb{R}^2} {\chi'(|y|) \log |y|\over |y| } dy +\pi \bigg)\\
&& +4 \int_{\mathbb{R}^2} {\chi'(|y|) \over
|y|^3 } dy  \sum_{j=1}^m \rho_j(\xi_j) +8 \int_S
\left[ke^{8\pi \sum_{j=1}^m G(x,\xi_j)}-\sum_{j=1}^m \chi_j
e^{-\varphi_j} \frac{
P_2(e^{\varphi_j}\rho_j)}{|y_{\xi_j}(x)|^4}\right] dv_g. \nonumber
\end{eqnarray*}
By (\ref{chiovery2}), (\ref{chiovery2bis}) and the splitting of $S$ as the union of $\cup_{j=1}^m B_r(\xi_j)$ and $S \setminus \cup_{j=1}^m B_r(\xi_j)$, $r\leq r_0$, we easily deduce that
\begin{eqnarray*}
B_\chi(\xi)&=& -2\pi \sum_{j=1}^m [\Delta_g  \rho_j(\xi_j) -2 K(\xi_j) \rho_j(\xi_j)] \log \rho_j(\xi_j) -\frac{A(\xi)}{2}\\
&&+ 8 \int_{S \setminus \cup_{j=1}^m B_r(\xi_j)}   ke^{8\pi \sum_{j=1}^m G(x,\xi_j)} dv_g-\frac{8\pi}{r^2} \sum_{j=1}^m \rho_j(\xi_j)-A(\xi) \log \frac{1}{r} \\
&&+8\sum_{j=1}^m \int_{B_r(\xi_j)}\frac{
e^{\varphi_j(x)}\rho_j(x)-P_2(e^{\varphi_j}\rho_j)(x)}{|y_{\xi_j}(x)|^4}\,e^{-\varphi_j(x)} dv_g
\end{eqnarray*}
in view of $\frac{\rho_j}{|y_{\xi_j}(x)|^4}=ke^{8\pi \sum_{j=1}^m
G(x,\xi_j)}$ in $B_{r_0}(\xi_j)$, \eqref{gaussian} and the definitions of $A(\xi)$, $P_2(e^{\varphi_j}\rho_j)$. As a by-product we have that $B_\chi(\xi)$ does not depend on $\chi$ and $r \leq r_0$. Since
$$\lim_{r \to 0} \int_{B_r(\xi_j)}\frac{
e^{\varphi_j(x)}\rho_j(x)-P_2(e^{\varphi_j}\rho_j)(x)}{|y_{\xi_j}(x)|^4}\,e^{-\varphi_j(x)} dv_g=0$$
in view of $e^{\varphi_j(x)}\rho_j(x)-P_2(e^{\varphi_j}\rho_j)(x)=o(|y_{\xi_j}(x)|^2)$ as $x \to \xi_j$, we have that $B_\chi(\xi)$ coincides with $B(\xi)$ as defined in (\ref{B}).

\medskip \noindent Finally, we get the following expansion for $J_{8\pi m}(W)$ as $\delta \to 0$:
\begin{eqnarray} \label{energyMF}
J_{8\pi m}(W)=-8 \pi m(1+\log (\pi m)) -32 \pi^2
\varphi_m(\xi)+A(\xi) \de^2 \log
\de-B(\xi)\delta^2+o(\delta^2).
\end{eqnarray}
Since
$$\log \int_S k e^W dv_g=-2\log \delta+\log (\pi m)+O(\delta^2|\log
\delta|),$$ by \eqref{repla1} we then deduce that
\begin{eqnarray*}
J_\lambda(W)&=&J_{8\pi m}(W)-(\lambda-8\pi m) \log \int_S k e^W
dv_g \\
&=&J_{8\pi m}(W)-(\lambda-8\pi m)(-2\log \delta+\log (\pi
m))+O(\delta^4 |\log \delta|^2)
\end{eqnarray*}
and the proof is complete.
\end{proof}

\medskip \noindent We establish now expansion \eqref{JUt} in a $C^1$-sense in $\xi$, where the derivatives in $\xi$  are with respect to a given coordinate system.
\begin{proof}[{\bf Proof (of (\ref{JUt}) in $C^1(\Xi)$):}]
We just need to expand the derivatives of $J_\lambda(W)$ in $\xi$. Let us fix $i\in\{1,2\}$ and $j\in\{1,\dots,m\}$. We have that
$$\fr_{(\xi_j)_i}[J_\la(W)]=-\int_S\lf[\lab W+{\la ke^{W}\over
\int_S ke^W dv_g}\rg]\fr_{(\xi_j)_i}W dv_g.$$ Arguing as in Lemma \ref{ewfxi}, it is easy to show that
\begin{eqnarray}\label{dxiw} \fr_{(\xi_j)_i}W_q&=&-2{\chi_q \over
\de_q^2+|y_{\xi_q}(x)|^2}\left[\fr_{(\xi_j)_i} |y_{\xi_q}(x)|^2+\delta_q^2
\fr_{(\xi_j)_i} (\log \rho_q(\xi_q))
\right]\\
&&-4 \log |y_{\xi_q}(x)|\fr_{(\xi_j)_i}\chi_q +8\pi
\fr_{(\xi_j)_i} H(x,\xi_q)+O(\de^2|\log\de|)\nonumber
\end{eqnarray}
does hold uniformly in $S$. In particular there hold
$$\fr_{(\xi_j)_i}W_q=-16\pi {\chi_q \delta_q^2 \over
\de_q^2+|y_{\xi_q}(x)|^2} \fr_{(\xi_j)_i} G(\xi_q,\xi_j)
+O(\de^2|\log\de|)  \qquad \forall\:q \not= j$$ uniformly in $S$
and
\begin{eqnarray*} \fr_{(\xi_j)_i}W_j=8\pi \fr_{(\xi_j)_i}
G(x,\xi_j)+O(\de^2|\log\de|)
\end{eqnarray*}
locally uniformly in $S \setminus \{\xi_j \}$. Then we have that:\\
- for $q\ne l,j$
$$\int_S \chi_l e^{-\varphi_l} e^{U_l}\fr_{(\xi_j)_i}W_q dv_g=O(\de^2|\log\de|)$$
in view of $\chi_l \chi_q \equiv 0$;\\
- for $l\ne j$
\begin{eqnarray*}
\int_S \chi_l e^{-\varphi_l} e^{U_l}\fr_{(\xi_j)_i}W_l dv_g&=& -16
\pi \fr_{(\xi_j)_i} G(\xi_l,\xi_j) \int_S \chi_l^2 e^{-\varphi_l}
e^{U_l} {\delta_l^2 \over
\de_l^2+|y_{\xi_l}(x)|^2} dv_g +O(\de^2|\log\de|)\\
&=&-128 \pi  \fr_{(\xi_j)_i} G(\xi_l,\xi_j) \int_{B_{r_0\over
\de_l}(0)}{dy\over
(1+|y|^2)^3}+O(\de^2|\log\de|)\\
&=&-64\pi^2 \fr_{(\xi_j)_i} G(\xi_l,\xi_j)+O(\de^2|\log\de|)
\end{eqnarray*}
and
\begin{eqnarray*}
\int_S \chi_l e^{-\varphi_l}e^{U_l} \fr_{(\xi_j)_i}W_j dv_g=64 \pi^2
\fr_{(\xi_j)_i} G(\xi_l,\xi_j)+O(\de^2|\log\de|)
\end{eqnarray*}
in view of Lemma \ref{ieuf}. So we have that for $l \not= j$
$$\int_S \chi_l e^{-\varphi_l}e^{U_l} \fr_{(\xi_j)_i}W dv_g=O(\de^2 |\log \de|).$$
If $l=j$, by Lemma \ref{ieuf} we have that
\begin{eqnarray*}
&&\int_S \chi_j e^{-\varphi_j}e^{U_j}\fr_{(\xi_j)_i}W dv_g=\int_S \chi_j e^{-\varphi_j}e^{U_j}\fr_{(\xi_j)_i}W_j dv_g+O(\de^2 |\log \de|)\\
&&=\int_S \chi_j e^{-\varphi_j} e^{U_j} \left[
\chi_j \fr_{(\xi_j)_i}( U_j-\log(8 \de_j^2))+8\pi
\fr_{(\xi_j)_i} H(x,\xi_j)\right]dv_g+O(\de^2|\log\de|)\\
&&=
\fr_{(\xi_j)_i}\left[\int_S \chi_j^2 e^{-\varphi_j} e^{U_j}  dv_g \right]+\int_S \chi_j^2 e^{-\varphi_j} e^{U_j} \fr_{(\xi_j)_i} \varphi_j dv_g
-8\pi \fr_{(\xi_j)_i} \log \rho_j(\xi_j)+64\pi^2 \fr_{(\xi_j)_i} H(x,\xi_j)\Big|_{x=\xi_j}\\
&& +O(\de^2|\log\de|)=-8\pi \fr_{(\xi_j)_i} \log \rho_j(\xi_j)+64\pi^2 \fr_{(\xi_j)_i} H(x,\xi_j)\Big|_{x=\xi_j} +O(\de^2|\log\de|)
\end{eqnarray*}
in view of $\fr_{(\xi_j)_i} \log (8\de_j^2)=\fr_{(\xi_j)_i} \log \rho_j(\xi_j)$, $\fr_{(\xi_j)_i} \varphi_j(\xi_j)=0$ and
$$\fr_{(\xi_j)_i} \bigg(\int_S \chi_j^2 e^{-\varphi_j} e^{U_j}dv_g\bigg)= 8 \fr_{(\xi_j)_i}  \log \rho_j(\xi_j) \int_{\mathbb{R}^2}  \frac{|y|^2-1}{(1+|y|^2)^3} dy +O(\de^2)=O(\de^2).$$ In view of $\int_S \fr_{(\xi_j)_i}W dv_g=0$  we can compute
\begin{eqnarray}
-\int_S \lab W\fr_{(\xi_j)_i}W dv_g &=& \sum_{l=1}^m \int_S
\chi_l e^{-\varphi_l} e^{U_l}\fr_{(\xi_j)_i}W dv_g \nonumber \\
&=& -8\pi \fr_{(\xi_j)_i} \log \rho_j(\xi_j)+64\pi^2 \fr_{(\xi_j)_i} H(x,\xi_j)\Big|_{x=\xi_j} +O(\de^2|\log\de|) \nonumber \\
&=& -32 \pi^2 \fr_{(\xi_j)_i} \varphi_m(\xi) +O(\de^2|\log\de|)
\label{1term}
\end{eqnarray}
in view of $G(x_1,x_2)=G(x_2,x_1)$ for all $x_1 \not=x_2$  and $\fr_{(\xi_j)_i} H(x,\xi_j)\Big|_{x=\xi_j}= \frac{1}{2}\fr_{(\xi_j)_i} [H(\xi_j,\xi_j)]$. In order to give an expansion of the second term in $\fr_{(\xi_j)_i}[J_\la(W)]$, first observe that by Lemma \ref{ewfxi} there hold $ke^W={e^{\alpha_{\de,\xi}} \over 8\de_j^2} \rho_je^{U_j}[1+O(\de^2)]$ uniformly in $B_{r_0}(\xi_j)$ and $ke^W =O(1)$ uniformly in $S \sm\cup_{j=1}^m B_{r_0}(\xi_j)$.
So we have that
\begin{eqnarray*}
&& 8 e^{-\alpha_{\de,\xi}} \int_S ke^W \fr_{(\xi_j)_i}W dv_g=
\sum_{l,q=1}^m \de_q^{-2} \int_{B_{r_0}(\xi_q)} \rho_q e^{U_q}
(1+O(\delta^2)) \fr_{(\xi_j)_i}W_l dv_g+O(1)\\
&&= \de_j^{-2} \int_{B_{r_0}(\xi_j)} \rho_j e^{U_j}
 (1+O(\delta^2)) \fr_{(\xi_j)_i}W_j dv_g-16\pi
 \sum_{l \not= j} \fr_{(\xi_j)_i} G(\xi_l,\xi_j) \int_{B_{r_0}(\xi_l)} \rho_l e^{U_l} {dv_g \over \de_l^2+|y_{\xi_l}(x)|^2}\\
 &&+  8\pi \sum_{l \not=q} \de_q^{-2} \int_{B_{r_0}(\xi_q)} \rho_q
e^{U_q}  \fr_{(\xi_j)_i} H(x,\xi_l) dv_g +O(|\log \delta|)
\end{eqnarray*}
in view of $\fr_{(\xi_j)_i}W_l=8\pi \fr_{(\xi_j)_i} H(x,\xi_l)+O(\de^2|\log\de|)$ in $B_{r_0}(\xi_q)$ when $q \not=l$. Since
$$\fr_{(\xi_j)_i}W_j=\chi_j \fr_{(\xi_j)_i}[U_j-\log(8\de_j^2)] -4 \log |y_{\xi_j}(x)|\fr_{(\xi_j)_i}\chi_j +8\pi
\fr_{(\xi_j)_i} H(x,\xi_j)+O(\de^2|\log\de|),$$
we have that
\begin{eqnarray*}
&& \int_{B_{r_0}(\xi_j)} \rho_j e^{U_j} \fr_{(\xi_j)_i}W_j dv_g\\
&&=  \int_S \chi_j \rho_j e^{U_j}\fr_{(\xi_j)_i} U_j dv_g
 + \int_S \chi_j \rho_j e^{U_j}[ 8\pi \fr_{(\xi_j)_i} H(x,\xi_j)- \fr_{(\xi_j)_i}  \log \rho_j(\xi_j)]+O(\de^2|\log\de|)\\
&&=\fr_{(\xi_j)_i} \left[\int_S \chi_j \rho_j e^{U_j} dv_g \right]
 - \fr_{(\xi_j)_i}  \log \rho_j(\xi_j) \int_S \chi_j \rho_j e^{U_j}+O(\de^2|\log\de|)
\end{eqnarray*}
in view of $ \fr_{(\xi_j)_i}  \log \rho_j(x)=8\pi  \fr_{(\xi_j)_i}  H(x,\xi_j)$. Since by the Taylor expansion of $e^{\hat \varphi_j} (\rho_j \circ y_{\xi_j}^{-1})$ at $0$ and the symmetries we have that
\begin{eqnarray*}
&& \fr_{(\xi_j)_i} \left[\int_S  \chi_j \rho_j e^{U_j} dv_g \right]= \int_{B_{2r_0}(0)} \chi(|y|) \fr_{(\xi_j)_i} \left[ e^{\hat \varphi_j(y)} (\rho_j \circ y_{\xi_j}^{-1})(y)\right]  \frac{8\delta_j^2}{(\delta_j^2+|y|^2)^2}  dy \\
&&+\fr_{(\xi_j)_i}  \log \rho_j(\xi_j)\int_{B_{2r_0}(0)} \chi(|y|) e^{\hat \varphi_j(y)} (\rho_j \circ y_{\xi_j}^{-1})(y)  \frac{8\delta_j^2(|y|^2-\de_j^2)}{(\delta_j^2+|y|^2)^3}  dy =8\pi \fr_{(\xi_j)_i} \rho_j(\xi_j)+O(\delta^2|\log \delta|)
\end{eqnarray*}
in view of $\fr_{(\xi_j)_i} \left[ e^{\hat \varphi_j(0)} (\rho_j \circ y_{\xi_j}^{-1})(0)\right]
=\fr_{(\xi_j)_i} \rho_j(\xi_j)$ and $\int_{\mathbb{R}^2} \frac{|y|^2-1}{(1+|y|^2)^3}  dy=0$, by Lemma \ref{ieuf} we then deduce that
\begin{eqnarray*}
\int_{B_{r_0}(\xi_j)} \rho_j e^{U_j} \fr_{(\xi_j)_i}W_j dv_g= O(\de^2|\log\de|).
\end{eqnarray*}
Since by the Taylor expansion of $e^{\hat \varphi_l} (\rho_l \circ y_{\xi_l}^{-1})$ at $0$ and the symmetries we have that
\begin{eqnarray*}
\int_{B_{r_0}(\xi_l)} \rho_l e^{U_l} {dv_g \over
\de_l^2+|y_{\xi_l}(x)|^2} &=&
\delta_l^{-2} \int_{B_{r_0/\delta_l}(0)} (\rho_l \circ y_{\xi_l}^{-1} )(\delta_l y) e^{\hat \varphi_l(\delta_l y)}\frac{8}{(1+|y|^2)^3} dy\\
&=& \de^{-2} \int_{\mathbb{R}^2} \frac{8}{(1+|y|^2)^3} dy+O(1)=
\frac{4 \pi}{\de^2}+O(1),
\end{eqnarray*}
by Lemma \ref{ieuf} we obtain that
\begin{eqnarray*}
e^{-\alpha_{\de,\xi}} \int_S ke^W \fr_{(\xi_j)_i}W dv_g&=& - \frac{8
\pi^2}{\de^2}  \sum_{l \not= j} \fr_{(\xi_j)_i}
G(\xi_l,\xi_j)+  \frac{8 \pi^2}{\de^2} \sum_{l \not= j} \fr_{(\xi_j)_i}
H(\xi_l,\xi_j)+O(|\log \delta|).
\end{eqnarray*}
Since by (\ref{intkeW}) $\int_{S}ke^W dv_g = \frac{\pi
m}{\delta^2} e^{\alpha_{\delta,\xi}}  \left(1+O(\delta^2 |\log
\de | )\right)$, we finally get that
\begin{eqnarray}
&& \int_S {ke^W\over \int_S
ke^W dv_g} \fr_{(\xi_j)_i}W dv_g \label{2term}\\
&&=  - {8\pi \over m}\sum_{l \not=j}\fr_{(\xi_j)_i} G(\xi_l,\xi_j)+{8 \pi \over m}  \sum_{l\not=j }\fr_{(\xi_j)_i} H(\xi_l,\xi_j)
+O(\delta^2|\log \delta|)=O(\delta^2|\log \delta|) \nonumber
\end{eqnarray}
in view of $G(\xi_l,\xi_j)=H(\xi_l,\xi_j)$ for $l\not=j$. In conclusion, by (\ref{1term})-(\ref{2term}) we can write
\begin{eqnarray} \label{derivativexiMF}
\fr_{(\xi_j)_i}[J_{8\pi m}(W)]= -32 \pi^2 \fr_{(\xi_j)_i} \varphi_m(\xi)+O(\delta^2 |\log \delta|).
\end{eqnarray}
By (\ref{repla1}) we have that $\fr_{(\xi_j)_i}[J_\lambda(W)]= \fr_{(\xi_j)_i}[J_{8\pi}(W)]+O(\delta^2|\log \delta|)$, and the proof is complete.
\end{proof}

\medskip \noindent Finally, we address the expansions for the derivatives of $J_\la(W)$ in $\delta$.
\begin{proof}[{\bf Proof (of (\ref{JUt}) in $C^2(\mathbb{R})$):}]
We just focus on the first and second derivative of $J_\lambda(W)$ in $\de$. Since $\fr_\de= \rho_l^{\frac{1}{2}}(\xi_l) \fr_{\de_l} $ in view of \eqref{repla0}, arguing as in Lemma \ref{ewfxi}, it is easy to show that
\begin{eqnarray}\label{ddw}
&&\rho_l^{-\frac{1}{2}}(\xi_l)  \fr_\de W_l=- \chi_l \frac{4
\delta_l}{\delta_l^2+|y_{\xi_l}(x)|^2}+ \beta_{\delta_l,\xi_l}-4
\delta_l F_{\xi_l}+O(\delta^3 |\log
\delta|)\\
&&\rho_l^{-1}(\xi_l) \fr_{\de\de} W_l=4\chi_l
\frac{\delta_l^2-|y_{\xi_l}(x)|^2}{(\delta_l^2+|y_{\xi_l}(x)|^2)^2}+
\gamma_{\delta_l,\xi_l}-4 F_{\xi_l}+O(\delta^2 |\log
\delta|)\label{dddw}
\end{eqnarray}
do hold uniformly in $S$, where
$$\beta_{\delta_l,\xi_l}=-{8\pi\over|S|} \delta_l \log \delta_l+{4\delta_l \over|S|}\lf(\int_{\mathbb{R}^2}
\chi(|y|) \frac{e^{\hat \varphi_\xi(y)}-1}{|y|^2}dy- \int_{\mathbb{R}^2} {\chi'(|y|) \log |y|\over
|y| } dy \rg)$$
and
$$\gamma_{\delta_l,\xi_l}=-{8\pi\over|S|} \log \delta_l+{4 \over|S|}\lf(\int_{\mathbb{R}^2}
\chi(|y|) \frac{e^{\hat \varphi_\xi(y)}-1}{|y|^2}dy-2\pi -
\int_{\mathbb{R}^2} {\chi'(|y|) \log |y|\over |y| } dy \rg).$$ By
Lemma \ref{ieuf} we then have that
\begin{eqnarray*}
&&\rho_l^{-\frac{1}{2}}(\xi_l) \int_S \chi_j e^{-\varphi_j}e^{U_j} \fr_\de W_l dv_g\\
&&=
-\int_S \chi_j \chi_l e^{-\varphi_j}e^{U_j}  \frac{4 \delta_l}{\delta_l^2+|y_{\xi_l}(x)|^2}dv_g+8\pi
 \beta_{\delta_l,\xi_l}-32 \pi \delta_l F_{\xi_l}(\xi_j)+O(\delta^3|\log \delta|^2)\\
 &&= - \frac{32}{\delta_j} \delta_{jl} \int_{B_{r_0/\delta_j}(0)}  \frac{dy}{(1+|y|^2)^3}+8\pi
 \beta_{\delta_l,\xi_l}-32 \pi  \delta_l F_{\xi_l}(\xi_j)+O(\delta^3|\log \delta|^2)\\
&&= - \frac{16 \pi}{\delta_j} \delta_{jl}  +8\pi
\beta_{\delta_l,\xi_l}-32 \pi  \delta_l F_{\xi_l}(\xi_j)+O(\delta^3|\log \delta|^2),
\end{eqnarray*}
\begin{eqnarray*}
&&\rho_l^{-1}(\xi_l)\int_S \chi_j e^{-\varphi_j}e^{U_j} \fr_{\de \de} W_l dv_g\\
&&=
4\int_S \chi_j \chi_l e^{-\varphi_j}e^{U_j}  \frac{\delta_l^2-|y_{\xi_l}(x)|^2}{(\delta_l^2+|y_{\xi_l}(x)|^2)^2} dv_g+8\pi
 \gamma_{\delta_l,\xi_l}-32 \pi  F_{\xi_l}(\xi_j)+O(\delta^2 |\log \delta|^2)\\
&&= \frac{32 }{\delta_j^2} \delta_{jl} \int_{B_{r_0/\delta_j}(0)} \left[\frac{2}{(1+|y|^2)^4}-\frac{1}{(1+|y|^2)^3} \right]dy+8\pi
 \gamma_{\delta_l,\xi_l}-32 \pi  F_{\xi_l}(\xi_j)+O(\delta^2 |\log \delta|^2)\\
 &&= \frac{16 \pi}{3\delta_j^2} \delta_{jl}  +8\pi
 \gamma_{\delta_l,\xi_l}-32 \pi F_{\xi_l}(\xi_j)+O(\delta^2 |\log \delta|^2)
\end{eqnarray*}
and
\begin{eqnarray*}
&& \rho_l^{-\frac{1}{2}}(\xi_l) \int_S \chi_j e^{-\varphi_j}e^{U_j} \fr_\de U_j \fr_\de W_l
dv_g= \frac{2}{\de} \rho_l^{-\frac{1}{2}}(\xi_l) \int_S \chi_j e^{-\varphi_j}e^{U_j}
\frac{|y_{\xi_j}(x)|^2-\de_j^2}{\de_j^2+|y_{\xi_j}(x)|^2} \fr_\de
W_l dv_g\\
&&= 8 \rho_j(\xi_j)^{\frac{1}{2}} \de_{jl} \int_S \chi_j
e^{-\varphi_j}e^{U_j}
\frac{\de_j^2-|y_{\xi_j}(x)|^2}{(\de_j^2+|y_{\xi_j}(x)|^2)^2} dv_g
+ \frac{16}{\de} \left(\beta_{\de_l,\xi_l}-4\de_l
F_{\xi_l}(\xi_j)\right)\int_{\mathbb{R}^2}
\frac{|y|^2-1}{(1+|y|^2)^3}dy\\
&&+O(\de^2 |\log \de|)=\frac{32 \pi}{3 \de_j^2} \rho_j(\xi_j)^{\frac{1}{2}}
\de_{jl}+O(\de^\gamma)
\end{eqnarray*}
in view of $\int_{\mathbb{R}^2} \frac{|y|^2-1}{(1+|y|^2)^3}dy=0$, where $\delta_{jl}$ denotes the
Kronecker's symbol. Since $\int_S \fr_\de W dv_g=\int_S \fr_{\de
\de} W dv_g=0$, we then deduce the following expansions:
\begin{eqnarray}
\int_S (-\Delta_g W)  \fr_\de W dv_g&=& \sum_{j,l=1}^m \int_S \chi_j e^{-\varphi_j}e^{U_j} \fr_\de W_l dv_g \label{deW} \\
&=&- \frac{16 \pi m}{\de}
+8\pi m \sum_{l=1}^m \rho_l^{\frac{1}{2}}(\xi_l) \beta_{\delta_l,\xi_l}-32 \pi \de \sum_{j,l=1}^m
\rho_l(\xi_l) F_{\xi_l}(\xi_j)+O(\delta^3|\log \de|^2), \nonumber
\end{eqnarray}
\begin{eqnarray}
\int_S (-\Delta_g W)  \fr_{\de \de} W dv_g&=&  \sum_{j,l=1}^m \int_S \chi_j e^{-\varphi_j}e^{U_j} \fr_{\de \de} W_l dv_g \label{dedeW}\\
&=&\frac{16 \pi m}{3 \de^2} +8\pi m \sum_{l=1}^m \rho_l(\xi_l)
\gamma_{\delta_l,\xi_l}-32 \pi \sum_{j,l=1}^m
\rho_l(\xi_l) F_{\xi_l}(\xi_j)+O(\de^2 |\log \de|^2)  \nonumber
\end{eqnarray}
and
\begin{eqnarray}
\int_S -\Delta_g (\fr_\de W)  \fr_\de W dv_g= \sum_{j,l=1}^m \int_S
\chi_j e^{-\varphi_j}e^{U_j} \fr_\de U_j \fr_\de W_l dv_g =
\frac{32 \pi m}{3 \de^2} +O(\de^\gamma) \label{deW2}
\end{eqnarray} as $\de \to 0$. Since by Lemma \ref{ewfxi}
there hold
$$ke^W={e^{\alpha_{\de,\xi}-2 F_{\delta,\xi}(x)}\over 8\de_j^2} \rho_je^{U_j}[1+O(\de^4|\log \delta|)]$$
uniformly in $B_{r_0}(\xi_j)$ and $ke^W=O(1)$, $\fr_\de W=
O(\delta |\log \delta|)$ uniformly in $S \setminus \cup_{j=1}^m
B_{r_0}(\xi_j)$,
by Lemma \ref{ieuf} we can write that
\begin{eqnarray*}
&&\int_S ke^W \fr_\de W dv_g = \sum_{j,l=1}^m
\int_{B_{r_0}(\xi_j)} ke^W \fr_\de W_l dv_g +O(\delta |\log
\delta|)
\\
&&=- \sum_{j=1}^m {e^{\alpha_{\de,\xi}}\over 2 \de}
\int_{B_{r_0}(\xi_j)} e^{-2 F_{\delta,\xi}(x)}
 \frac{ \rho_je^{U_j}}{\delta_j^2+|y_{\xi_j}(x)|^2}
dv_g+\pi {e^{\alpha_{\de,\xi}}\over \de^2}  \left(m \sum_{l=1}^m \rho_l^{\frac{1}{2}}(\xi_l)
\beta_{\delta_l,\xi_l}-4 \sum_{j,l=1}^m \rho_l^{\frac{1}{2}}(\xi_l) \delta_l F_{\xi_l}(\xi_j) \right)\\
&&+O(|\log \delta|)= - \sum_{j=1}^m
{e^{\alpha_{\de,\xi}}\over 2 \de} \left(
\frac{4\pi}{\de^2} e^{-2F_{\de,\xi}(\xi_j)}+\pi
(\Delta_g \rho_j(\xi_j)-2K(\xi_j) \rho_j(\xi_j)) \right)\\
&& +\pi {e^{\alpha_{\de,\xi}}\over \de^2} \left(m \sum_{l=1}^m
\rho_l^{\frac{1}{2}}(\xi_l) \beta_{\delta_l,\xi_l}-4 \sum_{j,l=1}^m \rho_l^{\frac{1}{2}}(\xi_l) \delta_l
F_{\xi_l}(\xi_j) \right)+O(\delta^{-1+\gamma})\\
&&=\pi {e^{\alpha_{\de,\xi}}\over \de^2} \left(-\frac{2m}{\de}+m \sum_{l=1}^m
\rho_l^{\frac{1}{2}}(\xi_l) \beta_{\delta_l,\xi_l}-{\de \over 8\pi} A(\xi)+O(\delta^{1+\gamma})\right)
\end{eqnarray*}
in view of (\ref{gaussian}) and
\begin{eqnarray} \label{expansionF}
\frac{1}{\de} \sum_{j=1}^m e^{-2F_{\de,\xi}(\xi_j)}=\frac{m}{\de}-\frac{2}{\de} \sum_{j,l=1}^m \de_l^2 F_{\xi_l}(\xi_j)+O(\de^3)=
\frac{m}{\de}-2 \sum_{j,l=1}^m \rho_l^{\frac{1}{2}}(\xi_l)  \de_l F_{\xi_l}(\xi_j)+O(\de^3).
\end{eqnarray}
Combining with (\ref{intkeW}) we then get that
\begin{eqnarray} \label{Merry}
\frac{\int_S ke^W \fr_\de W dv_g}{\int_S ke^W dv_g} &=& -\frac{2}{\de}+\sum_{l=1}^m
\rho_l^{\frac{1}{2}}(\xi_l) \beta_{\delta_l,\xi_l}-{\de \over 8\pi m} A(\xi)-\frac{A(\xi)}{4 \pi m} \delta \log \delta+\frac{B(\xi)}{4\pi m} \delta\\
&&-\frac{4}{m \delta}\sum_{j=1}^m F_{\delta,\xi}(\xi_j)+o(\de),\nonumber
\end{eqnarray}
which yields to
\begin{eqnarray} \label{derivdelta}
\fr_{\de}[J_{8\pi m}(W)]&=&
\int_S (-\Delta_g W)  \fr_\de W dv_g-8\pi m \frac{\int_S ke^W \fr_\de W dv_g}{\int_S ke^W dv_g} \nonumber\\
&=& 2 A(\xi) \delta \log \delta+[A(\xi)-2 B(\xi)] \delta+o(\de).
\end{eqnarray}
Since by (\ref{repla1}) and (\ref{Merry})  there holds
$$-(\lambda-8\pi m) \frac{\int_S ke^W \fr_\de W dv_g}{\int_S ke^W dv_g}=\frac{2(\lambda-8\pi m)}{\de}+O(\de^3 |\log \de|^2),$$
by (\ref{derivdelta}) we deduce the validity of (\ref{JUt}) for the first derivative in $\de$.

\medskip \noindent Towards the expansion of the second derivative, we proceed in a similar way with the aid of the expansion for $\fr_{\de\de} W_l$.
Since
$$ke^W={e^{\alpha_{\de,\xi}-2 F_{\delta,\xi}(x)}\over 8\de_j^2} \rho_je^{U_j}[1+O(\de^4|\log \delta|)]$$
and $ke^W=O(1)$, $\fr_{\de \de} W+(\fr_\de W)^2=
O(|\log \delta|)$ do hold uniformly in $B_{r_0}(\xi_j)$ and $S \setminus \cup_{j=1}^m
B_{r_0}(\xi_j)$, respectively, by Lemma \ref{ieuf} we can write that
\begin{eqnarray*}
&&\int_S ke^W [\fr_{\de \de}W+(\fr_\de W)^2]  dv_g = \sum_{j=1}^m
\int_{B_{r_0}(\xi_j)} ke^W [\fr_{\de \de} W+ (\fr_\de W)^2 ] dv_g +O(|\log
\delta|)
\\
&&= \sum_{j=1}^m {e^{\alpha_{\de,\xi}}\over 2 \de^2}
\int_{B_{r_0}(\xi_j)} e^{-2 F_{\delta,\xi}(x)}
 \rho_je^{U_j} \frac{5\delta_j^2-|y_{\xi_j}(x)|^2}{(\delta_j^2+|y_{\xi_j}(x)|^2)^2} dv_g\\
 &&-\sum_{j=1}^m {e^{\alpha_{\de,\xi}}\over \de}
\int_{B_{r_0}(\xi_j)} e^{-2 F_{\delta,\xi}(x)}
 \frac{\rho_je^{U_j} }{\delta_j^2+|y_{\xi_j}(x)|^2} \left( \sum_{l=1}^m \rho_l^{\frac{1}{2}}(\xi_l)  \beta_{\delta_l,\xi_l}- 4 \de \sum_{l=1}^m
\rho_l(\xi_l) F_{\xi_l}\right)
dv_g\\
&&+\pi {e^{\alpha_{\de,\xi}}\over \de^2}  \left(m \sum_{l=1}^m \rho_l(\xi_l)
\gamma_{\delta_l,\xi_l}-4 \sum_{j,l=1}^m \rho_l (\xi_l) F_{\xi_l}(\xi_j) \right)+O(|\log \delta|^2),
\end{eqnarray*}
and then
\begin{eqnarray*}
&&{\delta^2 e^{-\alpha_{\de,\xi}}\over \pi} \int_S ke^W [\fr_{\de \de}W+(\fr_\de W)^2]  dv_g = \sum_{j=1}^m
\left(
\frac{6}{\de^2} e^{-2F_{\de,\xi}(\xi_j)}+\frac{1}{2}
(\Delta_g \rho_j(\xi_j)-2K(\xi_j) \rho_j(\xi_j)) \right)\\
&&-{4\over \de} \sum_{j=1}^m
\left( \sum_{l=1}^m \rho_l^{\frac{1}{2}}(\xi_l)  \beta_{\delta_l,\xi_l}- 4 \de \sum_{l=1}^m
\rho_l(\xi_l) F_{\xi_l}(\xi_j) \right)+m \sum_{l=1}^m
\rho_l (\xi_l) \gamma_{\delta_l,\xi_l}-4 \sum_{j,l=1}^m \rho_l (\xi_l) F_{\xi_l}(\xi_j) +O(\delta^{\gamma})\\
&&=\frac{6m}{\de^2}+m \sum_{l=1}^m
\rho_l(\xi_l) \gamma_{\delta_l,\xi_l}-{4m \over \de} \sum_{l=1}^m  \rho_l^{\frac{1}{2}}(\xi_l)  \beta_{\delta_l,\xi_l}
+{A(\xi) \over 8 \pi} +O(\delta^{\gamma})
\end{eqnarray*}
in view of (\ref{gaussian}) and (\ref{expansionF}).
Combining with (\ref{intkeW}) we then get that
\begin{eqnarray}
&&\frac{\int_S ke^W [\fr_{\de \de} W+(\fr_{\de} W)^2] dv_g}{\int_S ke^W dv_g} = \frac{6}{\de^2}+{3A(\xi)\over 4 \pi m} \log \de +
\sum_{l=1}^m
\rho_l(\xi_l) \gamma_{\delta_l,\xi_l} \label{Merry1}\\
&&-{4 \over \de} \sum_{l=1}^m  \rho_l^{\frac{1}{2}}(\xi_l)  \beta_{\delta_l,\xi_l}+{A(\xi) \over 8 \pi m} -\frac{3 B(\xi)}{4 \pi m}+\frac{12}{m \delta^2}\sum_{j=1}^m F_{\delta,\xi}(\xi_j) +o(1). \nonumber
\end{eqnarray}
Since
\begin{eqnarray*}
\fr_{\de \de}[J_\lambda(W)]&=&
\int_S (-\Delta_g W)  \fr_{\de \de} W dv_g-\lambda \frac{\int_S ke^W [\fr_{\de \de} W+(\fr_{\de} W)^2] dv_g}{\int_S ke^W dv_g} \\
&&+\int_S (-\Delta_g \fr_\de W)  \fr_\de W dv_g+\lambda
\bigg(\frac{\int_S ke^W \fr_\de W dv_g}{\int_S ke^W dv_g}\bigg)^2,
\end{eqnarray*}
by (\ref{dedeW}), (\ref{deW2}), (\ref{Merry}) and (\ref{Merry1}) we deduce that
\begin{eqnarray} \label{derivdelta2}
\fr_{\de \de}[J_{8\pi m}(W)]=2 A(\xi)\log \de +3A(\xi)-2 B(\xi) +o(1).
\end{eqnarray}
Since by (\ref{repla1}), (\ref{Merry}) and (\ref{Merry1})
$$(\lambda-8\pi m ) \bigg[\frac{\int_S ke^W [\fr_{\de \de} W+(\fr_{\de} W)^2] dv_g}{\int_S ke^W dv_g} - \bigg(\frac{\int_S ke^W \fr_\de W dv_g}{\int_S ke^W dv_g}\bigg)^2\bigg]=2{\la-8\pi m \over\de^2}+O(\de^2 |\log \de|^2),$$
by (\ref{derivdelta2}) we deduce the validity of (\ref{JUt}) also for the second derivative in $\de$, and the proof is complete. \end{proof}


\section{Variational reduction and proof of main results}
\noindent In the so-called nonlinear Lyapunov-Schimdt reduction,
the first step is the solvability theory for the operator $L$
given in (\ref{ol}), obtained as the linearization of
\eqref{mfeot} at the approximating solution $W$. As $\de \to 0$
observe that formally the operator $L$, scaled and centered at $0$
by setting $y =y_{\xi_j}(x)/\delta_j $, approaches $\hat L$
defined in $\mathbb{R}^2$ as
$$\hat L(\phi) = \Delta\phi+{8\over (1+|y|^2)^2}\lf(\phi -
{1\over\pi}\int_{\R^2}{\phi(z)\over (1+|z|^2)^2}\,dz\rg).$$
Due to the intrinsic invariances, the kernel of $\hat L$ in $L^\infty(\R^2)$  is non-empty and is spanned by $1$ and $Y_j$, $j=0,1,2$, where
\begin{equation*}
Y_{i}(y) = { 4 y_i \over 1+|y|^2} ,\qquad
i=1,2,\qquad\text{and}\qquad
Y_{0}(y) = 2\,{1-|y|^2\over 1+|y|^2}.
\end{equation*}
Since \cite{DeKM,EGP} it is by now rather standard to show the invertibility of $L$ in a suitable ``orthogonal" space, and a sketched proof of it will be given in Appendix A. However, for Dirichlet Liouville-type equations on bounded domains as in \cite{DeKM,EGP}, the corresponding limiting operator $\tilde L$ takes the form $\tilde L(\phi)=\Delta\phi+{8\over (1+|y|^2)^2}\phi$ and the function $1$ does not belong to its kernel, making possible to disregard the ``dilation parameters" $\de_i$ in the reduction. As we will see, one additional parameter $\de$ is needed in the reduction and in this respect our problem displays a new feature w.r.t. Dirichlet Liouville-type equations, making our situation very similar to the one arising in the study of critical problems in higher dimension.

\medskip \noindent To be more precise, for $i=0,1,2$ and $j=1,\dots,m$ introduce the functions
\begin{equation*}
Z_{ij}(x) = Y_i\lf({y_{\xi_j}(x)\over \delta_j}\rg)=\left\{\begin{array}{ll}
2  {\de_j^2- |y_{\xi_j}(x)|^2\over \de_j^2+|y_{\xi_j}(x)|^2} &\hbox{for }i=0\\
 {4\de_j(y_{\xi_j}(x))_i \over \de_j^2+|y_{\xi_j}(x)|^2}&\hbox{for }i=1,2, \end{array} \right.
\end{equation*}
and set $Z=\displaystyle \sum_{l=1}^m Z_{0l}$. For $i=1,2$ and $j=1,\dots,m$, let $PZ$, $PZ_{ij}$ be the projections of
$Z$, $Z_{ij}$ as the solutions in $\bar H$ of
\begin{equation} \label{ePZ}
\begin{array}{rl}
\Delta_g PZ&=\chi_j \Delta_g Z-\frac{1}{|S|}\int_S
\chi_j \Delta_g Z dv_g\\
 \Delta_g PZ_{ij} &=\chi_j \Delta_g Z_{ij}-\frac{1}{|S|}\int_S
\chi_j  \Delta_g Z_{ij} dv_g.
\end{array}
\end{equation}
In Appendix A we will prove the following result:
\begin{prop} \label{p2}
There exists $\delta_0>0$ so that for all $0<\delta\leq \delta_0$,
$h\in C(S)$ with $\int_Sh\,dv_g=0$, $\xi \in \Xi$ there is a
unique solution $\phi \in \bar H \cap W^{2,2}(S)$ and
$c_0,c_{ij} \in \mathbb{R}$ of
\begin{equation}\label{plco}
\left\{ \begin{array}{ll}
L(\phi) = h +c_0 \Delta_g PZ + \displaystyle \sum_{i=1}^{2} \sum_{j=1}^m c_{ij} \Delta_g PZ_{ij}&\text{in }S\\
\int_S \phi \Delta_g PZ dv_g=\int_S \phi \Delta_g PZ_{ij} dv_g=0 &\forall\: i=1,2,\, j=1,\dots,m.
\end{array} \right.
\end{equation}
Moreover, the map $(\de,\xi) \mapsto (\phi,c_0,c_{ij})$ is
twice-differentiable in $\de$ and one-differentiable in $\xi$ with
\begin{eqnarray}
&&\|\phi \|_\infty \le C |\log \de| \|h\|_*\:,\qquad |c_0|+\displaystyle \sum_{i=1}^{2} \sum_{j=1}^m  |c_{ij}|\le C\|h\|_* \label{estmfe1} \\
&& \|\fr_\de \phi\|_\infty+\sum_{i=1}^2 \sum_{j=1}^m  \|\fr_{(\xi_j)_i} \phi\|_\infty +
{\de \over |\log \de|} \|\fr_{\de\de} \varphi \|_\infty \le C {|\log \de|^2 \over\de} \|h\|_*\label{estd}
\end{eqnarray}
for some $C>0$.
\end{prop}

\medskip \noindent Let us recall that  $u=W+\phi$  solves \eqref{mfeot} if $\phi\in \bar H$ does satisfy (\ref{ephi}). Since the operator $L$ is not fully invertible, in view of Proposition \ref{p2} one can solve the nonlinear problem (\ref{ephi}) just up to a linear combination of $\Delta_g PZ$ and $\Delta_g PZ_{ij}$, as explained in the following (see Appendix B for the proof):
\begin{prop}\label{lpnlabis}
There exists $\delta_0>0$ so that for
all $0<\delta\leq \delta_0$, $\xi \in \Xi$ problem
\begin{equation}\label{pnlabis}
\left\{ \begin{array}{ll}
L(\phi)= -[R+N(\phi)] +c_0\Delta_g PZ +\displaystyle \sum_{i=1}^{2}\sum_{j=1}^m c_{ij} \Delta_g
PZ_{ij}& \text{in } S\\
\int_S \phi \Delta_g PZ dv_g=\int_S \phi \Delta_g PZ_{ij} dv_g= 0 &\forall\:
i=1,2,\, j=1,\dots,m
\end{array} \right.
\end{equation}
admits a unique solution $\phi(\de,\xi) \in \bar H \cap
W^{2,2}(S)$ and $c_0(\de,\xi),\,c_{ij}(\de,\xi) \in \R$, $i=1,2$
and $j=1,\dots,m$, where $\de_j>0$ are as in \eqref{repla0} and
$N$, $R$ are given by \eqref{nlt}, \eqref{R}, respectively. Moreover, the map
$(\delta,\xi)\mapsto
(\phi(\delta,\xi),c_0(\de,\xi),c_{ij}(\de,\xi))$ is
twice-differentiable in $\de$ and one-differentiable in $\xi$ with
\begin{eqnarray} \label{cotaphi1bis}
&& \|\phi\|_\infty\le C\left( \delta |\log \delta |  |\nabla \varphi_m(\xi)|_g+ \delta^{2-\sigma}|\log\delta|^2\right)\\
\label{cotadphi1bis}
&& \|\fr_\de \phi \|_\infty+\sum_{i=1}^2 \sum_{j=1}^m \|\fr_{(\xi_j)_i} \phi \|_\infty\le C \left( |\log \delta |^2 |\nabla \varphi_m(\xi)|_g+\delta^{1-\sigma}|\log\delta|^3\right) \\
\label{cotad2phi1bis} && \|\fr_{\de\de}\phi\|_\infty\le C \left(
\delta^{-1} |\log \delta |^3 |\nabla \varphi_m(\xi)|_g +\delta^{-\sigma}|\log\delta|^4\right).
\end{eqnarray}
\end{prop}
\noindent The function $W+\phi(\de,\xi)$ will be a true solution of \eqref{ephi} if $\delta$ and $\xi$ are such that
$c_0(\de,\xi)=c_{ij}(\delta,\xi)=0$ for all $i=1,2,$ and $j=1,\dots,m$. This problem is equivalent to finding critical
points of the reduced energy $E_\lambda(\delta, \xi)= J_\lambda(W+\phi(\delta,\xi))$,
where $J_\lambda$ is given by \eqref{energy}, as stated in
\begin{lem}\label{cpfc0bis}
There exists $\delta_0$ such that, if $(\delta,\xi)\in
(0,\de_0]\times \Xi$ is a critical point of $E_\lambda$, then
$u=W+\phi(\delta,\xi)$ is a solution of \eqref{mfeot}, where
$\de_i$ are given by \eqref{repla0}.
\end{lem}
\noindent Once equation \eqref{mfeot} has been reduced to the search of c.p.'s for $E_\lambda$, it becomes crucial to show that the main asymptotic term of $E_\lambda$ is given by $J_\lambda(W)$, for which an expansion has been given in Theorem \ref{expansionenergy}. More precisely, by the estimates in Appendix B we have that
\begin{theo} \label{fullexpansionenergy}
Assume \eqref{repla0}-\eqref{repla1}. The following expansion does
hold
\begin{eqnarray} \label{fullJUt}
E_\lambda (\de,\xi) &=&-8\pi m -\lambda \log (\pi m) -32\pi^2
\varphi_m(\xi)+ 2(\lambda -8\pi m)\log\delta
 +A(\xi) \delta^2 \log \delta\\
 &&-B(\xi)\delta^2+o(\de^2)+r_\lambda(\de,\xi) \nonumber
\end{eqnarray}
in $C^2(\mathbb{R})$ and $C^1(\Xi)$ as $\de\to 0^+$, where
$\varphi_m(\xi)$, $A(\xi)$ and $B(\xi)$ are given by \eqref{fim},
\eqref{v} and \eqref{B}, respectively. The term
$r_\lambda(\de,\xi)$ satisfies
\begin{eqnarray} \label{rlambda}
|r_\lambda(\de,\xi)|+\frac{\de}{|\log \de|} |\nabla
r_\lambda(\de,\xi)|+\frac{\de^2}{|\log \de|^2} |\partial_{\de \de}
r_\lambda(\de,\xi)| \leq C \delta^2 |\log \delta |\, |\nabla \varphi_m(\xi)|_g^2
\end{eqnarray}
for some $C>0$ independent of $(\de,\xi)\in(0,\de_0]\times\Xi$.
\end{theo}
\noindent We are now in position to establish the main result stated in the Introduction.
\begin{proof}[{\bf Proof (of Theorem \ref{main2}):}] According to Lemma \ref{cpfc0bis}, we just need to find a critical point of $E_\lambda(\de,\xi)$. By Theorem \ref{fullexpansionenergy} for $\lambda > 8\pi m$ we have that
\begin{eqnarray*}
\frac{(\de\, \partial_\de E_\lambda)(\sqrt{\lambda-8\pi m}\, \mu,\xi)}{\lambda-8\pi m} &=& 2+ A(\xi) \log (\lambda-8\pi m) \mu^2+2 A(\xi) \mu^2 \log \mu+(A(\xi)-2B(\xi))\mu^2\\
&&+o(1)+O \left(\mu^2  |\log (\sqrt{ \lambda-8\pi m}\mu)|^2 |\nabla \varphi_m(\xi)|_g^2\right)
\end{eqnarray*}
and
\begin{eqnarray*}
\frac{(\de^2\, \partial_{\de \de} E_\lambda)(\sqrt{\lambda-8\pi m} \, \mu,\xi)}{\lambda-8\pi m } &=& -2+A(\xi) \mu^2 \left[ 2 \log \mu+ \log (\lambda-8\pi m) +3 \right]-2B(\xi) \mu^2\\
&&+o(1)+O \left( \mu^2 |\log (\sqrt{ \lambda-8\pi m}\mu)|^3 |\nabla \varphi_m
(\xi)|_g^2\right)
\end{eqnarray*}
as $\lambda \to 8\pi m$. By assumption we can find $a_0>0$ small so that $B(\xi)>0$ for all $\xi \in U$ with $|A(\xi)| \leq a_0$. Let
$$\DD_\lambda=\{\xi \in U:\, |\nabla \varphi_m(\xi)|_g \leq 	\sqrt 2 \;|\log(\lambda-8\pi m )|^{-3}\}$$
and consider the interval
$$I_\lambda =\Big[\frac{m_0}{\sqrt{|\log(\lambda-8\pi m)|}},M \Big],$$
where
$$0<m_0 =\inf_{\xi \in U} |A(\xi)|^{-\frac{1}{2}}\:,\qquad M=2 \sup_{\{\xi \in U: \, |A(\xi)| \leq a_0 \} } B^{-\frac{1}{2}}(\xi)<+\infty.$$
For $\lambda$ close to $8\pi m$ and for all $\xi \in \DD_\lambda$ we have that
$$\frac{(\de\, \partial_\de E_\lambda)(\sqrt{\lambda-8\pi m}\, \mu,\xi)}{\lambda-8\pi m} \Big|_{\mu=\frac{m_0}{\sqrt{|\log(\lambda-8\pi m)|}}}
=2-A(\xi) m_0^2(1+o(1))-2B(\xi)\frac{m_0^2}{|\log(\lambda-8\pi m)|}+o(1)>0$$
in view of $A(\xi)m_0^2\leq 1$, and
$$\frac{(\de\, \partial_\de E_\lambda)(\sqrt{\lambda-8\pi m}\, \mu,\xi)}{\lambda-8\pi m} \Big|_{\mu=M}=
2-A(\xi)M^2 |\log(\lambda-8\pi m)| (1+o(1))-2B(\xi)M^2+o(1)<0$$
since either $A(\xi)\geq a_0$ or $0\le A(\xi)\leq a_0$, $B(\xi) M^2\geq 4$. Moreover, in $I_\lambda \times \DD_\lambda$ we have that
$$\frac{(\de^2\, \partial_{\de \de} E_\lambda)(\sqrt{\lambda-8\pi m} \, \mu,\xi)}{\lambda-8\pi m }=
-2-A(\xi) \mu^2 |\log (\lambda-8\pi m)|(1+o(1))-2B(\xi) \mu^2+o(1)
\leq -1$$
since either $A(\xi)\geq a_0$ or $0\le A(\xi)\leq a_0$, $B(\xi)>0$.
So, for all $\lambda$ close to $8\pi m$ and $\xi \in \DD_\la$ there exists an unique $\mu(\lambda,\xi)
\in \hbox{Int }I_\lambda$ so that $\de(\lambda,\xi):= \sqrt{\lambda-8\pi
m} \, \mu(\lambda,\xi)$ satisfies $\partial_\de
E_\lambda(\de(\lambda,\xi),\xi)=0$. Moreover, by the IFT the map
$\xi \in  \DD_\la \to \de(\lambda,\xi)$ is a $C^1-$function of
$\xi$ with
\begin{eqnarray*}
\partial_\xi \de(\lambda,\xi)=-\frac{\partial_{\de \xi} E_\lambda(\de(\la,\xi),\xi)}{\partial_{\de\de}E_\la(\de(\la,\xi),\xi)}
=O\big( |\log(\lambda-8\pi m)|^{-3} \big),
\end{eqnarray*}
in view of $\mu^2(\la,\xi) |\partial_{\de \de} E_\lambda(\sqrt{\lambda-8\pi m} \, \mu(\la,\xi),\xi) |\geq1 $ and $\partial_{\de \xi} E_\lambda(\de,\xi)=O(\de |\log \de|+|\log\de|^3|\grad\varphi_m(\xi)|_g^2)$ (as it can be easily shown by the methods in the proof of Theorem \ref{expansionenergy}).\\
The aim now is to extend the map $\de(\la,\xi)$ to the whole $U$ in a $C^1-$way. Letting $\eta \in C_0^\infty[-2,2]$ be a cut-off function so that $\eta =1$ in $[-1,1]$, we define the $C^1-$extension $\tilde
\delta$ of $\de$ to $\DD$ as
\begin{eqnarray*}
\tilde \delta(\lambda,\xi)&=&\eta\lf(|\log(\lambda-8\pi m )|^6
|\nabla \varphi_m(\xi)|_g^2\rg)   \delta (\lambda,\xi)+\sqrt{\la-8\pi m}\bigg[1-\eta\lf(|\log(\lambda-8\pi m )|^6
|\nabla \varphi_m(\xi)|_g^2\rg) \bigg]
\end{eqnarray*}
and $\tilde E_\lambda (\xi)=E_\lambda(\tilde
\de(\lambda,\xi),\xi)$. Since $|\fr_\xi \tilde
\delta(\lambda,\xi)|=O( |\log(\lambda-8\pi m )|^{-3})$, by Theorem
\ref{fullexpansionenergy} we have that
$$\tilde E_\lambda (\xi)=-8\pi m -\lambda \log (\pi m) -32\pi^2
\varphi_m(\xi)+ O(|\lambda -8\pi m|\, |\log (\lambda-8\pi m )|)$$
 and
$$\nabla_\xi \tilde E_\lambda (\xi)= \nabla_\xi E_\lambda (\tilde\de (\lambda,\xi),\xi)+\partial_\de E_\lambda (\tilde \de(\lambda,\xi),\xi)\partial_\xi \tilde \delta(\lambda,\xi)=-32\pi^2 \nabla \varphi_m(\xi)+ O(\sqrt{\lambda -8\pi m} |\log (\lambda-8\pi m )|^{2})$$
uniformly in $\xi \in U$. Since $\DD$ is a stable critical set of $\varphi_m$ (according to Definition \ref{stable}), we find a
critical point $\xi_\lambda \in U$ of $\tilde E_\lambda (\xi)+8\pi m+ \la \log(\pi m)$, which is also a c.p. of $\tilde E_\lambda (\xi)$. By $\nabla_\xi
\tilde E_\lambda (\xi_\lambda)=0$ we get that
$$\nabla \varphi_m(\xi_\lambda)=O(\sqrt{\lambda -8\pi m} |\log (\lambda-8\pi m )|^{2}),$$
and then $\xi_\lambda \in \DD_\la$. Moreover $\tilde
\de(\lambda,\xi)=\de(\lambda,\xi)$ satisfies $\partial_\de
E_\lambda(\de(\lambda,\xi_\lambda),\xi_\lambda)=0$, and then
$\nabla_\xi \tilde E_\lambda (\xi_\lambda)=0$ is equivalent to
$\nabla_\xi E_\lambda (\de(\lambda,\xi_\lambda),\xi_\lambda)=0$.
In conclusion, up to take $U$ smaller so that $\grad\varphi_m(\xi)\ne 0$ for all $\xi\in U\sm\DD$, the pair $(\de(\lambda,\xi_\lambda),\xi_\lambda)$
is a c.p. of $E_\lambda(\de,\xi)$ and, along a sub-sequence, $\xi_\la \to q \in \DD$ as $\la \to 8\pi m$. By construction, the corresponding solution has the required asymptotic properties.
\end{proof}
\begin{obs} \label{minmax} i) The validity of condition (\ref{cond}) just on $\DD$ is enough to provide Theorem \ref{main2} in the case of $\DD=\{\xi_0\}$, where $\xi_0$ is a non-degenerate local minimum/maximum point of $\varphi_m$. In this case, we just consider a small ball $B_{s_\lambda}(\xi_0)$ as $\DD_\lambda$, with $s_\lambda=|\log(\lambda-8\pi m)|^{-3}$. Since $A(\xi_0)\geq 0$ and $\nabla \varphi_m(\xi_0)=0$ we have that $A(\xi)\geq -C_0 s_\lambda$ and $|\nabla \varphi_m(\xi)|_g \leq C_0 s_\lambda$ for all $\xi \in B_{s_\lambda}(\xi_0)$ and some $C_0>0$. Since $B(\xi)>0$ for all $\xi \in B_{s_\lambda}(\xi_0)$  if $A(\xi_0)=0$, it is easy to see as before that for $\lambda$ close to $8\pi m$ and for all $\xi \in B_{s_\lambda}(\xi_0)$
$$\partial_\de E_\lambda(\sqrt{\lambda-8\pi m}\, \mu,\xi) \Big|_{\mu=\frac{m_0}{\sqrt{|\log(\lambda-8\pi m)|}}}>0,\:\:\:\:\: \partial_\de E_\lambda(\sqrt{\lambda-8\pi m}\, \mu,\xi) \Big|_{\mu=M}<0$$
with
$$\partial_{\de \de} E_\lambda(\sqrt{\lambda-8\pi m} \, \mu,\xi)\leq -\frac{1}{\mu^2}$$
in $I_\lambda \times B_{s_\la}(\xi_0)$. So, for all $\lambda$ close to $8\pi m$ we can still find a $C^1-$map
$\xi \in  B_{s_\lambda}(\xi_0) \to \de(\lambda,\xi)$ so that $\partial_\de E_\lambda(\de(\lambda,\xi),\xi)=0$.
Setting $\tilde E_\lambda (\xi)=E_\lambda(\de(\lambda,\xi),\xi)$ for $\xi \in B_{s_\lambda}(\xi_0)$, by Theorem \ref{fullexpansionenergy} we have that
$$\tilde E_\lambda (\xi)=-8\pi m -\lambda \log (\pi m) -32\pi^2
\varphi_m(\xi)+ O(|\la-8 \pi m|\,|\log(\la-8\pi m)|).$$
Since by the non-degeneracy of $\xi_0$ we have on $\partial B_{s_\lambda}(\xi_0)$ that
$\varphi_m(\xi)\geq \varphi_m(\xi_0)+ C_1 s_\la^2$ / $\varphi_m(\xi)  \leq  \varphi_m(\xi_0)- C_1 s_\la^2$ for some $C_1>0$, we can find an interior minimum/maximum point $\xi_\lambda \in  B_{s_\lambda}(\xi_0)$ of $\tilde E_\lambda (\xi)$ on $ B_{s_\lambda}(\xi_0)$. By $\partial_\de
E_\lambda(\de(\lambda,\xi_\lambda),\xi_\lambda)=0$, we also deduce that $\nabla_\xi E_\lambda (\de(\lambda,\xi_\lambda),\xi_\lambda)=0$, and the pair $(\de(\lambda,\xi_\lambda),\xi_\lambda)$
is the c.p. of $E_\lambda(\de,\xi)$ we were searching for.\\[0.2cm]
ii) If (\ref{cond}) does hold just in $\DD$, Theorem \ref{main2} is also valid in the special case $|A(\xi)|=O(|\nabla \varphi_m(\xi)|_g)$. Indeed, condition (\ref{cond}) reduces to $B(\xi)>0$ ($<0$) on $\DD$ and in $\DD_\la$ we have that $A(\xi)\geq -C_0 |\log(\lambda-8\pi m )|^{-3}$ for some $C_0>0$. Similarly as in point (i), it is still possible to define the map $\xi \in \DD_\lambda \to \delta(\lambda,\xi)$, and the remaining argument in the proof of Theorem \ref{main2} works also in this case by extending $\delta(\lambda,\xi)$ on a small neighborhood $U$ of $\DD$ in $\tilde S^m\setminus \Delta$. \end{obs}


\section{Proof of Theorem \ref{main3}}
\noindent In this section, we shall study the existence of
non-topological solutions of \eqref{CSoriginal}. To this
purpose we look for a solution to the equivalent problem
\eqref{CS} of the form $w=u+c_-(u)$ with $\int_Tu=0$ and we are
lead to study \eqref{CSMF}. Assume that $N$ is even, so that equation \eqref{CSMF} is a perturbation of (\ref{mfeot})$_ {\la=8\pi m}$ with $m={N\over 2}$. Notice that the energy
functional of \eqref{CS} is given by
$$\ti I_\e(w)={1\over 2}\int_T|\grad w|^2+{1\over
2\e^2}\int_T(ke^{w}-1)^2+{4\pi N\over |T|}\int_Tw, \qquad w\in
H^1(T).$$ Introduce the notation $\ds C(u):=16\pi N {\int_T k^2
e^{2u}\over (\int_Tke^{u})^2}$, so that $\ds e^{c_-(u)}=\frac{8\pi
N \epsilon^2}{\int_T k e^u \big(1+\sqrt{1-\e^2C(u)}\big)}$ and
\begin{equation}\label{iuc}
\begin{split}
I_\e(u):=\ti I_\e(u+c_-(u))=&\,J_{4\pi N}(u)-4\pi
N\log\Big(1+\sqrt{1-\e^2C(u)}\Big)-{4\pi N\over
1+\sqrt{1-\e^2C(u)}}\\
&\,+\;4\pi N\log(8\pi N\e^2)+{|T|\over 2\e^2}-2\pi N.
\end{split}
\end{equation}
Hence, if $u\in\ml{A}_\e=\{u\in \bar H\mid \e^2 C(u)\le 1\}$ is
a critical point of $I_\e$ with $\epsilon^2 C(u)<1$, then $u+c_-(u)$ is a solution to
\eqref{CS} and $u$ is a solution to \eqref{CSMF}. Observe that
$I_\e$ is a perturbation of $J_{8\pi m}$ as $\e\to
0^+$, in view of $4\pi N=8\pi m$.

\medskip \noindent Given $m$ distinct points $\xi_j\in T\sm\{p_1,\dots,p_l\}$, $j=1,\dots,m$, we will define
$\de_j$ according to \eqref{repla0} and assume
\begin{equation}\label{repcs1}
\exists\, C>1\,:\,\e\le C \de^2.
\end{equation}
Letting $\ds
W(x)=\sum_{j=1}^mW_j(x)$, we look for a solution of \eqref{CSMF} in the form $u=W+\phi$, for
some small remainder term $\phi$. In terms of $\phi$, problem
\eqref{CSMF} is equivalent to find $\phi\in \bar H$ so that
$W+\phi\in\ml{A}_\e$ and
\begin{equation}\label{ephit}
L^\e(\phi)=-[R^\e+N^\e(\phi)] \qquad\text{ in $T$}.
\end{equation}
The linear operator $L^\e$ is defined as
$$L^\e(\phi) = \Delta \phi + 4\pi N{ke^{W}\over\int_T
ke^{W}}\lf(\phi - {\int_{T} ke^{W}\phi \over\int_T ke^{W}}
\rg)+\Lambda^\epsilon(\phi),$$
where \bebs \Lambda^\epsilon(\phi)=&\,\frac{4\pi
N\e^2C(W)}{(1+\sqrt{1-\e^2C(W)})^2}\bigg({ke^W\over\int_Tke^W}-{2k^2e^{2W}\over\int_Tk^2e^{2W}}\bigg)\bigg[\phi-{\int_T
ke^{W}\phi\over \int_Tke^W}\\
&\,+\frac{\e^2
C(W)}{(1+\sqrt{1-\e^2C(W)})\sqrt{1-\e^2C(W)}}\bigg({\int_T
k^2e^{2W}\phi\over \int_Tk^2e^{2W}}-{\int_Tke^W \phi\over \int_T
ke^{W}}\bigg)\bigg]\\
&+{4\pi Nke^W\over
\int_Tke^W}\frac{\e^2C(W)}{(1+\sqrt{1-\e^2C(W)})\sqrt{1-\e^2C(W)}}\bigg({\int_Tk^2e^{2W}\phi\over\int_Tk^2e^{2W}}-{\int_Tke^W\phi
\over\int_Tke^W}\bigg). \end{split}\ee
Observe that $L^\e$ is defined for all $\phi\in \bar H$. The nonlinear part $N^\e$ is well-defined for $\phi\in
\bar H$ such that $W+\phi\in\ml{A}_\e$ and is given by
\begin{equation}\label{nltcs}
\begin{split}
N^\e(\phi)=&\,4\pi N\left(\frac{k e^{W+\phi}}{\int_T
ke^{W+\phi}}-{ke^{W}\phi\over\int_T ke^{W}}+ {ke^W\int_{T}
ke^{W}\phi \over(\int_T ke^{W})^2}-\frac{k e^W}{\int_T
ke^W}\right)-\Lambda_\epsilon(\phi)\\
&\,+ \frac{4 \pi N \epsilon^2
C(W+\phi)}{\big(1+\sqrt{1-\epsilon^2C(W+\phi)}\big)^2}\left(\frac{ke^{W+\phi}}{\int_T
k e^{W+\phi}}-\frac{k^2
e^{2(W+\phi)}}{\int_T k^2 e^{2(W+\phi)}}\right)\\
&\,- \frac{4 \pi N \epsilon^2
C(W)}{\big(1+\sqrt{1-\epsilon^2C(W)}\big)^2}\left(\frac{ke^W}{\int_T
ke^W}-\frac{k^2 e^{2W}}{\int_T k^2 e^{2W}}\right).
\end{split}
\end{equation}
The approximation rate of $W$ becomes
\begin{equation}\label{Rcs}
\begin{split}
R^\e=&\,\Delta W+4\pi N\left(\frac{k e^W}{\int_T
ke^W}-\frac{1}{|T|}\right)+\frac{4 \pi N \epsilon^2
C(W)}{\big(1+\sqrt{1-\epsilon^2C(W)}\big)^2}\left(\frac{ke^W}{\int_T
k e^W}-\frac{k^2
e^{2W}}{\int_T k^2 e^{2W}}\right).
\end{split}
\end{equation}
We have that
\begin{lem}
Let $N$ be an even number and $m={N\over 2}$. Assume \eqref{repla0} and \equ{repcs1}. There
exists a constant $C>0$, independent of $\de>0$ small, \st for all
$\xi \in \Xi$
\begin{equation}\label{ree}
\|R^\e\|_*\le  C\left(\delta |\nabla \varphi_m(\xi)|_g+\de^{2-\sigma} \right).
\end{equation}
\end{lem}
\begin{proof}[\dem] First, note that $\ds R^\e=R_{8\pi m} + {8\pi m\e^2C(W)\over
\big(1+\sqrt{1-\e^2C(W)}\big)^2}\lf({ke^W\over
\int_Tke^W}-{k^2e^{2W}\over \int_Tk^2e^{2W}}\rg)$ in view of
$N=2m$. As in (\ref{ikeW}) we have that
\begin{equation*}
\begin{split}
\int_T k^2 e^{2W}\,dx&=\sum_{j=1}^m{1\over
64\de_j^4}\int_{B_{r_0}(\xi_j)}\rho_j^2(x)e^{2U_j}(1+O(\de^2|\log\de|))\,dx+O(1)\\
&=\sum_{j=1}^m{1\over 64\de_j^4}\lf({64\pi \rho_j^2(\xi_j)\over
3\de_j^2}+O(|\log\de|)\rg)+O(1)={\pi\over3\de^6}\sum_{j=1}^m
{1\over \rho_j(\xi_j)}\lf(1+O(\de^2|\log\de|)\rg).
\end{split}
\end{equation*}
Hence, in $T \setminus \cup_{j=1}^m B_{r_0}(\xi_j)$ there holds
${k^2e^{2W}\over \int_Tk^2e^{2W}}=O(\de^6)$  in view of
$W(x)=O(1)$, and in $B_{r_0}(\xi_j)$, $j\in\{1,\dots,m\}$, there
holds
\begin{eqnarray*}
{k^2e^{2W}\over \int_T k^2e^{2W}}
&=&\frac{3\de^2[ \rho_j^2(x)+O(\de^2|\log \de|)]}
{64\pi\rho_j^2(\xi_j)\sum_{l=1}^m[\rho_l(\xi_l)]^{-1}(1 +
O(\de^2|\log\de|))}e^{2U_j}=O(\de^2 e^{U_j}),
\end{eqnarray*}
which summarize as follows: $ {k^2e^{2W}\over
\int_Tk^2e^{2W}}=O\Big(\de^2 \sum_{j=1}^m \chi_j e^{U_j}+\de^6 \chi_{T
\setminus \cup_{j=1}^m B_{r_0}(\xi_j)}\Big).$  On the other hand, from \eqref{ikeW} we get that
\begin{equation}\label{cw}
C(W)=16\pi N{{\pi\over 3\de^6}\sum_{j=1}^m{1\over
\rho_j(\xi_j)}(1+O(\delta^2|\log\de|))\over [{\pi m\over
\de^2}+O(|\log\de|)]^2}={32\over 3m\de^2}\sum_{j=1}^m{1\over
\rho_j(\xi_j)}\lf(1+O(\de^2|\log\de|)\rg), \end{equation} which
implies by (\ref{repcs1}) that for $\e$ and $\de$ sufficiently small $W\in\ml{A}_\e$
and
$${8\pi m\e^2C(W)\over
\big(1+\sqrt{1-\e^2C(W)}\big)^2}=2\pi
m\e^2C(W)+O([\e^2C(W)]^2)=O\Big({\e^2\over\de^2}\Big).$$
Therefore, by using \eqref{important} and the estimate on ${k^2e^{2W}\over \int_Tk^2e^{2W}}$ we find the following
estimate
$$\ds R^\e-R_{8\pi
m}=O\bigg({\e^2\over\de^2}\Big[\sum_{j=1}^m\chi_je^{U_j}+\de^2\Big]\bigg),$$
and then $\|R^\e-R_{8\pi m}\|_*=O(\e^2\de^{-2})$. Thus,
in view of \eqref{R8pim} and \eqref{repcs1} the conclusion
follows.
\end{proof}

\medskip \noindent Now, we are going to establish the expansion of $I_\e(W)$.
\begin{theo} \label{energyexpansion}
Assume \eqref{repla0} and \eqref{repcs1}. The following expansion
does hold
\begin{equation} \label{IUt}
I_\e (W)=-16\pi m +8\pi m\log (8\e^2)+{|T|\over 2\e^2} -32\pi^2
\varphi_m(\xi)+A(\xi)\de^2\log\de-B(\xi)\delta^2+\ti
B(\xi){\e^2\over \de^2}+o(\de^2)
\end{equation}
in $C^2(\mathbb{R})$ and $C^1(\Xi)$ as $\de\to 0^+$, where
$\varphi_m(\xi)$, $A(\xi)$ and $B(\xi)$ are given by \eqref{fim},
\eqref{v} and \eqref{B}, respectively, and
\begin{equation}\label{Bti}
\ti B(\xi)={32\pi \over 3}\sum_{j=1}^m{1\over\rho_j(\xi_j)}.
\end{equation}
\end{theo}
\begin{proof}[\dem] By \eqref{cw} we have that
$$\ds\big[1+\sqrt{1-\e^2C(W)}\big]^{-1}={1\over 2}+{\e^2\over
8}C(W)+O\Big(\frac{\e^4}{\delta^4}\Big)\,,\quad \log\Big(1+\sqrt{1-\e^2C(W)}\Big)=\log 2-{\e^2\over 4}C(W)+O\Big(\frac{\e^4}{\delta^4})\Big).$$
Hence, by using \eqref{iuc} we find that
$$I_\e(W)=J_{8\pi m}(W)+\pi m\e^2C(W)+8\pi m\log(8\pi
m\e^2)+{|T|\over 2\e^2}-8\pi m+O(\e^4\de^{-4}).$$
Thus, the expansion \eqref{IUt} follows by \eqref{energyMF},
\eqref{repcs1} and $C(W)=[\pi m\de^2]^{-1}\ti
B(\xi)[1+O(\de^2|\log\de|)]$ in view of (\ref{cw}). Finally, the expansions for the
derivatives follow similarly as in the proof of Theorem
\ref{expansionenergy}, in view of
\begin{equation*}
\begin{split}
\fr_\beta[I_\e(W)]&=\fr_\beta[J_{8\pi m}(W)]+\frac{4\pi
m\e^2\fr_\beta[C(W)]}{(1+\sqrt{1-\e^2C(W)})^2}
\\
&=\fr_\beta[J_{8\pi m}(W)]+\pi
m\e^2\fr_\beta[C(W)]+O(\e^4 C(W)|\fr_\beta[C(W)]|)
\end{split}
\end{equation*}
for either $\beta=(\xi_j)_i$ or $\beta=\de$, and
\begin{equation*}
\fr_{\de \de}[I_\e(W)]=\fr_{\de \de}[J_{8\pi m}(W)]+\pi
m\e^2\fr_{\de \de}[C(W)]+O(\e^4C(W)|\fr_{\de \de}[C(W)]| +\e^4 |\fr_\de[C(W)]|^2),
\end{equation*}
by using
\eqref{derivativexiMF},
\eqref{derivdelta}, \eqref{derivdelta2} and the expansions for the derivatives of $C(W)$ in the line of \eqref{cw}.
\end{proof}

\medskip \noindent Since $L^\e$ and $N^\e$ are small perturbations of $L_{8\pi m}$
and $N_{8\pi m}$ in view of
$\|\Lambda^\epsilon(\phi)\|_*=O\Big(\ds{\e^2\over\de^2}\|\phi\|_\infty\Big)$
and $N^\e(\phi)=N_{8\pi
m}(\phi)+O\Big(\ds{\e^2\over\de^2}\|\phi\|_\infty^2\Big)$,
as for Proposition \ref{lpnlabis}, in view of \eqref{ree} it
follows
\begin{prop}\label{lpnlacs}
There exists $\delta_0>0$ so that for all $0<\delta\leq \delta_0$,
$\xi \in \Xi$ problem
\begin{equation*}
\left\{ \begin{array}{ll} L^\e(\phi)= -[R^\e+N^\e(\phi)]
+c_0\Delta PZ +\displaystyle \sum_{i=1}^{2}\sum_{j=1}^m c_{ij}
\Delta
PZ_{ij}& \text{in } T\\
\int_T \phi \Delta PZ=\int_T \phi \Delta PZ_{ij}= 0 &\forall\:
i=1,2,\, j=1,\dots,m
\end{array} \right.
\end{equation*}
admits a unique solution $\phi(\de,\xi) \in \bar H \cap
W^{2,2}(T)$ and $c_0(\de,\xi),\,c_{ij}(\de,\xi) \in \R$, $i=1,2$
and $j=1,\dots,m$, where $\de_j>0$ are as in \eqref{repla0} and
$N^\e$, $R^\e$ are given by \eqref{nltcs}, \eqref{Rcs}. Moreover,
the map $(\delta,\xi)\mapsto
(\phi(\delta,\xi),c_0(\de,\xi),c_{ij}(\de,\xi))$ is
twice-differentiable in $\de$ and one-differentiable in $\xi$ with
$$\|\phi\|_\infty+{\de\over |\log\de|}\lf[\|\fr_\de \phi \|_\infty+\sum_{i,j} \|\fr_{(\xi_j)_i} \phi \|_\infty
+{\de\|\fr_{\de\de}\phi\|_\infty\over |\log\de|}\rg] \le
C|\log\de|\left( \delta |\nabla \varphi_m(\xi)|_g+
\delta^{2-\sigma}\right).$$
\end{prop}
\begin{obs}
Notice that if $\|\phi\|_\infty\le \nu\de|\log\de|$ then
$W+\phi\in\ml{A}_\e$ for $\de$ and $\e$ small enough.
\end{obs}
\noindent The function $\phi(\de,\xi)$ will be a solution
 to \eqref{ephit}, namely, $W+\phi(\de,\xi)$ will be a true solution of \eqref{CSMF} if
$\delta$ and $\xi$ are such that
$c_0(\de,\xi)=c_{ij}(\delta,\xi)=0$ for all $i=1,2,$ and
$j=1,\dots,m$. Similarly to Lemma \ref{cpfc0bis}, this problem is
equivalent to finding critical points of the reduced energy
$\ml{E}^\e(\delta, \xi)=I_\e\big(W+\phi(\de,\xi)\big)$.

\begin{theo} \label{fullexpansionenergycs}
Assume \eqref{repla0} and \eqref{repcs1}. The following expansion
does hold
$$\ml{E}^\e(\de,\xi)=-16\pi m +8\pi m\log
(8\e^2)+{|T|\over 2\e^2} -32\pi^2
\varphi_m(\xi)+A(\xi)\de^2\log\de-B(\xi)\delta^2+\ti
B(\xi){\e^2\over \de^2}+o(\de^2)+r^\e(\de,\xi)$$ in
$C^2(\mathbb{R})$ and $C^1(\Xi)$ as $\de\to 0^+$, where
$\varphi_m(\xi)$, $A(\xi)$, $B(\xi)$ and $\ti B(\xi)$ are given by
\eqref{fim}, \eqref{v}, \eqref{B} and \eqref{Bti}, respectively.
The term $r^\e(\de,\xi)$ satisfies for some $C>0$ independent of
$(\de,\xi)$
\begin{eqnarray*}
|r^\e(\de,\xi)|+\frac{\de}{|\log \de|} |\nabla r^\e(\de,\xi)|
+\frac{\de^2}{|\log \de|^2}|\partial_{\de \de} r^\e(\de,\xi)| \leq
C \delta^2 |\log \delta | |\nabla \varphi_m (\xi)|_g^2.
\end{eqnarray*}
\end{theo}
\begin{proof}[{\bf Proof (of Theorem \ref{main3}):}]
Similarly to Theorem \ref{main2}, to find a critical point of
$\ml{E}^\e(\de,\xi)$ the key step is to get the existence of a function
$\de=\de(\e,\xi)=\sqrt{\e}\mu(\e,\xi)$ such that
$\fr_\de\ml{E}^\e(\de(\e,\xi),\xi)=0$ in a small neighborhood of
the critical set $\DD$. Even if $A(\xi) \geq 0$ for all $\xi \in (T\setminus \{p_1,\dots,p_l\})^m \setminus \Delta$, this is possible in view of $\ti B(\xi)>0$ and ``the correct sign'' $B(\xi)<0$ in
$\DD$. The argument is based on the same one explained in Remark \ref{minmax}-(ii) and uses the crucial smallness property of $A(\xi)$ near $\DD$: $A(\xi)=O(|\nabla \varphi_m(\xi)|_g^2)$.
\end{proof}


\section{Appendix A}
\noindent We consider the operator
$$L_{8\pi m}(\phi) = \Delta_g \phi +  {8\pi m ke^{W}\over\int_S
ke^{W}dv_g}\lf(\phi - {\int_{S} ke^{W}\phi dv_g \over\int_S
ke^{W}dv_g} \rg),$$
for which we first address a-priori estimates when all the $c_{ij}$'s vanish:
\begin{prop} \label{p1}
There exists $\delta_0>0$ and $C>0$ so that, for all $0<\delta\leq
\delta_0$, $h\in C(S)$ with $\int_Sh dv_g=0$, $\xi \in \Xi$ and
$\phi \in H_0^1(S) \cap W^{2,2}(S)$ a solution of \eqref{plco}
with $L=L_{8\pi m}$ and $c_0=c_{ij}=0$, $i=1,2$ and $j=1,\dots,m$,
one has
\begin{equation}\label{estmfe}
\|\phi \|_\infty \le C | \log \de | \|h\|_*.
\end{equation}
\end{prop}
\begin{proof}[\dem] By contradiction, assume the existence of sequences $\de \to 0$, points $\xi \in \Xi$ with $\xi \to \xi^*$, functions $h$ with $|\log \de| \|h\|_*=o(1)$ and solutions $\phi$ with $\|\phi\|_\infty=1$. Recall that $\de_j^2=\de^2 \rho_j(\xi_j)$. Setting $\mathcal{K}=\frac{8\pi m ke^{W}}{\int_S ke^{W}dv_g}$ and $\psi=\phi-\frac{\int_S ke^{W}\phi dv_g}{\int_S
ke^{W}dv_g}$, we have that $\Delta_g \psi+\mathcal{K} \psi=h$ in $S$ and $\psi$ does satisfy the same orthogonality conditions as $\phi$.

\medskip \noindent Since $\|\psi_n\|_\infty\le 2\|\phi_n\|_\infty \le 2$ and $\Delta_g\psi =o(1)$ in $C_{\hbox{loc}}(S \setminus \{\xi_1^*,\dots, \xi_m^*\})$, we can assume that $\psi \to \psi_\infty$ in $C^1_{\hbox{loc}}(S \sm\{\xi_1^*,\dots,\xi_m^*\})$. Since $\psi_\infty$ is bounded, it extends to an harmonic function in $S$, and then $\psi_\infty=c_0:= -\lim \frac{\int_S k e^W \phi dv_g}{\int_S k e^W dv_g}$ in view of ${1\over |S|}\int_S \psi dv_g=-\frac{\int_S k e^W \phi dv_g}{\int_S k e^W dv_g}$.

\medskip \noindent The function $\Psi_j =\psi(y_{\xi_j}^{-1}(\delta_j y))$ satisfies $\lap \Psi_j + \mathcal{\tilde K}_j \Psi_j = \tilde h_j$ in $B_{2r_0 \over \de_j}(0)$, where $\mathcal{\tilde K}_j=\de_j^2 \mathcal{K}(y_{\xi_j}^{-1}(\de_j y))$ and $\tilde
h_j=\de_j^2 h (y_{\xi_j}^{-1}(\de_j y))$. Since $|\ti h_j| \le C\|h \|_*$ and $\mathcal{\tilde K}_j={8\over (1+|y|^2)^2}(1+O(\delta^2|\log \delta|))$ uniformly in $B_{\frac{2r_0}{\delta}}(0)$ in view of Lemma \ref{ewfxi} and (\ref{intkeW}), up to a sub-sequence, by elliptic estimates $\Psi_j \to \Psi_{j,\infty}$ in $C^1_{\hbox{loc}}(\mathbb{R}^2)$, where $\Psi_{j,\infty}$ is a bounded solution of $\Delta \Psi_{j,\infty} + {8\over(1+|y|^2)^2}\Psi_{j,\infty}= 0$ of the form $\Psi_{j,\infty}=\displaystyle \sum_{i=0}^2 a_{ij}Y_i$ (see for example \cite{bp}). Since $-\Delta_g PZ_{ij} =\chi_j e^{-\varphi_j} e^{U_j} Z_{ij}-\frac{1}{|S|}\int_S
\chi_j  e^{-\varphi_j} e^{U_j} Z_{ij} dv_g$ in view of (\ref{ePZ}) and $\Delta_g=e^{-\varphi_j} \Delta$ in $B_{2r_0}(\xi_j)$ through $y_{\xi_j}$, we have that
$$0=-\int_S \psi \Delta_g PZ_{ij}=32 \int_{\mathbb{R}^2} \Psi_j \frac{y_i}{(1+|y|^2)^3} dy
-\frac{32}{|S|} \int_{\mathbb{R}^2} \frac{y_i}{(1+|y|^2)^3} dy \int_S \psi_n +O(\de^3).$$
Since then $\int_{\mathbb{R}^2} \Psi_j \frac{y_i}{(1+|y|^2)^3} dy=0$, we deduce that $a_{1j}=a_{2j}=0$. By the other orthogonality condition $\int_S \psi \Delta_g PZ=0$ similarly we deduce that
$$0=-\sum_{j=1}^m\int_S \psi \Delta_g PZ_{0j}dv_g=16 \int_{\mathbb{R}^2} \Psi_j \frac{1-|y|^2}{(1+|y|^2)^3} dy
-\frac{16}{|S|} \int_{\mathbb{R}^2} \frac{1-|y|^2}{(1+|y|^2)^3} dy \int_S \psi_n
+O(\de^2),$$
which implies $\displaystyle \sum_{j=1}^m a_{0j}=0$ in view of $\int_{\mathbb{R}^2} \frac{1-|y|^2}{(1+|y|^2)^3} dy=0$. By dominated convergence we have that
\begin{eqnarray*}
&&\int_S G(y,\xi_j) \mathcal{K} \psi dv_g=  -{1\over 2\pi } \log \de \int_{B_{r_0}(\xi_j)}  \mathcal{K} \psi dv_g+\int_{\mathbb{R}^2} \Big[-{1\over 2\pi }\log |y|+H(\xi_j,\xi_j) \Big] \frac{8}{(1+|y|^2)^2} \Psi_{j,\infty} dy\\
&&+ \sum_{i\not=j} G(\xi_i,\xi_j)  \int_{\mathbb{R}^2} \frac{8}{(1+|y|^2)^2} \Psi_{j,\infty} dy+o(1)=
-{1\over 2\pi } \log \de \int_{B_{r_0}(\xi_j)}  \mathcal{K} \psi dv_g+4 a_{0j}+o(1)
\end{eqnarray*}
in view of $\int_{\mathbb{R}^2} \log |y| \frac{1-|y|^2}{(1+|y|^2)^3} dy=-\frac{\pi}{2}$. In view of $\int_S \ml{K} \psi=0$ and
$$\bigg|\int_S G(y,\xi_j) h dv_g \bigg| \le C |\log\de | \int_S |h| dv_g+\frac{\|h\|_*}{\delta^2}\bigg|\int_{B_\de(\xi_j)} G(y,\xi_j) dv_g\bigg|\leq C' |\log \delta|\|h\|_*=o(1),$$
by the Green's representation formula
$$\sum_{j=1}^m  \Psi_j(0)= \sum_{j=1}^m \psi (\xi_j)={m \over |S|}\int_S \psi dv_g + \sum_{j=1}^m \int_S G(y,\xi_j) [ \mathcal{K} \psi- h ] dv_g=m c_0+4 \sum_{j=1}^m a_{0j}+o(1)$$
which gives $\displaystyle \sum_{j=1}^m a_{0j}= m c_0+4 \displaystyle \sum_{j=1}^m a_{0j}$ as $n \to +\infty$. Since $\displaystyle \sum_{j=1}^m a_{0j}=0$, we get that $c_0=0$.\\
Following \cite{EGP}, let $PZ_j \in H_0^1(S)$ be s.t. $\Delta_g P Z_j =\chi_j \Delta_g Z_j -\frac{1}{|S|}\int_S
\chi_j \Delta_g Z_j dv_g$ in $S$, where
$$Z_j(x)=\beta_j\Big(\frac{y_{\xi_j}(x)}{\delta_j}\Big)\,,\qquad \beta_j(y)={4\over 3}[2\log \delta_j+\log (1 + |y|^2 )]\frac{1
- |y|^2}{1 + |y|^2} + {8\over 3} \frac{1}{1+ |y|^2}$$
satisfies $e^{\varphi_j}\Delta_g Z_j+e^{U_j} Z_j=e^{U_j}Z_{0j}$ in $B_{2r_0}(\xi_j)$. Since it is easily seen that $ PZ_j= \chi_j Z_j +{16 \pi \over 3}H(\cdot, \xi_j) +O(\de^2 |\log \de |^2)$ uniformly in $S$, we test the equation of $\psi$ against $PZ_j$ to get:
\begin{eqnarray*}
&&\int_S \psi \bigg[\chi_j \Delta_g Z_j -\frac{1}{|S|}\int_S
\chi_j \Delta_g Z_j dv_g \bigg]dv_g + \int_S \mathcal{K} \psi P
Z_j dv_g= \int_S \chi_j \psi [e^{\varphi_j} \Delta_g Z_j
+ \mathcal{K} Z_j]dv_g+o(1)\\
&&=\int_S \chi_j \psi e^{U_j}Z_{0j} dv_g+o(1)=16 \int_{\mathbb{R}^2} \Psi_j \frac{1-|y|^2}{(1+|y|^2)^3}dy+o(1)
=\int_S h PZ_j=o(1)
\end{eqnarray*}
in view of $\int_S \mathcal{K} \psi dv_g=0$, $\int_S \psi dv_g=o(1)$, $\int_S \chi_j \Delta_g Z_j dv_g=O(1)$, $\int_S \chi_j \psi [\mathcal{K}-e^{U_j}]Z_j dv_g=O(\delta^2 |\log \delta|^2)$ and $\int_S h PZ_j=O(|\log \delta|\|h\|_*)=o(1)$. Since $\int_{\mathbb{R}^2} \Psi_j \frac{1-|y|^2}{(1+|y|^2)^3}dy=0$ we have that $a_{0j}=0$. So far, we have shown that $\psi \to 0$ in $C_{\hbox{loc}}(S\setminus \{\xi_1^*,\dots,\xi_m^*\})$ and uniformly in $\cup_{j=1}^m B_{R \delta_j}(\xi_j)$, for all $R>0$.

\medskip \noindent Setting $\hat \psi_j(y)=\psi (y_{\xi_j}^{-1}(y))$, $\mathcal{\hat
K}_j(y)=\mathcal{K} (y_{\xi_j}^{-1}(y))$ and $\hat h_j(y)=h(y_{\xi_j}^{-1}(y))$ for $y \in B_{2r_0}(0)$, we have that $
e^{\hat \varphi_j} \Delta \hat \psi_j + \mathcal{\hat K}_j \hat \psi_j=\hat h_j$. By now it is rather standard to show that the
operator $\hat L_j=e^{\varphi_j} \Delta + \mathcal{\hat K}_j$ satisfies the maximum principle in $
B_r(0) \sm B_{R\delta_j}(0)$ for $R$ large and $r>0$ small enough, see for example \cite{DeKM}. As a consequence, we get that $\psi \to 0$ in $L^\infty(S)$. Since $\frac{\int_S ke^{W}\phi dv_g}{\int_S ke^{W}dv_g}\to c$ along a sub-sequence, $\|\psi \|_\infty \to 0$ implies $\phi \to c$ in $L^\infty(S)$ with $c=0$ in view of $\int_S \phi=0$, in contradiction with $\|\phi\|_\infty=1$. This completes the proof.
\end{proof}
\noindent We are now ready for
\begin{proof}[{\bf Proof (of Proposition \ref{p2}):}] Since $\|\lap_g PZ_{ij}\|_*\le C$ for all $i=0,1,2$, $j=1,\dots,m$, and
$$\bigg\|(\lambda-8\pi m) { ke^{W}\over\int_S
ke^{W}dv_g}\lf(\phi - {\int_{S} ke^{W}\phi dv_g \over\int_S
ke^{W}dv_g} \rg)\bigg\|_*=O(|\lambda-8\pi m| \|\phi\|_\infty),$$
by Proposition \ref{p1}  for $\lambda$ close to $8\pi m$ any
solution of \equ{plco} satisfies
$$\|\phi\|_\infty\le
C |\log \de| \lf[\|h\|_*+|c_0|+\sum_{i=1}^2\sum_{j=1}^m |c_{ij}|\rg].$$
To estimate the values of the $c_{ij}$'s, test equation \equ{plco} against $PZ_{ij}$, $i=1,2$ and $j=1,\dots,m$:
$$\int_S \phi L(PZ_{ij})dv_g =\int_S h PZ_{ij}dv_g+c_{0}\sum_{l=0}^m \int_S  \lap_g
PZ_{0l} PZ_{ij}dv_g + \sum_{k=1}^2\sum_{l=1}^m c_{kl} \int_S \lap_g
PZ_{kl} PZ_{ij}dv_g.$$
Since for $j=1,\dots,m$ we have the following estimates in $C(S)$
\begin{equation}\label{pzij}
PZ_{ij}=\chi_jZ_{ij}+O(\de)\,,\:\:\:i=1,2\,, \qquad  PZ_{0j}=\chi_j(Z_{0j}+2)+O(\de^2|\log\de|),
\end{equation}
it readily follows that $\int_S \lap_g PZ_{kl} PZ_{ij}dv_g=-{32\pi\over 3}\delta_{ki}\delta_{lj}+O(\de)$, where the $\delta_{ij}$'s are the Kronecker's symbols. By Lemma \ref{ewfxi}, (\ref{repla1}), (\ref{intkeW}) and \eqref{pzij} we have that for $i=1,2$
$$L(PZ_{ij})=\chi_j \Delta_g Z_{ij}+e^{U_j} PZ_{ij}+O\Big(\delta^2 +\delta \sum_{k=1}^m e^{U_k}\Big)=
e^{U_j} [PZ_{ij}-e^{-\varphi_j }\chi_j
Z_{ij}]+O\Big(\delta^2+\delta \sum_{k=1}^m e^{U_k}\Big)$$ in view
of $\frac{\int_S k e^W PZ_{ij}dv_g}{\int_S k e^W dv_g}=O(\delta)$,
leading to $\|L(PZ_{ij})\|_*=O(\delta)$. Similarly, we have that
\begin{eqnarray*}
L(PZ)&=&\sum_{j=0}^m [\chi_j \Delta_g Z_{0j}+e^{U_j} PZ_{0j}- \frac{2}{m}\sum_{k=1}^m \chi_k  e^{U_k} ]+O(\delta^2)+O\bigg(\delta \sum_{k=1}^m e^{U_k}\bigg)\\
&=&\sum_{j=0}^m e^{U_j} [PZ_{0j}-\chi_j e^{-\varphi_j } Z_{0j}-
2\chi_j ]+O(\delta^2)+O\bigg(\delta \sum_{k=1}^m e^{U_k}\bigg)
\end{eqnarray*}
in view of $\frac{\int_S k e^W PZ_{0j}dv_g}{\int_S k e^W dv_g}=\frac{2}{m}+O(\delta^2|\log \delta|)$, leading to $\|L(PZ)\|_*=O(\delta)$.
Hence, we get that
$$|c_0|+\sum_{i=1}^2 \sum_{j=1}^m
|c_{ij}| \leq C \|h\|_*+\delta
O\Big(\|\phi\|_\infty+|c_0|+\sum_{i=1}^2 \sum_{j=1}^m
|c_{ij}|\Big) \leq C' \|h\|_*+\delta |\log
\delta|O\Big(|c_0|+\sum_{i=1}^2 \sum_{j=1}^m |c_{ij}|\Big),$$
yielding to the desired estimates $\|\phi\|_\infty=O(|\log \delta|
\|h\|_*)$ and $|c_0|+\displaystyle \sum_{i=1}^2 \sum_{j=1}^m
|c_{ij}|=O(\|h\|_*)$. To prove the solvability assertion, problem
\eqref{plco} is equivalent to finding $\phi\in H$ such that
$$ \int_S \langle\nabla \phi, \nabla \psi\rangle_g dv_g=\int_S \lf[{\la ke^W\over \int_S ke^W dv_g}\lf(\phi-{\int_S ke^W \phi dv_g \over \int_S ke^W dv_g}\rg)-h\rg]\psi dv_g\qquad \forall \, \psi \in H,$$
where $H=\{\phi \in H_0^1(S) \,:\: \int_S \phi \lap_g PZ_{ij} dv_g=\int_S \phi \lap_g PZ dv_g=0,\,
i=1,2,\, j=1,\dots,m \}$. With the aid of Riesz representation theorem, the Fredholm's alternative guarantees unique solvability for any $h$ provided that the homogeneous equation has only the trivial solution: for \eqref{plco} with $h=0$, the a-priori estimate (\ref{estmfe1}) gives that $\phi=0$.

\medskip \noindent So far, we have seen that, if $T(h)$ denotes the unique solution $\phi$ of \eqref{plco}, the operator $T$ is a continuous linear map from $\{h \in L^\infty(S):\, \int_S h dv_g =0 \}$, endowed with the $\|\cdot\|_*$-norm, into $\{\phi \in L^\infty(S):\, \int_S \phi dv_g =0 \}$, endowed with $\|\cdot\|_\infty$-norm. The argument below is heuristic but can be made completely rigourous. The operator $T$ and the coefficients $c_0,\,c_{ij}$ are differentiable w.r.t. $\xi_{l}$, $l=1,\dots,m$, or $\de$. Differentiating equation \eqref{plco}, we formally get that $X=\fr_\beta \phi$, where $\beta=\xi_{l}$ with $l=1,\dots,m$ or $\beta=\de$, satisfies $L(X)=\ti h(\phi)+d_0\lap_g PZ+\sum_{i,j} d_{ij}\lap_g PZ_{ij}$, where
\begin{eqnarray*}
\ti h(\phi)&=& -\fr_\beta\lf({\la ke^W\over \int_S k
e^W dv_g}\rg)\phi+\fr_\beta\lf[{\la ke^W\over \lf(\int_S k
e^W dv_g \rg)^2}\rg]\int_S ke^W\phi dv_g+{\la ke^W\over \lf(\int_S k
e^W dv_g \rg)^2}\int_S ke^W \fr_\beta
W \phi dv_g\\
&& + c_{0}\fr_\beta(\lap_g
PZ)+\sum_{i,j}c_{ij}\fr_\beta(\lap_g PZ_{ij})
\end{eqnarray*}
and $d_0=\fr_\beta c_0$, $d_{ij}=\fr_\beta c_{ij}$, and the
orthogonality conditions  become
$$\int_S X \lap_g PZ_{ij}dv_g =-\int_S \phi  \fr_\beta(\lap_g
PZ_{ij}) dv_g\,, \qquad \int_S X \lap_g PZ dv_g=-\int_S \phi \fr_\beta(\lap_g PZ)dv_g.$$
Find now coefficients $b_0$, $b_{ij}$ so that $Y=X+b_0 PZ+ \sum_{k,l} b_{kl}PZ_{kl}$ satisfies the orthogonality conditions $\int_S Y \lap_g PZ dv_g=\int_S Y \lap_g PZ_{ij}dv_g=0$. The coefficients $b_0,\,b_{ij}$ have to satisfy an almost diagonal system, and are then well-defined with $|b_0|+\displaystyle \sum_{ij} |b_{ij}| \le C {|\log \de| \over \de} \|h\|_*$ in view of
$\|\fr_\beta(\lap_g PZ_{ij})\|_*\le {C\over \de}$. Hence, the function $X$ can be uniquely expressed as $X=T(f)-b_0PZ-
\sum_{i,j} b_{ij}PZ_{ij}$, where $f=\ti h(\phi)+b_0L(PZ)+\sum_{i,j} b_{ij}L(PZ_{ij})$. Moreover, since $\|\fr_\beta
W \|_\infty+\|\fr_\beta \mathcal{K } \|_*\le{C\over\de}$, $\|\mathcal{K }  \|_*\le C$ and $\|\fr_\beta [{ \mathcal{K} \over \int_S ke^W dv_g } ] \|_*\le {C\over\de}(\int_S ke^W dv_g)^{-1}$ we find that
$$\|f\|_*\le \|\ti
h(\phi)\|_*+|b_0|\,\|L(PZ)\|_*+\sum_{i,j} |b_{ij}|\,\|L(PZ_{ij})\|_* \le C
{|\log \de|  \over \de} \|h\|_*,$$
and by (\ref{estmfe1}) we deduce that for any first derivative
$$\|\fr_\beta \phi\|_\infty \le C \Big[|\log \de|\|f\|_*+{\|\phi\|_\infty \over \de}\Big] \le C' {|\log \de|^2 \over\de} \|h\|_*.
$$
Differentiating once more in $\delta$ the equation satisfied by $\fr_\de \phi$ and arguing as above, we finally obtain that
$\|\fr_{\de\de} \phi \|_\infty \le C {|\log \de|^3 \over\de^2} \|h\|_*$, and the proof is complete.
\end{proof}


\section{Appendix B}
\noindent By Proposition \ref{p2} we now deduce the following.
\begin{proof}[{\bf Proof (of Proposition \ref{lpnlabis}):}] In terms of the operator $T$, problem \eqref{pnlabis} takes the form $\ml{A}(\phi)=\phi$, where $\ml{A}(\phi):=-T(R+N(\phi))$. Given $\nu>0$, let us consider the space
$$\ml{F}_\nu=\left \{\phi\in C(S)\,:\: \|\phi\|_\infty \le \nu
\bigg[\delta |\log \delta| \sum_{j=1}^m |\nabla \log(\rho_j \circ
y_{\xi_j}^{-1})(0)|+\de^{2-\sigma}|\log \de|^2 \bigg] \right \}.$$
Notice that if $\phi\in\ml{F}_\nu$ then $W+\phi\in\ml{A}_\e$ for
$\de$ and $\e$ small enough. Since in view of Lemma \ref{ewfxi}
and (\ref{intkeW}) we have
\begin{eqnarray*}
&& \lf\|  {\la ke^{W+\phi} \psi_1 \psi_2 \over\int_S ke^{W+\phi} dv_g } - {\la ke^{W+\phi} \psi_1 \int_S k e^{W+\phi} \psi_2 dv_g \over (\int_S ke^{W+\phi} dv_g)^2 }-{\la ke^{W+\phi} \psi_2 \int_S k e^{W+\phi} \psi_1 dv_g \over (\int_S ke^{W+\phi} dv_g)^2 }\right.\\
&& \left. -{\la ke^{W+\phi} \int_S ke^{W+\phi} \psi_1 \psi_2 dv_g \over (\int_S ke^{W+\phi} dv_g)^2 }+2{\la ke^{W+\phi} (\int_S ke^{W+\phi} \psi_1 dv_g)(\int_S ke^{W+\phi} \psi_2 dv_g) \over (\int_S ke^{W+\phi} dv_g)^3 }
\rg\|_* \le C \|\psi_1 \|_\infty \|\psi_2 \|_\infty ,
\end{eqnarray*}
 for any $\phi_1,\phi_2\in\ml{F}_\nu$ we obtain that $\|N(\phi_1)-N(\phi_2)\|_* \leq C (\| \phi_1\|_\infty+\|\phi_2)\|_\infty) \| \phi_1-\phi_2\|_\infty$ and then
$$\|\ml{A}(\phi_1)-\ml{A}(\phi_2)\|_\infty \leq C|\log\de| (\| \phi_1\|_\infty+\|\phi_2)\|_\infty) \| \phi_1-\phi_2\|_\infty \leq \frac{1}{2} \| \phi_1-\phi_2\|_\infty$$
for $\delta$ small in view of Proposition \ref{p2}. Moreover, we have that for any $\phi \in\ml{F}_\nu$
$$\|\ml{A}(\phi)\|_\infty \leq C |\log\de| (\| \phi\|^2_\infty+\|R\|_*)\leq
C |\log\de|\, \| \phi\|^2_\infty+ C_0 \left[\delta |\log \delta|
\sum_{j=1}^m |\nabla \log(\rho_j \circ
y_{\xi_j}^{-1})(0)|+\de^{2-\sigma}|\log \de|^2\right]$$ in view of
Lemma \ref{estrr0}. Then, for $\nu=2C_0$ and $\de$ small $\ml{A}$
is a contraction mapping of $\ml{F}_\nu$ into itself, and
therefore has a unique fixed point $\phi \in \ml{F}_\nu$.

\medskip \noindent By the Implicit Function Theorem it follows that the map $(\de,\xi) \to (\phi(\delta,\xi), c_0(\delta,\xi), c_{ij}(\de,\xi))$ is (at least) twice-differentiable in $\de$ and one differentiable in $\xi$. Differentiating  $\phi=-T(R+N(\phi))$ w.r.t. $\beta=\xi_l$, $l=1,\dots,m$, or $\beta=\de$, we get that $\fr_\beta\phi=-\fr_\beta T(R+N(\phi))-
T(\fr_\beta R+\fr_\beta N(\phi))$. By  Lemma \ref{estrr0} and (\ref{estd}) we have that
$$ \|\fr_\beta T(R+N(\phi))\|_\infty \le C {|\log \de|^2 \over\de} (\|R\|_*+\|N(\phi)\|_*)=O\bigg( |\log \de|^2  \sum_{j=1}^m |\nabla\log (\rho_j \circ y_{\xi_j}^{-1})(0)|+\de^{1-\sigma}|\log \de|^3 \bigg),$$
and, in view of $\|\partial_\beta W\|_\infty \leq \frac{C}{\de}$, we can estimate
\begin{eqnarray*}
&& \fr_\beta N(\phi)= N(\phi)\fr_\beta W+\la\lf({ke^{W+\phi}\over
\int_S k e^{W+\phi}dv_g}-{ke^{W}\over \int_S k e^{W}dv_g }\rg)\fr_\beta
\phi  \\
&&-\la \left({ke^{W+\phi} \int_S k e^{W+\phi} \fr_\beta
W  dv_g \over\lf(\int_S k e^{W+\phi}dv_g \rg)^2}-{ke^{W} \int_S k
e^{W} \fr_\beta W dv_g \over\lf(\int_S k
e^{W}dv_g \rg)^2}-{ke^{W}\phi \int_S k e^{W}\fr_\beta W dv_g \over\lf(\int_S k
e^{W}dv_g \rg)^2}- {ke^{W} \int_S k e^{W}\fr_\beta W \phi dv_g \over\lf(\int_S
k e^{W}dv_g \rg)^2}\right.\\
&& \left.+2{ke^{W} \lf(\int_S k e^{W}\fr_\beta
W dv_g \rg)\lf(\int_S k e^{W} \phi dv_g \rg)\over\lf(\int_S
k e^{W}dv_g \rg)^3}\right) -\la \left({ke^{W+\phi} \int_S k e^{W+\phi} \fr_\beta \phi dv_g \over\lf(\int_S k e^{W+\phi}dv_g \rg)^2}-{ke^{W} \int_S k
e^{W} \fr_\beta\phi dv_g \over\lf(\int_S k
e^{W}dv_g \rg)^2}\right)
\end{eqnarray*}
as follows
\begin{eqnarray}
\|\fr_\beta N(\phi) \|_* &\le& C\lf[\|\fr_\beta
W\|_\infty\|\phi\|_\infty^2+\|\phi\|_\infty\|\fr_\beta\phi\|_\infty\rg] \nonumber \\
&=& O\bigg(\de |\log \delta|^2 \sum_{j=1}^m |\nabla \log(\rho_j
\circ y_{\xi_j}^{-1})(0)|^2+\de^{3-2\sigma}|\log \de|^4\bigg)
+o\lf(\frac{\|\fr_\beta\phi\|_\infty}{|\log \de|}\rg).
\label{derivN}
\end{eqnarray}
Since $\int_S \chi_j e^{-\varphi_j} e^{U_j}dv_g=\int_{\mathbb{R}^2} \chi(|y|)\frac{8\de^2 \rho_j(\xi_j)}{(\de^2 \rho_j(\xi_j)+|y|^2)^2}dy$, we have that
$$\partial_{\xi_l}\bigg(\int_S \chi_j e^{-\varphi_j} e^{U_j}dv_g\bigg)= 8 \partial_{\xi_l} \log \rho_j(\xi_j) \int_{\mathbb{R}^2}  \frac{1-|y|^2}{(1+|y|^2)^3}+O(\de^2)=O(\de^2)$$
and
$$\partial_\de (\int_S \chi_j e^{-\varphi_j} e^{U_j}dv_g)=\int_{\mathbb{R}^2} \chi(|y|)\frac{16 \de \rho_j(\xi_j) (|y|^2-\de^2 \rho_j(\xi_j))}{(\de^2 \rho_j(\xi_j)+|y|^2)^3} dy=
\frac{16}{\de} \int_{\mathbb{R}^2} \frac{|y|^2-1}{(1+|y|^2)^3}dy+O(\de)=O(\de).$$
Since $\varphi_j(\xi_j)=0$ and $\nabla \varphi_j(\xi_j)=0$, we have that $e^{-\varphi_j}=1+O(|y_{\xi_j}(x)|^2)$ and $\partial_\beta(\chi_j e^{-\varphi_j}(x))=O(|y_{\xi_j}(x)|)$, and then
$$\lab \fr_\beta W=-\sum_{j=1}^m \chi_j  e^{U_j}\fr_\beta U_j+O(\de^{1-\sigma})$$
in view of $|\partial_\beta U_j|=O(\frac{1}{\de})$, where the big
$O$ is estimated in $\|\cdot \|_*$-norm. Since  in
$B_{r_0}(\xi_j)$
\begin{eqnarray*}
\partial_{\xi_l} W= \partial_{\xi_l} U_j+O(\de^2 |\log \de|+|y_{\xi_j}(x)|+ |\nabla\log(\rho_j \circ y_{\xi_j}^{-1})(0)|)\,,\qquad\partial_\de W=\partial_\de U_j-\frac{2}{\de}+O(\de |\log \de|),
\end{eqnarray*}
in the same line as Lemma \ref{ewfxi} and
$$\frac{\lambda k e^W}{\int_S ke^W dv_g}=\sum_{j=1}^m \chi_j e^{U_j} [1+O( |\nabla \log(\rho_j \circ y_{\xi_j}^{-1})(0)||y_{\xi_j}(x)|+\de^2 |\log \de|)]+O(\de^2)$$
in view of (\ref{important}), by (\ref{2term}) and (\ref{Merry}) we deduce for
$$\fr_\beta R=\lab \fr_\beta W+{\la ke^W\over\int_S
ke^W dv_g}\lf(\fr_\beta W-{\int_S ke^W \fr_\beta W\over\int_S
ke^W dv_g}\rg)$$
the estimate
$$\|\partial_\beta R\|_*=O\bigg(\sum_{j=1}^m |\nabla \log(\rho_j \circ y_{\xi_j}^{-1})(0)|+\de^{1-\sigma} |\log \delta|\bigg).$$
Combining all the estimates, we then get that
$$\|\fr_\beta\phi\|_\infty=O\bigg(|\log \delta|^2 \sum_{j=1}^m |\nabla \log(\rho_j \circ y_{\xi_j}^{-1})(0)|+\de^{1-\sigma}|\log \de|^3\bigg)
 +o\big(\|\fr_\beta\phi\|_\infty\big),$$
which in turn provides the validity of (\ref{cotadphi1bis}). We proceed in the same way to obtain the estimate (\ref{cotad2phi1bis}) on $\fr_{\de\de}\phi$, and the proof is complete.
\end{proof}
\noindent Lemma \ref{cpfc0bis} is rather standard and we will omit its proof. Since the problem has been reduced to find c.p.'s of the reduced energy $E_\lambda(\delta, \xi)= J_\lambda(W+\phi(\delta,\xi))$, where $J_\lambda$ is given by \eqref{energy}, the last key step is show that the main asymptotic term of $E_\lambda$ is given by $J_\lambda(W)$.
\begin{proof}[{\bf Proof (of Theorem \ref{fullexpansionenergy}):}] Write
\begin{eqnarray*}
J_\la(W+\phi)-J_\la(W)&=&DJ_\lambda(W)[\phi]+{ D^2
J_\lambda(W)[\phi,\phi]\over 2}+\int_0^1\!\!\! \int_0^1 [D^2
J_\lambda(W+ts \phi)-D^2J_\lambda(W)][\phi,\phi]\,t\, dsdt \\
&=&-{1\over 2}\int_S R\phi\,dv_g+{1\over 2}\int_S
N(\phi)\phi\,dv_g+\int_0^1\!\!\! \int_0^1 [D^2 J_\lambda(W+ts
\phi)-D^2J_\lambda(W)][\phi,\phi]\,t\, dsdt,
\end{eqnarray*}
since $D J_\lambda (W) (\phi)=-\int_S R \phi dv_g$, $D^2
J_\lambda(W)[\phi,\phi]= -\int_S L(\phi) \phi dv_g$ and
$$DJ_\lambda(W)[\phi]+ D^2 J_\lambda(W)[\phi,\phi]=\int_S N(\phi) \phi dv_g$$
in view of $\int_S \phi dv_g=0$ and (\ref{pnlabis}).
Since $\frac{1}{2}\int_S ke^W dv_g \leq \int_S k e^{W+ts \phi}dv_g \leq 2\int_S k e^W dv_g$ and $|e^{W+ts\phi}-e^W| \leq C e^W \|\phi\|_\infty$, it is straighforward to see that
\begin{eqnarray*}
&&\bigg|DJ_\lambda(W)[\phi]+ D^2 J_\lambda(W)[\phi,\phi]\bigg|+\bigg|\int_0^1 t dt \int_0^1 ds [D^2 J_\lambda(W+ts \phi)-D^2J_\lambda(W)][\phi,\phi] \bigg|\\
&&=O(\|N(\phi)\|_* \|\phi\|_\infty+\|\phi\|_\infty^3)=O(\|\phi\|_\infty^3),
\end{eqnarray*}
and then we deduce that
$$|J_\la(W+\phi)-J_\la(W)|=O(\|R\|_*\|\phi\|_\infty + \|\phi\|_\infty^3)
=O\left(\delta^2 |\log \delta |\, |\nabla \varphi_m(\xi)|^2+
\delta^{3-\sigma}|\log\delta|^2 \right)$$ in view of
(\ref{cotaphi1bis}) and
$4\pi\grad_{\xi_j}\varphi_m(\xi)=\grad\log(\rho_j\circ
y_{\xi_j}^{-1})(0)$. Differentiating w.r.t. $\be=\xi_{l}$,
$l=1,\dots,m$, or $\be=\de$ we get that
\begin{eqnarray*}
\partial_\beta[J_\la(W+\phi)-J_\la(W)]&=& -{1\over 2}\int_S[\fr_\beta R\,\phi + R\,\fr_\beta\phi]\,dv_g
+{1\over 2}\int_S \lf(\partial_\beta[N(\phi)] \phi +N(\phi) \partial_\beta \phi\rg)\, dv_g\\
&&
+\int_0^1 t dt \int_0^1 ds \partial_\beta \{[D^2 J_\lambda(W+ts \phi)-D^2J_\lambda(W)][\phi,\phi]\}.
\end{eqnarray*}
Since it is straightforward to see that
$$\bigg|\int_0^1 t dt \int_0^1 ds \partial_\beta \{[D^2 J_\lambda(W+ts \phi)-D^2J_\lambda(W)][\phi,\phi]\}\bigg|=
O( \|\phi\|_\infty^2  \|\partial_\beta \phi\|_\infty+\|\phi\|_\infty^3 \|\partial_\beta W\|_\infty),$$
by (\ref{derivN}) we deduce that
\begin{eqnarray*}
|\partial_\beta[J_\la(W+\phi)-J_\la(W)]|&=& O(\|\partial_\beta
R\|_* \|\phi\|_\infty + \|R\|_* \|\partial_\beta \phi\|_\infty+
\|\phi\|_\infty^2  \|\partial_\beta \phi\|_\infty+\|\phi\|_\infty^3 \|\partial_\beta W\|_\infty)\\
&=&O\lf(\big[\delta^2 |\log \delta|\, | \nabla
\varphi_m(\xi)|^2+\de^{3-\sigma}|\log \de|^2\big]{|\log\de|\over
\de}\rg)
\end{eqnarray*}
in view of (\ref{cotaphi1bis})-(\ref{cotadphi1bis}) and $\|\partial_\beta W\|_\infty=O(\frac{1}{\de})$. Arguing similarly for the second derivative in $\de$, we get that
\begin{eqnarray*}
\lf|\partial_{\de
\de}[J_\la(W+\phi)-J_\la(W)]\rg|=O\lf(\big[\delta^2 |\log
\delta|\, | \nabla \varphi_m(\xi)|^2+\de^{3-\sigma}|\log
\de|^2\big]{|\log\de|^2\over \de^2}\rg).
\end{eqnarray*}
Combining the previous estimates on the difference $J_\la(W+\phi)-J_\la(W)$ with the expansion of $J_\lambda(W)$ contained in Theorem \ref{expansionenergy}, we deduce the validity of the expansion (\ref{fullJUt}) with an error term which can be estimated (in $C^2(\mathbb{R})$ and $C^1(\Xi)$) like $o(\de^2)+r_\lambda(\de,\xi)$ as $\de \to 0$, where $r_\lambda(\de,\xi)$ does satisfy (\ref{rlambda}). \end{proof}



\begin{center}
{\bf Acknowledgements}
\end{center}
\noindent Part of this work was carried out while the second
author was visiting the Department of Mathematics, University of ``Roma Tre''. He would like to express his deep gratitude to Prof. Esposito for the many stimulating discussions about these topics
and the warm hospitality.



\end{document}